\documentclass[a4paper,11pt]{amsart}
\usepackage{graphicx,amssymb,amsmath,amsfonts,amsthm, xcolor}%amstext,
\usepackage{nicefrac}
\usepackage{geometry}   
\usepackage{anysize}
\usepackage[T1]{fontenc}
\usepackage{bbm}
\usepackage{ulem}
\usepackage{svg}
\marginsize{2cm}{2cm}{2cm}{2cm}
\usepackage{mathtools}
\usepackage{textcomp}
\usepackage{paralist}
\usepackage[linkcolor=black,citecolor=black]{hyperref}
\usepackage[english]{babel}

\newcommand{\m}{\mathbbm{m}}

\newcommand{\EE}{\mathbb{E}}
\newcommand{\FF}{\mathbb{F}}

\newcommand{\NN}{\mathbb{N}}

\newcommand{\PP}{\mathbb{P}}

\newcommand{\RR}{\mathbb{R}}

\newcommand{\mm}{\mathfrak{m}}

\newcommand{\aA}{\mathcal{A}}
\newcommand{\bB}{\mathcal{B}}

\newcommand{\eE}{\mathcal{E}}
\newcommand{\fF}{\mathcal{F}}

\newcommand{\iI}{\mathcal{I}}

\newcommand{\lL}{\mathcal{L}}
\newcommand{\mM}{\mathcal{M}}

\newcommand{\oO}{\mathcal{O}}
\newcommand{\pP}{\mathcal{P}}

\newcommand{\sS}{\mathcal{S}}

\newcommand{\uU}{\mathcal{U}}

\newcommand{\xX}{\mathcal{X}}

\newcommand{\al}{\alpha}

\newcommand{\e}{\varepsilon}
\newcommand{\la}{\lambda}

\newcommand{\om}{\omega}

\newcommand{\ra}{\rightarrow}

\newcommand{\ti}{\tilde}

\newcommand{\lgl}{\ensuremath{\langle}}
\newcommand{\rgl}{\ensuremath{\rangle}}

\newcommand{\ind}{\mathbf{1}}

\newcommand{\lqq}{\leqslant}
\newcommand{\gqq}{\geqslant}

\newtheorem{thm}{Theorem}

\newtheorem{cor}{Corollary}

\newtheorem{lem}{Lemma}

\newtheorem{rem}{Remark}
\newtheorem{exm}{Example}
\newtheorem{defn}{Definition}

\DeclareMathSymbol{\ophi}{\mathalpha}{letters}{"1E}

\renewcommand{\phi}{\varphi}

\newcommand{\cF}{\mathcal{F}}

\newcommand{\bI}{\mathbf{1}} %\mathbb{J}

\newfont{\cyrfnt}{wncyr10}
\def\J3{\cyrfnt{\rm \u{\cyrfnt I}}}
\def\j3{\cyrfnt{\rm \u{\cyrfnt i}}}

\usepackage[]{color}
\definecolor{DarkGreen}{rgb}{0.1,0.7,0.3}   %define a custom color

\definecolor{DarkGreen}{rgb}{0.1,0.7,0.3}   %define a custom color
\definecolor{DarkOrange}{HTML}{FF7F50}

\allowdisplaybreaks[4]

\begin{document}
\title[Almost sure martingale convergence and its error instance]{
On the tradeoff between almost sure error tolerance and mean deviation frequency 
in martingale convergence 
}
  
\author{Luisa Fernanda Estrada\\
\textnormal{The University of Warwick, Department of Computer Sciences, \\
CV4 7AL, Coventry, United Kingdom.}\\
\url{luisa-fernanda.estrada-plata@warwick.ac.uk}}
\author{\hfill\\\hfill\\
Michael A. H\"ogele\\
Universidad de los Andes, Facultad de Ciencias, Departamento de Matem\'aticas, \\
Cra 1 \# 18A - 12, 111711 Bogot\'a, Colombia\\
\url{ma.hoegele@uniandes.edu.co}}
\author{\hfill\\\hfill\\
Alexander Steinicke\\
Montanuniversit\"at Leoben, Lehrstuhl f\"ur Angewandte Mathematik, \\
Peter Tunner-Straße 25/I, A-8700 Leoben, Austria. \\
\url{alexander.steinicke@leoben.ac.at}}
\date{\today}

% \hfill\\[-1cm]
    
\begin{abstract}
In this article we quantify almost sure martingale convergence theorems in terms of the tradeoff between asymptotic almost sure rates of convergence (error tolerance) and the respective 
modulus of convergence. For this purpose we generalize {an} elementary quantitative version of the first Borel-Cantelli lemma on the statistics of the deviation frequencies (error incidence), which was recently established by the authors. First we study martingale convergence in $L^2$, and in the setting of the Azuma-Hoeffding inequality. In a second step we study the strong law of large numbers for martingale differences in two settings: uniformly bounded increments in $L^p$, $p\gqq 2$, using the respective Baum-Katz-Stoica theorems, and uniformly bounded exponential moments with the help of the martingale estimates by Lesigne and Voln\'y. We also present applications for the tradeoff for the multicolor generalized P\'olya urn process, the Generalized Chinese restaurant process, statistical M-estimators, as well as the a.s.~excursion frequencies of the Galton-Watson branching process. Finally, we relate the tradeoff concept to the convergence in the Ky Fan metric. 
\end{abstract}

\maketitle

\bigskip
\section{\textbf{Introduction}} 

\noindent The notion of almost sure (a.s.) convergence of a sequence of random variables $(X_n)_{n\in \NN}$ to a random variable $X$ as $n\ra\infty$, is certainly one of the most natural concepts in probability and statistics in the assessment of the evolution of observed data. This type of convergence is intuitive to grasp due to its similarity to the pointwise converence of deterministic functions. We highlight the following two particularities of a.s.~convergence: 
\begin{enumerate}
 \item {On a practical level, we are not aware of a satisfactory quantification in the literature, since the modulus of convergence $\m_\e$, that is the \textit{last index} $\m_\e\in \NN$, when a given error threshold $\e>0$ is broken in the sense of $|X_{\m_\e}-X|>\e$, is inherently random and seemingly not easily accessible.} 
 \item On a theoretical level, almost sure convergence does not define a proper topology on the space of random variables $L^0$, see \cite{Ord66}. 
\end{enumerate}
In this article we address problem (a) in Lemma~\ref{lem:Quant BC for e_n} by a general result on the tradeoff between a given sequence of error tolerances $\epsilon = (\e_n)_{n\in \NN}$ and the integrability {for each of the random numbers $\oO_\epsilon$ and $\m_\e$,} where $\oO_\epsilon$ counts in how many indices $n$ we have $|X_n-X|>\e_n$ and $\m_\epsilon$ is given above. Clearly, the error frequency $\oO_\epsilon$ is always a lower bound of the last error occurrence $\m_\e$, and therefore a weaker measure than the desired modulus of convergence $\m_\e$ in (a). 
Yet, it is a meaningful statistical measure for a.s.~convergence, since it will turn out that both satisfy the same upper bounds in Lemma~\ref{lem:BC1} and 
we can show that the worst case gap between the two numbers turns out to be asymptotically negligible in many situations. This result is then applied to several classical martingale convergence theorems and strong laws for martingale differences, and finally applied to more concrete applications in machine learning, classical statistics, and biology. 

In many situations, almost sure convergence is established by an application of the first Borel-Cantelli lemma \cite{Bi99, Bo1909, Ca1917,Ch12, CE51, Hi83, Sh99} to the sequence of the error events $A_n(\e) = \{|X_n - X|>\e\}$ for any $\e>0$, $n\in \NN$, {and} $n\gqq n_0$ for some fixed $n_0\in \NN$. For an overview of the literature we refer to the introduction of \cite{EstraHoeg22}. Classical examples of this proof technique are Etemadi's strong law of large numbers~\cite{Et81}, L\'evy's construction of Brownian motion, the Kolmogorov-Chentsov theorem, and the law of the iterated logarithm{. See} \cite{HS23} for more examples in the context of Brownian path property approximation. 
This particular notion of a.s.~convergence stemming from the first Borel-Cantelli lemma is well-established in the literature as \textit{complete convergence} \cite{HR47, MS18, Yu99}: A sequence of random variables $(X_n)_{n\geq 0}$ converges completely to a random variable $X$, if for all $\varepsilon>0$ we have $\sum_{n=0}^\infty \PP(A_n(\e))=\sum_{n=0}^\infty \PP(\{|X_n - X|>\e\})<\infty$. We generalize this notion in the spirit of \cite{EstraHoeg22} with the help of the following refined first Borel-Cantelli lemma: 
Recall that the classical first Borel-Cantelli lemma can be formulated as follows: On a given probability space $(\Omega, \aA, \PP)$, the summability of the sequence of the probabilities of the events $(A_n)_{n\gqq n_0}$ implies that the overlap statistic $\oO:= \sum_{n= n_0}^\infty \ind(A_n)$ is finite with probability $1$. The result $\oO <\infty$ a.s.~with its elegant one-line proof, however, is suboptimal 
since by monotone convergence we even know the average size of $\oO$ 
\begin{equation}\label{e:BC1}
\EE[\oO] = \sum_{n= n_0}^\infty \PP(A_n),  
\end{equation}
which is finite by hypothesis. Moreover, the law of the random variable $\oO$ has been known for a long time by the Schuette-Nesbitt formula \cite{G79}. Not surprisingly, the value $\PP(\oO = k)$ is given by means of an inclusion-exclusion principle 
as the sum of the probabilities of all the intersections of 
exactly $k$ events of the sequence $(A_n)_{n\gqq n_0}$. 
Unfortunately, the complete sequence of all such probabilities 
of event intersections is hardly ever available in applications 
(for the case of independent events we refer to \cite[Subsection 2.2, Theorem 3]{EstraHoeg22}). 
On the other hand, the (top level) null sequence $(\PP(A_n))_{n\gqq n_0}$ 
is often well-known and turns out to tend to $0$ faster than just strictly necessary to be summable. In many situations, for instance in the presence of a large deviations principle, it is of exponential order of decay. It is natural to translate this structural surplus into the finiteness of higher moments of $\oO$ and the tail asymptotics $\PP(\oO\gqq k)$ as $k\ra\infty$.  
In \cite[Theorem 1]{EstraHoeg22} it is shown for $n_0=1$ that for a sequence of positive, nondecreasing weights $(a_n)_{n\gqq n_0}$ certain nonlinear higher moments of $\oO$ (depending on the sequence $(a_n)_{n\gqq n_0}$) can be bounded by the weighted sum 
\begin{equation}\label{e:GewSumme}
C_a := \sum_{n = n_0}^\infty  a_n \sum_{m= n}^\infty \PP(A_{{m}}),
\end{equation}
whenever the preceding series converges. We show a slight generalization of this result, which turns out to be useful in many applications.  

We illustrate the novelty of our results by the following example. Think of a Cram\'ers type estimate 
\[
\PP(|\bar X_n - \EE[X_1]| >\e_n)\lqq 2 e^{-\frac{1}{2} n \e_n^2}, \qquad n\in \NN,  
\]
for the law of large numbers with i.i.d.~summands $X_i$ with some finite exponential moment. 
While \cite{EstraHoeg22} treats the case of constant $\e$ we observe the following. 
The essentially optimal rates $\e_n = \sqrt{\alpha \ln(n) /n}$, $\alpha >2$, yield for $\alpha$ close to $2$ barely summable probabilities. This in turn implies by \eqref{e:BC1} that $\EE[\oO] <\infty$ and therefore, by Markov's inequality, $\PP(\oO\gqq \ell) \lqq \EE[\oO]/\ell$. However, if we consider the slightly suboptimal rate $\tilde \e_n = n^{-\frac{1}{3}}> \e_n$ we obtain the by far better rate $\PP(|\bar X_n - \EE[X_1]| >\tilde \e_n)\lqq 2 e^{-\frac{1}{2} n^{\frac{1}{3}}}$. Further, we get $\EE[\exp(p \oO^\frac{1}{3})] < \infty$ for any $p\in (0,1)$ (see Example~\ref{ex:Weibull})
and hence the much faster observation $\PP(\oO\gqq \ell)\lqq \EE[\exp(p \oO^\frac{1}{3})] / e^{p \ell^\frac{1}{3}}$ which then can still be minimized over all $p\in (0,1)$. 
More useful still, our results including all upper bounds are valid not only for $\oO$, that is the \textit{number} of error event indices, but also for the \textit{last index} $\m$ (defined in \eqref{eq:m}) where an error event occurs. 
In a word, there is often a tradeoff in the sense that relaxing the optimal a.s.~rate of convergence to a slightly worse one, we often ``speed up'' its emergence substantially.

Our quantitative Borel-Cantelli result allows for the solution of problem (a) for the special sequence of events $(A_n(\e_n))_{n\in \NN}$ defined above. More precisely, we study the relation between 
a given positive null sequence $\epsilon := (\e_n)_{n\in \NN}$, called \textit{error tolerance}, 
and the higher order integrability of $\oO_{\epsilon} := \sum_{n=n_0}^\infty \ind(A_n(\e_n))$, called the \textit{error incidence} or \textit{deviation frequency}, \textit{overlap count} or \textit{failure count}, which generalizes formula \eqref{e:BC1}. 
That is to say, $\oO_{\epsilon, n_0} = |\{n\gqq n_0~:~|X_n-X| >\e_n\}|$ and $\m_{\epsilon, n_o} = \max\{n\gqq n_0~:~|X_n-X| >\e_n\}$. The quantification of the a.s.~convergence $X_n\to X$ relies in the finiteness of higher moments of $\mathcal{O}_{\epsilon,n_0}$ (``how many errors occur before dying out'') and $\m_{\epsilon,n_0}$ (``at which position happens the last error''). The type of moments that consider is specified in Lemma \ref{lem:BC1} in Section \ref{s:qBC1}, a key result for the rest of the article. It states the following. Given events $A_n = A_n(\e_n)$ and a chosen sequence $a$ such that $C_a$ in \eqref{e:GewSumme} is finite, then  
\begin{itemize}
\item for the a.s.~asymptotic upper error rate, we have
\begin{equation}\label{e:errortolerance}
\limsup_{n\ra\infty } |X_n-X|\cdot \e_n^{-1} \lqq 1\qquad \PP\mbox{-a.s.}
\end{equation}
\item Further, for the respective mean deviation frequency (MDF) quantification we have
\begin{equation}\label{e:MDFmoment}
\EE[\sS_{a, n_0}(\oO_{\epsilon})] \lqq \EE[\sS_{a, n_0}(\m_{\epsilon})]\lqq C_a,\quad \mbox{ where } \quad \sS_{a, n_0}(N) := \sum_{n= 0
}^{N-1} a_{n_0+n}, \quad N\in \NN,
\end{equation}
with the convention $\sS_{a, n_0}(0) = 0$.
\end{itemize} 
A choice for the sequence $(a_n)_{n\geq n_0}$ that will appear often is a power sequence $a_n=n^p$ for some $p>0$. Then $\sS_{a, n_0}(N)$ grows polynomially in $N$ with degree $p+1$. It will be used to estimate moments such as $\EE[\m_{\epsilon}^{p+1}]$. Another choice are exponential sequences $a_n=e^{\alpha n}$ for some $\alpha>0$. Then also $\sS_{a, n_0}(N)$ grows exponentially in $N$. We use it to bound exponential moments of $\m_{\epsilon}$.

Note further that \eqref{e:MDFmoment} implies that for any $k\gqq 1$ 
\[
\PP(\oO_{\epsilon} \gqq k) \lqq \PP(\m_{\epsilon} \gqq k) \lqq \inf_{a} C_a \cdot (\sS_{a, n_0}(k))^{-1}, 
\]
where the infimum is taken over some meaningful subset of positive sequences of weights $(a_n)_{n\in \NN}$ such that $C_a<\infty$. Particular cases of such quantifications can be found  in \cite{HS23} in the context of Brownian sample path approximations. 

This result has three main benefits: 
\begin{enumerate}
 \item The tradeoff relation between $\epsilon = (\e_n)_{n\in \NN}$ and $\PP(\m_{\epsilon} \gqq k)$ for $|X_n-X|\ra 0$ a.s.~is completely intuitive and analogous to the convergence in any metric space. It can be described informally as follows: 
The \underline{faster} $\e_n\searrow 0$, as $n\ra\infty$, the \underline{higher} the last index at which $|X_n-X|>\e_n$.
Consequently, we have larger values of $\m_{\epsilon}$ and less integrability and a \underline{slower} decay of $\PP(\m_{\epsilon} \gqq k)$ as $k\ra\infty$. Conversely, the slower $\e_n\searrow 0$, as $n\ra\infty$, the lower the number of deviations and the smaller $\m_{\epsilon}$.
The same mechanism is valid for $\oO_\epsilon$.

\begin{figure}[ht]

\begin{center}
\includegraphics[scale=1]{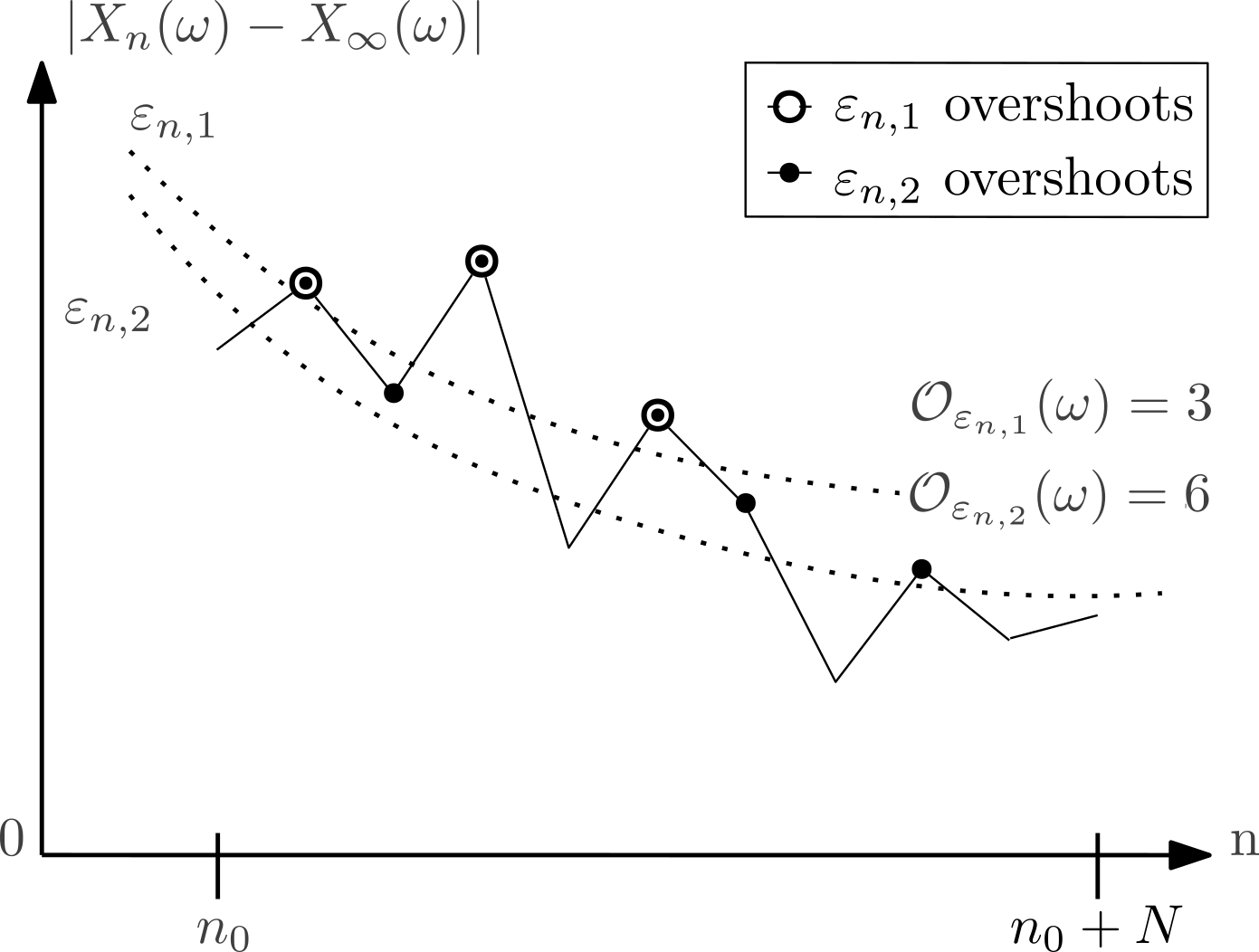}
 
\end{center}
\caption{Schematic of the error $|X_n(\omega)-X(\omega)|$ along the time index $n$. 
A larger error tolerance {$\e_{n,1}$} yields a smaller number of overshoots {$\oO_{\e_{n,1}}(\omega)$}, vice versa, a smaller error tolerance {$\e_{n,2}$} yields a larger number of overshoots {$\oO_{\e_{n,2}}(\omega)$}.}
\label{fig:overshoot}
\end{figure}
 \item The relation \eqref{e:MDFmoment} bounds \textit{nonlinear} higher order moments of $\oO_\epsilon = \sum_{n=n_0}^\infty \ind(A_n(\e_n))$ {and $\m_\epsilon$ by the constant $C_a$, whose finiteness is an elementary, weighted, \textit{linear} condition \eqref{e:GewSumme} on $(\PP(A_n(\e_n)))_{n\in \NN}$.} We refer to Example~\ref{ex:poly}, \ref{ex:exp} and \ref{ex:Weibull}. This is simple to verify and therefore allows for the retroactive and meaningful quantification of many known results of complete convergence (or even only a sufficiently strong convergence in probability) in the literature. A sample of applications (still for fixed $\e>0$) is given in \cite{EstraHoeg22}. This article shows the utility of such a concept for almost sure martingale convergence and strong laws for martingale differences more generally for nonincreasing sequences $\epsilon = (\e_n)_{n\in \NN}$. 
 
 \item The results offer {potential for applications in statistics}, due to the fine play between the integrability of $\oO_\epsilon$ and the asymptotic rate $\epsilon$. Therefore, this tradeoff relation looks like a good candidate for the construction of powerful tests by counting infraction frequencies of a error rate $\epsilon$, and allow to assess the a.s.~consistency of $M$-estimators, a statistical class of point estimators, as carried out in Subsection~\ref{ss:M-estimator}. 
\end{enumerate}
The estimates obtained through Lemma \ref{lem:BC1} are necessarily suboptimal, however, not by much. On the one hand, we note that for $b_n = a_n \cdot \sum_{m=n}^\infty \PP(A_n(\e_n))$ the space $\ell^1$ of summable sequences $(b_n)_{n\in \NN}$ is not closed in $\ell^\infty$. That is, for any such summable sequence $(b_n)_{n\in \NN}$, you can find another one which is asymptotically of {a} slightly higher order, but which is still summable. This rather subtle and theoretical objection is inherent in any kind of quantification of integrability by summability. On the other hand, our quantitative version of the first Borel-Cantelli lemma uses in a crucial step a suboptimal union bound. {Yet, this} union bound {encodes} the fact that our quantification of a.s.~convergence does not quantify the last occurrence of an error incidence, which would coincide with the (random) modulus of convergence, but the random number of ocurrences until eventually complying with the error threshold. The difference lies in possible sparseness of error indices before finally dying out. 
However, we show that this effect only affects the integrability {of} $\oO_\epsilon$ {and $\m_\epsilon$} for rates of $\PP(A_n(\e_n))$ given by inverse {monomials} with small exponents, see Example~\ref{ex:poly}. For high order polynomially, exponentially or Weibull-type fast rates $\PP(A_n(\e_n))\searrow 0$, as $n\ra\infty$, this effect is essentially negligible (see Example \ref{ex:exp} and \ref{ex:Weibull}). 

We highlight the utility of the previously mentioned tradeoff between error tolerance and deviation frequency (error incidence) in the context of martingale convergence theorems and the strong laws for martingale differences. 
There is a large literature on discrete martingales, which we cannot review here. The concept of martingale differences first emerged in L\'evy's monography \cite{Lev37} as a technical device to relax the independence in the central limit theorem even before the term martingale was coined and conceptualized by Ville \cite{Vi39} in the context of fair games and still formulated in the controversial language of von Mises' collectives, \cite[Section 1.3]{Maz09}. We refer to the classical monographs \cite{Doob53, FS04, Pr90, Wi91} for an introduction to discrete martingales. Nowadays, martingales are at the core of many applications. 
    
First we study martingales which are uniformly bounded in $L^p$, $p\gqq 2$, and  
with a.s.~uniformly bounded increments with the help of the {Azuma-Hoeffding} inequality. Next we establish the strong law for martingale differences, {for the cases where: they are} are not necessarily bounded in $L^p$ {; they} are uniformly bounded in $L^p${; and when they have} uniformly bounded exponential moments. Nowadays, there are many very fine martingale estimates in probability well-established, for an overview see \cite{FGL15}. Many of them are suitable for a run-off between the almost sure error tolerance \eqref{e:errortolerance} and the mean deviation frequency quantification in \eqref{e:MDFmoment}. 
The preceding tradeoff is applied in four major applications:  
1) (multicolor) P\'olya's urn with applications including preferential attachment trees, 2) the Generalized Chinese Restaurant Process with applications in machine learning, 
3) a quantification of the a.s.~convergence of statistical $M$-estimators in five different settings and 4) the number of outliers for the Galton branching processes. 

Finally, we address the theoretical problem of item (b) at the beginning in Corollary~\ref{cor:KFMDF} and ~\ref{cor:MDFKF}. We show that a.s.~MDF convergence cannot cure the fact that no topologization is possible {for} a.s.~convergence (see Remark~\ref{rem:notopology}). However, we relate the concept of a.s.~MDF convergence to the convergence in the classial Ky Fan metric, which metrizes (and topologizes) the convergence in probability on $L^0$, by meaningful quantitative estimates. In Corollary~\ref{cor:KFMDF} we obtain upper bounds of the Ky Fan metric in case of the summability of \eqref{e:GewSumme} and in Corollary~\ref{cor:MDFKF} we infer a certain a.s.~MDF convergence in case of a summable sequence of Ky Fan metric errors. 

\section*{\textbf{Organization of the article}}

We start in Section~\ref{s:qBC1} with the proof of a quantitative version of the Borel-Cantelli lemma in Lemma~\ref{lem:BC1} and the tradeoff between \eqref{e:errortolerance} and \eqref{e:MDFmoment} in Lemma~\ref{lem:Quant BC for e_n}. In Section~\ref{s:martconv} we study martingale convergence theorems. 
First we quantify the Pythagorean theorem of martingale convergence in $L^2$ in Subsection~\ref{ss:martboundL2}, in Subsection~\ref{ss:Azuma} 
we quantify the Azuma-Hoeffding exponential closure and its MDF consequences. 
Section~\ref{s:SLLNMD} starts with a.s.~MDF convergence results with the strong law of large numbers for not necessarily bounded data in $L^p$. For bounded data in $L^p$ we use the optimal Baum-Katz-Nagaev type results in Subsection~\ref{ss:BK}. Finally, Subsection~\ref{ss:LesigneVolny} treats the strong law for martingale differences which have uniformly bounded exponential moments. In Section~\ref{s:applications} we present several applications. Subsection~\ref{ss:multicolor} is dedicated to the assessment of the a.s.~convergence of multicolor P\'olya urn models. Subsection~\ref{ss:ChineseRestaurant} illustrates the convergence of a Generalized Chinese Restaurant Process. In Subsection~\ref{ss:M-estimator} we establish the statistical convergence results on $M$-estimators. Finally, Subsection~\ref{ss:branching} is dedicated to the MDF quantification of the convergence of the martingales associated to the Galton-Watson branching process. 
{Section~\ref{s:other} gives an outlook on a quantification of martingale maximal inequalities and the law of the iterated logarithm for martingales.} 
In Appendix~\ref{s:KyFan} we present the relation of a.s.~MDF convergence and bounds on the Ky Fan metric, and some auxiliary optimization results in Appendix~\ref{s:optimal}. 

\section*{\textbf{Preliminaries and notation} } 

In this article the natural numbers $\NN = \{1, 2, \dots, \}$ do not contain $0$, while $\NN_0 = \{0, 1,2, \dots \}$. Throughout this article all random vectors are defined over a common given probability space $(\Omega, \aA, \PP)$. A filtered probability space is a probability space $(\Omega, \aA, \PP, \FF)$ equipped with a filtration $\FF = (\fF_n)_{n\in \NN_0}$ that is a sequence of sub $\sigma$-algebras $\fF_n \subseteq \aA$ which satisfy $\fF_{n} \subseteq \fF_{n+1}$ for all $n\in \NN_0$. 
We use the convention that for sums $\sum_{n=n_0}^{n_0+N-1} a_n$ for some $n_0, N\in \NN_0$ and a real sequence $(a_n)_{n\in \NN_0}$, the value $\sum_{n=n_0}^{n_0-1} a_n$ is $0$.

In this article, all appearing Polish spaces $\xX$ are considered to be equipped with their respective Borel $\sigma$-algebra, that is, the $\sigma$-algebra generated by the open sets. In case of
a separable Banach space $(B, \|\cdot \|)$ equipped with its Borel-sigma-algebra $\bB$, we recall the definition of a martingale (and the one of a martingale difference sequence) with values in $B$:

\begin{enumerate}
 \item A stochastic process $(X_n)_{n\in \NN_0}$ on a given filtered probability space $(\Omega, \aA, \PP, \FF)$ with values in $B$ 
 is called martingale with respect to $\mathbb{F}$ if it satisfies the following three conditions: 
 \begin{enumerate}
  \item[\textnormal{(i)}] $\EE[|X_n| ] <\infty \qquad \mbox{ for all }n\in \NN_0$. 
 \item[\textnormal{(ii)}] $(X_n)_{n\in \NN_0}$ is $\FF$-adapted, that is, $X_n$ is $(\fF_n, \bB)$-measurable for all $n\in \NN_0$. 
 \item[\textnormal{(iii)}] $\EE[X_n~|~\fF_{n-1}] = X_{n-1}\qquad \PP\mbox{-a.s.~for all }n\in \NN$. 
 \end{enumerate}
\item A stochastic process $(X_n)_{n\in \NN_0}$ with values in $B$ 
 is called a sequence of martingale differences (MDs) with respect to $\FF$
 if it satisfies the following three conditions: items \textnormal{(i)} and \textnormal{(ii)} of (a) and 
\[
\EE[X_n~|~\fF_{n-1}] = 0 \qquad \PP\mbox{-a.s.~for all }n\in \NN.  
\]
\end{enumerate}

\noindent In Section~\ref{s:martconv}, \ref{s:SLLNMD} and \ref{s:applications} we apply the results of Section \ref{s:qBC1} to several examples of martingales.\medskip

\noindent Results for martingales with values in infinite dimensional spaces require the notion of $p$-smooth Banach spaces (following e.g.~\cite{Luo21} or \cite{Pisier75}) which we state here in brevity:

\noindent A Banach space is called {\it $p$-uniformly smooth} for a fixed $p\in (1,2]$ if there is a constant $s\geq 0$ such that for all $\tau>0$,
\begin{align*}
\sup\Big\{\tfrac{\|x+\tau y\|+\|x-\tau y\|}{2}-1:\|x\|=\|y\|=1\Big\}\leq s\tau^p.
\end{align*}
Note that all Hilbert spaces are $2$-uniformly smooth (by the parallelogram identity) and for $p>1$, the $L^p$ spaces (over a probability space) are $\min(p,2)$-uniformly smooth.

Most of our results for martingales in infinite dimensions rely on concentration equalities for Banach spaces. Our choices of such inequalities (a variety of the Azuma inequalities from \cite{Luo21} and Baum-Katz type-estimates \cite{Giraudo18}) can of course be extended, e.g.~using the findings in \cite{Naor12,Pisier75} or \cite{Pi94}.

\bigskip
    
\section{\textbf{A quantitative version of the first Borel-Cantelli lemma}}\label{s:qBC1}     
\noindent We start by extending the result given in \cite[Theorem 1]{EstraHoeg22}. 
\begin{defn}
Let $(A_n)_{n\in \NN_0}$ be a sequence of events in a probability space $(\Omega,\fF,\PP)$. For $n_0\in\NN_0$ we call 
\begin{equation}\label{def:O}
 \oO_{n_0}(\omega) := \sum_{n=n_0}^\infty \ind(A_n)(\omega), \qquad \omega \in \Omega,   
 \end{equation}
 the \textbf{overlap count} of $(A_n)_{n\in \NN_0}$ and 
  \begin{equation}
  \label{eq:m}\m_{n_0}(\om) := \max\{i\gqq n_0~|~\om \in A_i\}
  \end{equation}
  the \textbf{last occurrence index} of $(A_n)_{n\in \NN_0}$. \\
  \noindent For a nonnegative, nondecreasing sequence $a = (a_n)_{n\in \NN_0}$ we define 
  \begin{equation}\label{def:Sa}
\sS_{a, n_0}(N) := 
 \sum\limits_{n=0}^{N-1} a_{n_0+n} \qquad \mbox{ for }N\in \NN\qquad \mbox{ and }\qquad \sS_{a, n_0}(0) := 0.\qquad 
 \footnote{
 Note that this definition of $\sS_a$ corrects an off-by-one error in \cite[Thm 1]{EstraHoeg22}. 
 Compare with Example~\ref{ex:poly} and Example~\ref{ex:exp} below. 
}
\end{equation}
\end{defn}
\noindent The function $\sS_{a, n_0}$ represents the order of the moments of $\oO_{\epsilon, n_0}$ and $\m_{\epsilon, n_0}$. It is (due to summation by parts) the ``antiderivate'' of the sequence of ``weights'' $(a_n)_{n\in \NN}$. The following lemma gives sufficient conditions on upper bounds of $\EE[\sS_{a, n_0}(\oO_{\epsilon, n_0})$, and $\EE[\sS_{a, n_0}(\m_{\epsilon, n_0})$, respectively. The examples afterwards illustrate how these moments are upper bounds of polynomial, exponential or Weibull type moments in concrete situations. 
\begin{lem}[\textbf{Quantitative version of the first Borel-Cantelli lemma}]\label{lem:BC1} 
Given a probability space $(\Omega, \aA, \PP)$, $n_0\in \NN_0$, and a sequence of events $(A_n)_{n\gqq n_0}$, such that 
\[
\sum_{n=n_0}^\infty \PP(A_n) < \infty.  
\]
Then for any positive, nondecresing sequence $(a_n)_{n\gqq n_0}$, the following statements are true: 
\begin{enumerate}
 \item If the sequence $(A_n)_{n\gqq n_0}$ is nested, that is, $A_{n+1}\subseteq A_n$, $n\gqq n_0$, it follows that 
 \[
 \EE[\sS_{a, n_0}(\oO_{n_0})] = \sum_{n=n_0}^\infty a_{n} \PP(A_n),  
 \]
 
 \item {Consider a sequence $(A_n)_{n\gqq n_0}$, which is not necessarily nested.
 Then the following relations are valid: 
 \begin{enumerate}
  \item[i)] For all $\omega \in\Omega$ we have 
  \begin{equation}
  \m_{n_0}(\om) = \sum_{n=n_0}^\infty \ind\Big(\bigcup_{m=n}^\infty A_m\Big)(\omega).
  \end{equation}
  \item[ii)] We have the moment estimate 
  \begin{equation}\label{e:Ka}
  \EE[\sS_{a, n_0}(\oO_{n_0})]\lqq \EE[\sS_{a, n_0}(\m_{n_0})] = \sum_{n=n_0}^\infty a_n \PP\Big(\bigcup_{m=n}^\infty A_m\Big)
  \lqq  \sum_{n=n_0}^\infty a_n \sum_{m=n}^\infty \PP(A_m) = K_a.   
  \end{equation}
\end{enumerate}
} 
 \end{enumerate}
\end{lem}

\begin{rem}
\begin{enumerate} 
\item The nestedness hypothesis in item (a) in Lemma~\ref{lem:BC1} only applies {directly under} particular circumstances, see for instance Corollary~\ref{cor:BK} item (b), Theorem~\ref{thm:Freedman} or Remark~\ref{rem:critial}. However, we obtain an exact formula, whereas in the general case of item (b) we only obtain an upper bound. 
{For $\oO_{n_0}$ the} difference between the nested case (a) and (b) lies in the replacement of the sequence $\PP(A_n)$ by the sequence $\sum_{m=n}^\infty \PP(A_n)$, which is clearly suboptimal, as can be seen in Example~\ref{ex:poly}. 
However, in Example~\ref{ex:exp} and \ref{ex:Weibull} below, we see that this gap in the order is often negligible. 

  \item For a positive sequence of real numbers $a= (a_n)_{n\gqq n_0}$ and $N\gqq 0$, $n_0\in \NN$, we note that $\mathcal{S}_{a, n_0}(N):=\sum_{n=0}^{N-1}a_{n_0+n}$ is a 'discrete antiderivate' of $a$ w.r.t.~the counting measure. This function itself might seem a bit involved, however, it is often estimated from below without much effort. In order to obtain a lower bound of $\EE[\sS_{a, n_0}(\oO_\epsilon)]$ we use the comparison principle for sums and (Riemann-) integrals. 

 \item Since $a$ is nondecreasing, the relation \eqref{e:Ka} implies that 
 $\sum_{m=n}^\infty \PP(A_m)<\infty$ such that the classical first Borel-Cantelli lemma applies. 
 Note that Lemma~\ref{lem:BC1} can only quantify the excess of summability in $(\PP(A_n))_{n\gqq n_0}$, 
 it cannot turn non-summable sequences into summable ones. 
 \item Note that the finiteness on the right-hand side in estimate \eqref{e:Ka} is a linear condition in $a$ for a nonlinear higher moment of $\oO_{n_0}$. 
\end{enumerate}
\end{rem}

\begin{proof}[\textbf{Proof of Lemma~\ref{lem:BC1}:} ]
We start with the proof of (a). Fix some $n_0, N\in \NN$ and define 
$\oO_{n_0, N} := \sum_{m=n_0}^{N+n_0} \ind(A_m)$. Note that by construction 
\[
\oO_{n_0, N} \in \{0, \dots, N+1\}.  
\]
By the nestedness we have for each $k=1, \dots, N$ that 
\[
\PP(\oO_{n_0, N} = k) = \PP(A_{n_0 + k-1} \setminus A_{n_0 + k}) = \PP(A_{n_0 + k-1}) - \PP(A_{n_0 + k}). 
\]
In addition, $\PP(\oO_{n_0, N} = 0) = \PP(\Omega \setminus A_{n_0})$ and $\PP(\oO_{n_0, N} = N+1) = \PP(A_{N+n_0})$, compare with Figure~\ref{fig:Venn}.\\  
\begin{figure}[ht]
\begin{center}
\includegraphics[scale=0.6]{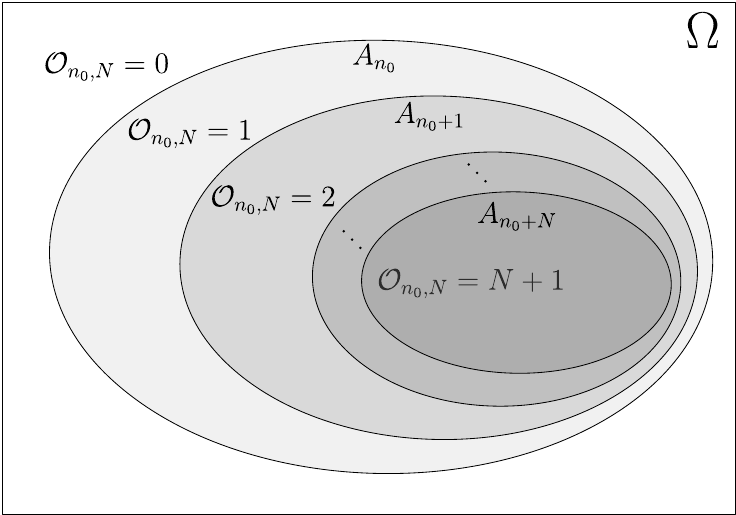}
 \end{center}
\caption{Overlap statistic $\oO_{n_0, N}$ of the nested events $A_{n_0}\supseteq A_{n_0+1}\supseteq \dots \supseteq A_{n_0+N}$}
\label{fig:Venn}
\end{figure} 

\noindent Note that by the definition of $\oO_{n_0, N}$ we have the following representation 
\begin{align*}
\EE[\sS_{a, n_0}(\oO_{n_0, N})] 
&= \sS_{a, n_0}(0) \PP( \oO_{n_0, N} = 0) + \sum_{k=1}^N \sS_{a, n_0}(k) \PP( \oO_{n_0, N} = k) + \sS_{a, n_0}(N+1) \PP( \oO_{n_0, N} = N+1)\\
&= \sS_{a, n_0}(0) \PP(\Omega \setminus A_{n_0}) + \sum_{k=1}^N \sS_{a, n_0}(k) \PP( \oO_{n_0, N} = k) + \sS_{a, n_0}(N+1) \PP(A_{N+n_0}). 
\end{align*}
Integration by parts yields for any sequences $(f_k)_{k\in \NN_0}$ and $(g_k)_{k\in \NN_0}$ that 
\begin{align*}
\sum_{k=0}^{N} f_k g_k = f_{N} \sum_{k=0}^{N} g_k - \sum_{j=0}^{N-1} (f_{j+1} - f_j) \sum_{\ell = 0}^j g_k. 
\end{align*}
For notational convenience we set $p_{n_0+k} = \PP(A_{n_0 + k})$. Hence for $f_k = p_{n_0+k}$ and $g_k = a_{n_0 +k}$ we obtain 
\begin{align*}
\sum_{k=0}^{N} a_{n_0+k} p_{n_0+k} = p_{n_0+N} \sum_{k=0}^{N} a_{n_0+k} + \sum_{j=0}^{N-1} (p_{n_0+j}-p_{n_0+j+1}) \sum_{\ell = 0}^j a_{n_0+k}. 
\end{align*}
In other words, for all $N\in \NN_0$ we have the formula (bear in mind the innermost set in Figure~\ref{fig:Venn})
\begin{align*}
\sum_{k=0}^{N} a_{n_0+k} &\PP(A_{n_0 + k})
= \PP(A_{n_0+N}) \sum_{k=0}^{N} a_{n_0+k}+ \sum_{j=0}^{N-1} 
(\PP(A_{n_0 +j})-\PP(A_{n_0 +j+1})) \sum_{\ell = 0}^j a_{n_0+\ell}\\ 
&= \PP(A_{n_0+N}) \sum_{k=0}^{N} a_{n_0+k} + \sum_{j=0}^{N-1} \Big(\sum_{\ell = 0}^j a_{n_0+\ell}\Big) \PP(\oO_{n_0, N} = j+1)\\ 
&= \PP(A_{n_0+N}) \sum_{k=0}^{N} a_{n_0+k} + \sum_{j=1}^{N} \Big(\sum_{\ell = 0}^{j-1}   a_{n_0+\ell}\Big) \PP(\oO_{n_0, N} = j)\\ 
&= \sum_{j=1}^{N} \Big(\sum_{\ell = 0}^{j-1}   a_{n_0+\ell}\Big) \PP(\oO_{n_0, N} = j)+\Big(\sum_{k=0}^{N} a_{n_0+k}\Big) \PP(\oO_{n_0, N} = N+1)\\[3mm]
&= \EE[\sS_{a, n_0}(\oO_{n_0, N})],\\
% &\cred{= \PP(A_{n_0+N}) \sum_{k=0}^{N} a_{n_0+k} + \sum_{j=1}^{N+1} \Big(\sum_{\ell = 0}^{j-1}   a_{n_0+\ell}\Big) \PP(\oO_{n_0, N} = j) - \sum_{\ell=0}^{N} a_{n_0+\ell} \PP(\oO_{n_0, N} = N+1)}\\ 
% &\cred{= \sum_{j=1}^{N+1} \Big(\sum_{\ell = 0}^{j-1}   a_{n_0+\ell}\Big) \PP(\oO_{n_0, N} = j)
% = \sum_{j=0}^{N+1} \Big(\sum_{\ell = 0}^{j-1}   a_{n_0+\ell}\Big) \PP(\oO_{n_0, N} = j)}\\ 
% &\cred{= \EE[\sS_{a, n_0}(\oO_{n_0, N})] {-  \sS_{a, n_0}(N+1) \PP(A_{N+n_0})},}
\end{align*}
if and only if $\sS_{a, n_0}(N) = \sum_{\ell=0}^{N-1}a_{n_0+\ell}$ with the convention that $
\sS_{a, n_0}(0)=0$. Sending $N\ra\infty$, the monotone convergence theorem implies 
\[
\EE[\sS_{a, n_0}(\oO_{n_0})] = \sum_{k=0}^{\infty} a_{n_0+k} \PP(A_{n_0 + k}) = \sum_{\ell = n_0}^\infty a_{\ell}\cdot \PP(A_{\ell}). 
\]
This shows item (a). 

We continue with item (b){(ii)}. Define {$\tilde A_n := \bigcup_{m=n}^\infty A_m$ and $\tilde O_{n_0} = \sum_{n=n_0}^\infty \ind(\tilde A_n)$}. Note that the sequence $(\tilde A_n)_{n\gqq n_0}$ is a nested sequence of events by construction. At the same time we have by construction the monotonicity 
$\oO_{n_0} \lqq \tilde \oO_{n_0}$ a.s.~and by the nonnegativity of the sequence $(a_n)_{n\gqq n_0}$ that $\sS_{a, n_0}$ is nondecreasing and 
\[
\EE[\sS_{a, n_0}(\oO_{n_0})] \lqq  \EE[\sS_{a, n_0}(\tilde \oO_{n_0})] = \sum_{n=n_0}^\infty a_n \PP(\tilde A_n), 
\]
while by a union bound we have 
\[
\EE[\sS_{a, n_0}(\oO_{n_0})] \lqq \EE[\sS_{a, n_0}(\tilde \oO_{n_0})] \lqq \sum_{n=n_0}^\infty a_n \sum_{m=n}^\infty \PP(A_m). 
\]
This shows item (b){(ii)}. 

{
It remains to show (b){(i)}. By definition, $\m_{n_0}$ being the last index of the sets $A_j$ to which $\om$ belongs, we have that $\m_{n_0}(\omega) = i$ implies that $\omega \in A_{n_0 +i}$ and $\omega \notin A_{n_0+j}$ for all $j\gqq i+1$. 
In particular, $\m_{n_0}(\omega) = i$ yields 
\[
\omega \in \tilde A_i = A_i \cup \bigcup_{j\gqq i+1}^\infty A_j \qquad \mbox{ and } \qquad \omega \notin \tilde A_j = \bigcup_{\ell\gqq j} A_\ell \quad \mbox{ for all }\quad     j\gqq i+1. 
\]
In addition, $\om \in \tilde A_k$ for $k\lqq i$, since $\tilde A_k\supseteq \tilde A_i$ for all $k\lqq i$ by construction. This implies $\tilde \oO_{n_0}(\omega) = i$. \\
Conversely, if we assume that $\tilde \oO_{n_0}(\omega) = i$, then the nestedness of the sequence $(\tilde A_i)_{i\gqq n_0}$ yields that $\omega \in \tilde A_{n_0+i} \setminus \tilde A_{n_0+i+1}$. By definition, this yields 
\begin{align*}
\omega &\in \Big(\bigcup_{\ell = i}^\infty A_\ell\Big) \setminus  \Big(\bigcup_{k = i+1}^\infty A_k\Big) 
= A_i \setminus  \Big(\bigcup_{k = i+1}^\infty A_k\Big).  
\end{align*}
That is, $\omega \in A_i$ and $\omega \notin A_{j}$ for all $j\gqq i+1$. That is, $\m_{n_0} = i$. This finishes the proof of (b)(i). 
}

\end{proof}

\begin{exm}[\textbf{Polynomial probability decay}]\label{ex:poly}
{Assume $\PP(A_m) \lqq c m^{-q}$ for all $m\gqq n_0$ for some given constants $q,c>0$ and $n_0\gqq 1$. 
Then it is shown below that for any $0 \lqq p < q-2$ we have 
\begin{equation}\label{e:polymoment}
 \EE[\oO_{n_0}^{p+1}] {\lqq \EE[\m_{n_0}^{p+1}]} \lqq cq \zeta(q-p-1;n_0).
\end{equation}
For $n_0 = 1$ this result coincides with \cite[Example 1]{EstraHoeg22} except for the corrected prefactor $cq$ here. In addition, for any $0 \lqq p < q-2$ it follows by Markov's inequality and \eqref{e:polymoment} that 
\begin{equation}\label{e:polytailsub}
\PP(\oO_{n_0}\gqq k) {\lqq \PP(\m_{n_0}\gqq k)} \lqq cq \cdot k^{-(p+1)} \cdot  \zeta(q-p-1;n_0) \qquad \mbox{for }k\gqq 1,\\ 
\end{equation}
where $\zeta(z; n_0) = \sum_{n=n_0} \frac{1}{n^z}$ is the classical Hurwitz zeta-function. 
The rate can be optimized and we obtain 
\begin{equation}\label{e:polytailopt}
\PP(\oO_{n_0}\gqq k) {\lqq \PP(\m_{n_0}\gqq k)}  \lqq 
 c_1  \cdot k^{-(q-1)} \cdot \Big(\ln(k)+\frac{1}{n_0}-\psi(n_0)\Big) \qquad \mbox{for }k\gqq e^{\frac{1}{q-2} + \psi(n_0)}, 
\end{equation}
where the constants $c_1$ and $\psi(n_0)$ in the case with optimal rate are given below. 
We note that the optimal rate $k^{-(q-1)}$ is only valid for sufficiently large values of $k$. 
}\\
\medskip
\noindent Statements~\ref{e:polymoment} and~\ref{e:polytailsub} are seen as follows. For $a_n = n^{p}$, $p>0$, we have the following estimate 
\begin{align}\label{e:Riemann}
\sum_{n=n_0}^\infty n^p \sum_{m = n}^\infty c m^{-q} 
&\lqq c \sum_{n=n_0}^\infty n^p \Big(n^{-q} + \int_{n}^\infty x^{-q} dx\Big)\nonumber\\
&= c \zeta(q-p;n_0) +  \frac{c}{q-1} \zeta(q-p-1; n_0)\lqq \frac{c q}{q-1}\zeta(q-p-1;n_0). 
 \end{align}
where $\zeta(z;n_0) = \sum_{n=n_0}^\infty \frac{1}{n^z}$ is the Hurwitz zeta-function.  
The right-hand side of \eqref{e:Riemann} 
is finite if and only if $p-q+ 1< -1$, i.e. $0 < p< q-2$. 
{
At the same time due to $n_0\gqq 1$ we have 
\begin{align*}
\sS_{a, n_0}(N) 
 &= \sum_{n=n_0}^{N+n_0-1} n^p \gqq \sum_{n=1}^N n^p\gqq \int_0^{N} x^p dx = \frac{N^{p+1}}{p+1}.
\end{align*}
Hence, by inequality \eqref{e:Ka} from Lemma \ref{lem:BC1}, we obtain \eqref{e:polymoment}:  
\begin{align*}
 \EE[\oO_{n_0}^{p+1}] { \lqq \EE[\m_{n_0}^{p+1}]} 
\lqq (p+1) \frac{cq}{q-1} \zeta(q-p-1;n_0)
\lqq cq \zeta(q-p-1;n_0).
\end{align*}
We apply the Markov inequality and obtain \eqref{e:polytailsub}, which we further optimize with respect to $p$ 
\begin{align*}
\PP(\oO_{n_0}\gqq k){ \lqq \PP(\m_{n_0}\gqq k)} \lqq \inf_{p\in [0, q-2)} k^{-(p+1)} c q \zeta(q-p-1;n_0), \qquad k\gqq 1.
\end{align*}

Differentiating w.r.t.~$p$ yields 
\begin{align*}
0 &= \frac{d}{dp} (k^{-(p+1)} c q \zeta(q-p-1;n_0))= -\ln(k) k^{-(p+1)} c q \zeta(q-p-1;n_0)- k^{-(p+1)} cq \zeta'(q-p-1;n_0), 
\end{align*}
such that 
$\ln(k) = - \frac{\zeta'(q-p-1;n_0)}{\zeta(q-p-1;n_0)}$. The right-hand side 
for large values in the $\zeta$-function (and $\zeta'$) 
yields that 
\[
\ln(k) \approx \frac{1}{q-p-2} + \psi(n_0), \qquad \mbox{ where }  \psi(n_0) = \frac{\Gamma'(n_0)}{\Gamma(n_0)}, 
\]
which yields the optimizer $p^* = q-2 - \frac{1}{\ln(k) - \psi(n_0)}$ and hence 
for $k\gqq e^{\frac{1}{q-2} + \psi(n_0)}$ we have 
\begin{align}\label{e:eq1}
\PP(\oO_{n_0}\gqq k) {\lqq \PP(\m_{n_0}\gqq k)}
&\lqq cq \cdot k^{-(q-1)} \cdot k^\frac{1}{\ln(k)-\psi(n_0)}\cdot \zeta(1+ \frac{1}{\ln(k)-\psi(n_0)}; n_0).
\end{align}
Note that the right-hand side for large $k$ behaves asymptotically as 
\[
cqe \cdot k^{-(q-1)} \cdot \ln(k). 
\]
In order to obtain an upper bound we use the integral comparison principle
\begin{align*}
\zeta(1+ \frac{1}{\ln(k)-\psi(n_0)}; n_0) 
&= \sum_{n= n_0}^\infty n^{-1+ \frac{1}{\ln(k)-\psi(n_0)}} 
\lqq n_0^{-1+ \frac{1}{\ln(k)-\psi(n_0)}} + \int_{n_0}^\infty x^{-1+ \frac{1}{\ln(k)-\psi(n_0)}} dx\\
&= n_0^{-1+ \frac{1}{\ln(k)-\psi(n_0)}} + (\ln(k)-\psi(n_0)) n_0^{\frac{1}{\ln(k)-\psi(n_0)}}.  
\end{align*}
For $k\gqq e^{\frac{1}{q-2} + \psi(n_0)}$ this yields 
\begin{equation}\label{e:eq2}
\zeta(1+ \frac{1}{\ln(k)-\psi(n_0)}; n_0) 
 \lqq n_0^{q-3} +(\ln(k)-\psi(n_0)) n_0^{q-2}. 
\end{equation}
Combining \eqref{e:eq1} and \eqref{e:eq2} we obtain \eqref{e:polytailopt}:  
for any $0 \lqq p < q-2$ and  $c_1 = cqe^{(q-2)\psi(n_0)} n_0^{q-2}$ it follows 
\[
\PP(\oO_{n_0}\gqq k) { \lqq \PP(\m_{n_0}\gqq k)} \lqq c_1  \cdot k^{-(q-1)} \cdot (\ln(k)+\frac{1}{n_0}-\psi(n_0)) \qquad  \mbox{for }k\gqq e^{\frac{1}{q-2} + \psi(n_0)}.
\]
}
\end{exm}

\medskip 
\begin{exm}[\textbf{Exponential probability decay}]\label{ex:exp}
{
Assume $\PP(A_m) \lqq c b^{m}$ for all $m\gqq n_0$, for some given constants $n_0\in \NN$, $b\in (0,1)$ and $c>0$. Then we have 
\begin{align}\label{e:expmoment}
\EE[b^{-p\oO_{n_0}}] {\lqq \EE[b^{-p\m_{n_0}}]} 
\lqq 1+  \frac{cb^{n_0-1}}{1-b^{1-p}}
\end{align}
and for all $k\gqq 1$ 
\begin{align}\label{e:exptailsub}
\PP(\oO_{n_0}\gqq k) {\lqq \PP(\m_{n_0}\gqq k)} 
&
\lqq 2e^{\frac{9}{8}} \cdot \big[k(c b^{n_0-1}+1)+1] \cdot b^{k}.\\\nonumber
\end{align}
}
\medskip 
This is seen as follows. For any $a_n = b^{-pm}$, $p\in (0,1)$, we have  
\[
\sum_{m=n}^\infty c b^{m} =   \frac{c}{1-b} b^{n}
\]
and by \eqref{def:Sa}
\[
\sS_{a, n_0}(N) = \sum_{n=0}^{N-1} b^{-p(n_0+m)} = b^{-pn_0}\frac{(b^{-p})^{N}-1}{b^{-p}-1} 
,\]
such that 
\[
b^{-pN} = b^{pn_0} (b^{-p}-1)\cdot \sS_{a, n_0}(N)+1.
\]
Consequently we have 
\begin{align*}
\EE[b^{-p\oO_{n_0}}] {\lqq \EE[b^{-p\m_{n_0}}] }
&\lqq b^{pn_0} (b^{-p}-1)\cdot \EE[\sS_{a, n_0}(\oO)]+1 \lqq b^{pn_0} (b^{-p}-1)\cdot \sum_{n=n_0}^\infty b^{-pn} \sum_{m=n}^\infty cb^{m}+1\\ 
&\lqq 1+ (b^{-p}-1) \frac{cb^{pn_0}}{1-b}\sum_{n=n_0}^\infty  b^{(1-p)n} = 1+ (b^{-p}-1) \frac{cb^{pn_0}b^{(1-p)n_0}}{(1-b)(1-b^{1-p})} \\
&\lqq 1+ (\frac{1}{b}-1) \frac{cb^{pn_0}b^{(1-p)n_0}}{(1-b)(1-b^{1-p})} = 1+ \frac{1-b}{b} \frac{cb^{n_0}}{(1-b)(1-b^{1-p})}  = 1+  \frac{cb^{n_0-1}}{1-b^{1-p}}.
\end{align*}
This shows \eqref{e:expmoment}. Markov's inequality and \cite[Lemma 5]{HS23} yield~\eqref{e:exptailsub}.
\end{exm}

\medskip 

\begin{exm}[\textbf{Weibull type probability decay}]\label{ex:Weibull}
{
Assume $\PP(A_m) \lqq c b^{m^\alpha}$ for all $m\gqq n_0$, for some given constants $n_0\in \NN$, $b\in (0,1), \alpha\in (0,1)$ and $c>0$. 
Then for all $p\in (0,1)$ there is a constant $K = K(b, p, \alpha, n_0)>0$ given below such that 
\[
\EE[b^{-p (\oO_{n_0} + n_0-1)^{\alpha}}] {\lqq \EE[b^{-p (\m_{n_0} + n_0-1)^{\alpha}}]}\lqq K,  
\]
such that for all $k\gqq 1$ we have 
\[
\PP(\oO_{n_0}\gqq k) {\lqq \PP(\m_{n_0}\gqq k)} \lqq b^{p (k-1)^\alpha}  K.
\]
Further optimization of the rate yields the existence of positive constants $d = d(c, \alpha, \beta, n_0), D = D(c, \alpha, \beta, n_0, p)>0$ such that for all $k\gqq 2$
\begin{align*}
\PP(\oO_{n_0}\gqq k) {\lqq \PP(\m_{n_0}\gqq k)} \lqq (d+D(k-1)^{2-\alpha})b^{(k-1)^\alpha}.\\
\end{align*}
}
\medskip 
This is seen as follows. For any sequence $a_n = b^{-pm^{\alpha}}$, $p\in (0,1)$,  
the integral comparison test yields 
\begin{align*}
\sum_{m=n+1}^\infty c b^{m^\alpha}\lqq c \int_{m=n}^\infty e^{-|\ln(b)| x^\alpha} dx. 
\end{align*}
For $t = |\ln(b)| x^\alpha$ that is $x = \big(\frac{t}{|\ln(b)|}\big)^\frac{1}{\alpha}$ 
and $dx = \frac{1}{\alpha |\ln(b)|} \Big(\frac{t}{|\ln(b)|}\big)^{\frac{1}{\alpha}-1} dt$ 
such that 
\begin{align*}
\int_{n}^\infty e^{-|\ln(b)| x^\alpha} dx 
&= \frac{1}{\alpha |\ln(b)|}\int_{|\ln(b)| n^\alpha}^\infty e^{-t}\big(\frac{t}{|\ln(b)|}\big)^{\frac{1}{\alpha}-1} dt= \frac{1}{\alpha |\ln(b)|^\frac{1}{\alpha}}\int_{|\ln(b)| n^\alpha}^\infty e^{-t}t^\frac{1-\alpha}{\alpha} dt\\
&\lqq \frac{1}{\alpha^2 |\ln(b)|^\frac{1}{\alpha}}  e^{-|\ln(b)| n^\alpha}(|\ln(b)| n^\alpha)^\frac{1-\alpha}{\alpha}= \frac{1}{\alpha^2 |\ln(b)|}\,  e^{-|\ln(b)| n^\alpha} n^{1-\alpha}.
\end{align*}
Hence for all $p\in (0,1)$ we have $a_n = b^{-p n^\alpha}$ 
\begin{align}\label{e:Weibullkonstante}
\sum_{n=n_0}^\infty a_n  \sum_{m=n}^\infty \PP(A_m) \lqq c \sum_{n=n_0}^\infty b^{-(1-p) n^\alpha}n^{1-\alpha} =: K(b, p, \alpha, n_0).
\end{align}
Finally, by \eqref{def:Sa}
\[
\sS_a(N) = \sum_{n=n_0}^{N+n_0-1} a_n =   \sum_{n=0}^{N-1} b^{-p (n+n_0)^\alpha} \gqq b^{-p (N+n_0-1)^\alpha},
\]
such that 
\[
{\EE[b^{-p (\m_{n_0} +n_0-1)^\alpha}]}\lqq K(b, p, \alpha, n_0), \qquad \mbox{ and }\quad \PP(\m_{n_0}\gqq k)\lqq  \inf_{p\in (0,1)} b^{p (k-1)^\alpha}K(b, p, \alpha, n_0).
\]
By Lemma~\ref{lem:Weibull} there are positive constants $d,D \in \NN$ such that for $k\gqq 2$ 
\begin{align*}
\PP(\oO_{n_0}\gqq k){\lqq \PP(\m_{n_0}\gqq k)} 
&\lqq \inf_{p\in (0,1)} b^{p (k-1)^\alpha} \sum_{n=n_0}^\infty c\bigg (1+\frac{1+\frac{\frac{1}{\alpha}-1}{|\ln(b)|}}{\alpha |\ln(b)|} n^{1-\alpha}\bigg)b^{(1-p) n^\alpha}\\
&\lqq (d+D(k-1)^{2-\alpha})b^{(k-1)^\alpha}.
\end{align*}
The precise values of $d$ and $D$ are given in the proof of Lemma \ref{lem:Weibull} in Appendix~\ref{s:optimal}.  
\end{exm}
\bigskip

\begin{defn}\label{def:completeMDF} Given a Polish space $\xX$ equipped with a complete metric $\mathrm{d}$ on $\xX$ which generates its topology. Consider a sequence of random vectors $(X_n)_{n\gqq n_0}$ and a random vector $X$ with values in $\xX$ and 
for any $\e>0$ the overlap statistics 
\begin{equation}\label{d:overlap}
\oO_{\e, n_0} := \sum_{n=n_0}^\infty \ind\{\mathrm{d}(X_n, X)>\e\}. 
\end{equation}
\begin{enumerate}
 \item We call $X_n\ra X$, as $n\ra\infty$, \textbf{completely convergent}, if for any $\e>0$ we have 
 \[
 \EE[\oO_{\e, n_0}] < \infty.   
 \]
 \item We say $X_n\ra X$ \textbf{converges a.s.~with mean deviation frequency convergence} (MDF convergence, for short) \textbf{of order $\Lambda_\cdot$}, if for any $\e>0$ there is a function $\Lambda_\e: \NN_0\ra(0, \infty)$ with 
 \begin{equation}\label{e:superlinear}\limsup_{n\ra\infty} \frac{\Lambda_\e(n)}{n} = \infty
 \end{equation}
 if for any $\e>0$ we have 
 \[
 \EE[\Lambda_\e(\oO_{\e, n_0})] < \infty.
 \] 
\end{enumerate} 
\end{defn}
\noindent We summarize the most important observations concerning complete and a.s.~MDF convergence. 
\begin{rem}
\begin{enumerate}
\item By monotone convergence we have $\EE[\oO_{\e, n_0}] = \sum_{n=n_0}^\infty \PP(\mathrm{d}(X_n,X)>\e)$, 
which is the original formulation by \cite{HR47, Yu99}. The notion of a.s.~MDF convergence was introduced for the first time in \cite[Definition 1]{EstraHoeg22}. 
\item MDF convergent random variables for any $\Lambda_\e$ satisfying \eqref{e:superlinear} are completely convergent. 
 \item Obviously in both cases, (a) and (b) of Definition~\ref{def:completeMDF}, we have $X_n\ra X$ in probability, and by the first Borel-Cantelli lemma, we have in both cases $X_n \ra X$ a.s. 
\item The motivation of the notion of MDF convergence is to quantify a.s.~convergence statistically by different orders of integrability of the associated family of overlap statistics $(\oO_{\e, n_0})_{\e>0}$.  
\item We may replace $\e$ by a sequence positive, nonincreasing sequence $\epsilon = (\e_n)_{n\gqq n_0}$ below, which allows to obtain almost sure rates of convergence with the help of the following quantitative version of the first Borel-Cantelli lemma. 
\item Estimates from MDF convergence yield decay rates for the probability $\PP(\oO_{\e, n_0}\gqq k)$ using
Markov’s inequality via
\[
\PP(\oO_{\e, n_0}\gqq k) \lqq \frac{\EE[\sS_{a, n_0}(\oO_{\e, n_0})]}{\sS_{a, n_0}(k)}, \qquad k\gqq 1,
\]
which then can be optimized over meaningful sequences of positive, nondecreasing weights $a = (a_n)_{n\gqq n_0}$. The same remains valid if we replace $\oO_{\e, n_0}$ by $\m_{\e, n_0}$. 
\end{enumerate}
 \end{rem}

\bigskip

\begin{lem}[\textbf{Almost sure error tolerance and mean deviation frequency}] \label{lem:Quant BC for e_n}
 Given a probability space $(\Omega, \aA, \PP)$, and a Polish space $\xX$ with a complete metric $\mathrm{d}$ on $\xX$ which generates the topology. 
We consider a sequence of random vectors $(X_n)_{n\gqq n_0}$ for some $n_0\in \NN$, $X_n: \Omega \to \xX$, $n\gqq n_0$, and a random vector $X: \Omega\to \xX$. Assume that $X_n$ converges to $X$ as $n\ra\infty$ in probability, that is,   
for any fixed $\delta>0$ we have 
\begin{align*}
p(\delta, n) := \PP(\mathrm{d}(X_n, X)>\delta) \ra 0, \qquad \mbox{ as }n\ra\infty.
\end{align*}
Then we have the following \textbf{tradeoff}:   
 For any positive, nonincreasing sequence $\epsilon = (\e_n)_{n\gqq n_0}$, and any positive, nondecreasing sequence $a= (a_n)_{n\gqq n_0}$ such that 
\begin{equation}\label{e:K}
K(a, \epsilon, n_0):= \sum_{n=n_0}^\infty a_n \sum_{m=n}^\infty p(\e_m, m) <\infty,
\end{equation}
it follows 
\begin{equation}\label{e:error tolerance}
\limsup_{n\ra\infty} \mathrm{d}(X_n,X)\cdot \e_n^{-1}\lqq 1\qquad \PP\mbox{-a.s.},  
\end{equation}
and 
\begin{equation}\label{e:error incidence}
\EE[\sS_{a, n_0}(\oO_{\epsilon, n_0})]{\lqq \EE[\sS_{a, n_0}(\m_{\epsilon, n_0})]} \lqq K(a, \epsilon, n_0),  
\end{equation}
where 
\begin{align}
&\oO_{\epsilon, n_0}(\omega) := \sum_{n=n_0}^\infty \ind\{\mathrm{d}(X_n(\omega),X(\omega))>\e_n\}, \qquad \omega\in \Omega,\label{def:Oe}\\
&{\m_{\epsilon, n_0}(\omega) := \max\{n\gqq n_0~|~\mathrm{d}(X_n(\omega),X(\omega))>\e_n\}, \qquad \omega\in \Omega, }\label{def:Me}
\end{align}
and $\sS_{a, n_0}$ is defined in \eqref{def:Sa}.
\end{lem}
\bigskip 

\begin{defn} In the situation of Lemma~\ref{lem:Quant BC for e_n} 
we call $\oO_{\epsilon, n_0}$ defined by \eqref{def:Oe} the \textbf{overlap statistic} or \textbf{deviation frequency} and $\m_{\epsilon, n_0}$ the \textbf{modulus of a.s.~convergence}.
We call the relation \eqref{e:error tolerance} an \textbf{a.s.~error tolerance of order $\epsilon$}, 
 while the relation \eqref{e:error incidence} is referred to as \textbf{mean deviation frequency (MDF) bound of order $\sS_{a,n_0}$.} 
\end{defn}

\begin{rem}\label{rem:tradeoff}\hfill
 \begin{enumerate} 
   
 \item Note that for different sequences of error tolerances $\epsilon = (\e_n)_{n\gqq 0}$ we obtain different rates $p_n = \PP(A_n(\e_n))$ as a function of $n$. 
 The tradeoff between \eqref{e:error tolerance} and \eqref{e:error incidence}
 is quantified by the play between $\epsilon = (\e_n)$ and $a = (a_n)$ 
 by the finiteness of the constant \eqref{e:K}. 
\item In this context, Lemma \ref{lem:Quant BC for e_n} generalizes the classical first Borel-Cantelli lemma \cite[Theorem 2.18]{Ka02} as follows: \\
For any positive, nonincreasing sequence $\epsilon = (\e_n)_{n\gqq n_0}$ such that 
 \[
 K_0(\epsilon, n_0) := \sum_{n=n_0}^\infty p(\e_n, n) <\infty 
 \]
 we have the error tolerance 
\begin{equation}\label{e:error tolerance2}
\limsup_{n\ra\infty} \mathrm{d}(X_n, X)\cdot \e_n^{-1}\lqq 1\qquad \PP\mbox{-a.s.},  
\end{equation}
and the mean deviation frequency of order $1$ 
\[
\EE[\oO_{\epsilon, n_0}] = K_0(\epsilon, n_0).
\]
In particular, the classical first Borel-Cantelli lemma does not yield information about the modulus of convergence $\m_{\epsilon, n_0}$.  

   \item For a constant sequence $\e_n = \e>0$, $n\gqq n_0$, we denote the same overlap statistic in a slight abuse of notation by $\oO_\e$, which coincides with the notation of \eqref{d:overlap} and Definition~\ref{def:completeMDF}. In this case the error incidence is of order 
   \[
   \Lambda_\e = \sS_{a, n_0}. 
   \]
   \item For fixed $\e_n = \e>0$ the rate $p(\e, n)\ra 0$ is the fastest possible among all nonincreasing sequences, which translates to the largest possible finite moments of the (random) overlap count (error incidence) in terms of $\EE[\sS_{a, n_0}(\oO_\e)]< \infty$. 
  \item For any $\epsilon = (\e_n)_{n\in \NN}$ such that $p(\e_n, n)$ is close to not being summable (such as for instance $\frac{1}{n^\theta}$, $\theta>1$ or $\frac{1}{n \ln^\theta(n+1)}$, $\theta>1$), 
  the usual Borel-Cantelli lemma implies a close to optimal almost sure error tolerance, however, the MDF bound is maximal, exhibiting linear decay at best, since by Markov's inequality 
  \[
  \PP(\oO_{\epsilon, n_0} \gqq k)\lqq k^{-1} \cdot \EE[\oO_\epsilon].  
  \]
   \end{enumerate}
\end{rem}

\noindent The proof of Lemma~\ref{lem:Quant BC for e_n} is based on Lemma~\ref{lem:BC1}.

\begin{proof}[\textbf{Proof of Lemma~\ref{lem:Quant BC for e_n}:}]
For any positive sequence $\epsilon = (\e_n)_{n\gqq n_0}$ we consider the events 
\[
A_n := \{\mathrm{d}(X_n,X)>\e_n\}, \qquad n\gqq n_0,    
\]
and by Lemma~\ref{lem:BC1}(b)(i) the respective overlap representation 
\[
 \m_{\epsilon, n_0} := \sum_{n=n_0}^\infty \ind\Big(\bigcup_{m=n}^\infty A_m\Big).  
\]
Since $K(\inf_{n} a_n, \epsilon) \lqq K(a,\epsilon) <\infty$ by hypothesis, 
we may apply Lemma~\ref{lem:BC1}. 
Then $\oO_{\e, n_0}\lqq \m_{\e, n_0}$ and the usual Borel-Cantelli lemma yields 
\[
0 = \PP(\mathrm{d}(X_n,X)>\e_n \mbox{ infinitely often }) = \PP(\limsup_{n\ra\infty} \mathrm{d}(X_n, X) \cdot \e_n^{-1} > 1), 
\]
and implies \eqref{e:error tolerance}. Furthermore, \eqref{e:K} and Lemma~\ref{lem:BC1} implies \eqref{e:error incidence}. This finishes the proof. 
\end{proof}

\noindent In Section \ref{s:martconv}, \ref{s:SLLNMD} and \ref{s:applications} 
we infer a.s.~MDF convergence results
for various classes of martingales and sequences of martingale differences of interest 
with the help of Lemma~\ref{lem:Quant BC for e_n}.

\bigskip
\section{\textbf{The tradeoff in almost sure martingale convergence theorems}}\label{s:martconv}

\noindent In the sequel we quantify the martingale convergence theorems with the help of Lemma~\ref{lem:Quant BC for e_n}. 

\subsection{\textbf{The tradeoff for martingales bounded in $\mathbf{L^p}$, $\mathbf{p\gqq2}$}}\label{ss:martboundL2}\hfill\\

\noindent We start with one of the most classical martingale convergence results in $L^2$ is due to Pythagoras' theorem. 

\begin{thm}[\textbf{Pythagoras' theorem for martingale differences}]\label{thm:L2convergence}
Given a filtered probability space $(\Omega, \aA, \PP, \FF)$, $n_0\in \NN_0$, we consider a martingale $X = (X_n)_{n\gqq n_0}$, with values in a Hilbert space $(H, \lgl\cdot, \cdot \rgl)$ and which satisfies \[\sup_{n\gqq n_0} \EE[\|X_n\|^2] <\infty.\] 
Then $X$ converges a.s.~and in $L^2(\Omega; H)$ to a random vector $X_\infty$ in $L^2(\Omega; H)$. 
In addition, for any positive, nonincreasing sequence $\epsilon = (\e_n)_{n\gqq n_0}$ it follows  
\begin{equation}\label{e:Py}
\PP(\|X_n-X_\infty\|>\e_n) \lqq \e_n^{-2}\cdot \EE[\|X_n-X_\infty\|^2] = \e_n^{-2}\cdot  \sum_{m=n+1}^\infty \EE[\|\Delta X_{m}\|^2] = \e_n^{-2}\cdot \pi_n, \qquad n\gqq n_0, 
\end{equation}
where $\Delta X_n := X_n - X_{n-1}$ and 
$\pi_n = \sum_{m=n+1}^\infty \EE[\|\Delta X_{m}\|^2]$ for $n\gqq n_0+1$.

\noindent Moreover, we have the following tradeoff: 
For all positive, nonincreasing sequences $\epsilon = (\e_n)_{n\gqq n_0}$ and 
positive, nondecreasing sequences $a = (a_n)_{n\gqq n_0}$ such that 
\[
K(a, \epsilon) := \sum_{n=n_0}^\infty a_n \sum_{m=n}^\infty  \e_m^{-2}\cdot \pi_m < \infty, 
\]
it follows 
\begin{equation}\label{e:Py1}
\limsup_{n\ra\infty} \|X_n-X_\infty\|\cdot \e_n^{-1} \lqq 1 \qquad \PP\mbox{-a.s.} 
\end{equation}
and
\begin{equation}\label{e:Py2}
\EE[\sS_{a, n_0}(\oO_{\epsilon, n_0})] {\lqq \EE[\sS_{a, n_0}(\m_{\epsilon, n_0})]} \lqq K(a, \epsilon),
\end{equation}
{where $\oO_{\epsilon, n_0}$ is given in \eqref{def:Oe}, $\m_{\epsilon, n_0}$ in \eqref{def:Me} and $\sS_{a, n_0}$ in \eqref{def:Sa}.} 
\end{thm}

\begin{proof}
The proof is a straight-forward extension of the Pythagoras theorem \cite[Subsection 14.18]{Wi91}
\[
\EE[|X_m-X_n|^2] = \sum_{\ell = n+1}^m \EE[|X_\ell|^2]
\]
to Hilbert spaces with a direct application of Lemma~\ref{lem:Quant BC for e_n}.   
\end{proof}

\begin{exm}[\textbf{Centered random walk}] Consider an independent sequence of centered square integrable random variables $(\Delta X_n)_{n\gqq 1}$ with values in some separable Hilbert space $(H, \lgl \cdot, \cdot \rgl)$.  
Hence the process of partial sums $(X_n)_{n\in \NN_0}$, $X_0 = 0$, $X_n := \sum_{i=1}^n \Delta X_i$, $n\gqq 1$, defines a martingale with respect to the natural filtration given by $\fF_n := \sigma(X_1, \dots, X_n)$. If 
\[
\sum_{n=1}^\infty \EE[\|\Delta X_n\|^2] <\infty,  
\]
we have that $X_n$ converges in $L^2$ and a.s.~For instance if $\mbox{Var}(\Delta X_n) = n^{-q}$, $n\gqq 1$, for some $q>3$ we obtain 
that 
\[
\pi_n = \sum_{m=n+1}^\infty \EE[\|\Delta X_m\|^2] \lqq \int_n^{\infty} \frac{1}{x^q} dx = \frac{n^{-(q-1)}}{q-1} . 
\]
In particular for $\epsilon = (\e_n)_{n\in \NN}$ with $\e_n = n^{-\alpha}$ such that $q-1-2\alpha >2$ it follows that 
\[
\sum_{m=n+1}^\infty \e_m^{-2} \pi_m = \sum_{m=n+1}^\infty (n+1)^{-(q-1-2\alpha)}\lqq \frac{n^{-(q-2-2\alpha)}}{q-2-2\alpha}. 
\]
Hence for $a_n = (n+1)^{p}$, $0< p< q-3-2\alpha$ we have by \eqref{def:Sa} that $\sS_{a, 1}(N) = \sum_{n=1}^{N} a_{n}$ for $N\in \NN$ and $\sS_{a, 1}(0) = 0$ 
such that  
\[
K := \sum_{n=1}^\infty a_n \sum_{m=n} \e_m^{-2} \pi_m \lqq \frac{1}{q-2-2\alpha} \sum_{n=1}^\infty n^{p-(q-2-2\alpha)} <\infty 
\]
implies \[\limsup_{n\ra\infty} \|X_n - X\|\cdot n^\alpha\lqq 1\qquad \PP\mbox{-a.s.}\]
and by Example~\ref{ex:poly}, 
$\EE[\oO_{\epsilon}^{1+p}]{\lqq \EE[\m_{\epsilon}^{1+p}]\lqq q \zeta(q - p - 1; n_0)}$, 
as well as for $k\gqq 1$ that 
\[
\PP(\oO_{\epsilon}\gqq k){\lqq \PP(\m_{\epsilon}\gqq k)}\lqq k^{-(p+1)} \cdot K,  
\]
which is further optimized in \eqref{e:polytailopt}. In other words, a sufficiently fast decay of the variances translates naturally into a higher order MDF convergence. 
\end{exm}

\bigskip

\begin{rem}
By a direct application of the the Burkholder-Davis-Gundy inequality \cite[Section (14.18)]{Wi91} it 
obvious how to generalize this result to a version for martingales which are uniformly bounded in $L^r$ for some $r>2$ and with rates 
\[
\tilde \pi_{n, r} :=\EE\left[\Big(\sum_{m=n+1}^\infty\|\Delta X_m\|^2\Big)^\frac{r}{2}\right]. 
\]
Inequality \eqref{e:Py} then reads 
\[
\PP(\|X_n-X_\infty\|>\e_n) \lqq  \e_n^{-r}\cdot \tilde \pi_{n,r}, \qquad n\gqq n_0,  
\]
and 
\[
K(a, \epsilon, r) := \sum_{n=n_0}^\infty a_n \sum_{m=n}^\infty  \e_m^{-r}\cdot \tilde \pi_{m,r} < \infty.  
\]
The formulation of the tradeoff between \eqref{e:Py1} and \eqref{e:Py2} result reads similar with the obvious adjustments of $a$ and~$\epsilon$. 
\end{rem}

\bigskip
\subsection{\textbf{The tradeoff for martingale convergence by the {Azuma-Hoeffding} inequality}}\label{ss:Azuma}\hfill\\

\noindent The Azuma-Hoeffding inequality replaces the absolute summability of the square integrals of $dX_n$ by the much stronger condition of a.s.~summability of the squares of $dX_n$. As a consequence, we obtain exponential estimates.  First, we consider the real-valued case, where Azuma-Hoeffding's inequality includes supermartingales.

\begin{thm}[\textbf{Azuma-Hoeffding inequality}]\label{thm:azuma}
Let $X=(X_n)_{n\in \NN_0}$ be a real-valued supermartingale with super-martingale differences $(\Delta X_n)_{n\in \NN}$. Assume that the sequence $(\Delta X_n)_{n\in \NN}$ is bounded almost surely by positive numbers $(c_n)_{n\in \NN}$, that is, 
\[
|\Delta X_n| \lqq c_n, \qquad \PP\mbox{-a.s.}\qquad \mbox{ for all } n\in \NN.
\]
Then it follows 
\begin{align*}
\PP(X_n-X_0\gqq \e)\lqq \exp\Big(-\tfrac{1}{2} \tfrac{ \e^2}{\sum_{k=1}^n c_k^2}\Big), \qquad \mbox{ for all }n\in \NN. 
\end{align*} 
\end{thm}

\noindent The proof goes back to \cite{Az67,Hoe63}. In the sequel we send $n\ra\infty$ 
and use the tail summability in order to infer almost convergence $X_n\ra X_\infty$ as $n\to \infty$, 
which can be quantified in terms of mean deviation frequencies in the sense of Definition~\ref{def:completeMDF}.

\begin{thm}(\textbf{The tradeoff via the Azuma-Hoeffding closure})\label{thm:cerraduraAzuma}\hfill\\
Assume the hypotheses of Theorem~\ref{thm:azuma} and, in addition, 
\[
\sum_{n=1}^\infty c_n^2<\infty.\] 
Let $r(n) := \sum_{k=n+1}^\infty c_k^2$ for $n\in \NN$. 
Then there exists an a.s.~finite random variable $X_\infty$ and we have $X_n\ra X_\infty$ a.s.~as $n\ra\infty$. 
More precisely, we have the following tradeoff:  
\begin{enumerate}
 \item For any nonincreasing positive sequence $\epsilon = (\e_n)_{n\in \NN}$
and any sequence of positive, nondecreasing weights $a = (a_n)_{n\in \NN}$ such that 
\begin{equation}\label{e:sumabilidad1}
K(a, \epsilon) := 2\sum_{n=1}^\infty a_n\sum_{m=n}^\infty\exp\Big(-\frac{1}{2}\frac{\e_m^2}{r(m)}\Big)  < \infty, 
\end{equation}
we have that 
\[
\limsup_{n\ra\infty} |X_n-X_\infty  | \cdot \e_n^{-1}\lqq 1, \qquad \PP\mbox{-a.s.} 
\]
and \begin{equation}\label{e:Azuma}
\mathbb{E}\left[\mathcal{S}_{a, 1}(\oO_\epsilon)\right]{\lqq \mathbb{E}\left[\mathcal{S}_{a, 1}(\m_\epsilon)\right]}\lqq K(a, \epsilon),  
\end{equation}
for $\oO_\epsilon=\sum_{n=1}^\infty\bI\{|X_\infty-X_n|\gqq\e_n\}${, $\m_{\epsilon} = \max\{n\gqq 1~|~|X_\infty-X_n|\gqq\e_n\}$}, and $\sS_{a, 1}$ is 
defined in \eqref{def:Sa}. \\ 
In particular, we have: 
\[
\PP(\oO_\epsilon \gqq k) {\lqq \PP(\m_\epsilon \gqq k)} \lqq \inf_{a}~ \sS_{a, 1}^{-1}(k)\cdot {2}\sum_{n=0}^\infty a_n\sum_{m=n}^\infty\exp\Big(-\frac{1}{2}\frac{\e_m^2}{r(m)}\Big), \qquad k\gqq 1, 
\]
where we optimize over suitable sequences of positive, nondecreasing numbers $a = (a_n)_{n\in \NN}$ satisfying \eqref{e:sumabilidad}. 
\item We obtain the following upper bound for the Ky Fan metric 
\begin{equation}\label{e:AzumaKyFan}
\mathrm{d}_{\mathrm{KF}}(X_n, X_\infty) \lqq \eta_n, \qquad \mbox{ where } \qquad \eta_n = \sqrt{r(n) \cdot W\big(r(n)^{-1}\big)},
\end{equation}
where $W$ is Lambert's W-function, with the well-known asymptotics \cite[Theorem 2.7]{HH08} 
\begin{equation}\label{e:Lambertasymptotics}
W(x) = \ln\Big( \frac{x}{\ln(\ln(x))}\Big) + o(1)_{x\ra\infty}.
\end{equation}
\end{enumerate}
\end{thm}

\begin{proof}
Clearly, by the martingale convergence theorem \cite{Wi91}, there is a closure $X_\infty$ such that $X_n\to X_\infty$ a.s.~For $0\lqq n\lqq m$, we have by Azuma's inequality that
\begin{align*}
\PP(X_m-X_n\gqq \e_n)\lqq \exp\biggl(-\frac{1}{2} \frac{\e_n^2}{\sum_{k=n}^mc_k^2}\biggr).
\end{align*}
Sending $m \to \infty$, the left hand side converges as $X$ 
converges to $X_\infty$ in probability and we obtain by Fatou's lemma 
\begin{align}
\PP(X_\infty&-X_n\gqq \e_n)
=\EE[\liminf_{m\ra\infty} \ind\{X_m-X_n\gqq \e_n\}]\nonumber\\ 
&\lqq \liminf_{m\ra\infty} \EE[ \ind\{X_m-X_n\gqq \e_n\}]=\liminf_{m\ra\infty}
\exp\biggl(-\frac{1}{2}\frac{\e_n^2}{\sum_{k=n}^mc_k^2}\biggr) =\exp\biggl(-\frac{1}{2}\frac{\e_n^2}{r(n)}\biggr).\label{e:AzumainWkt}
\end{align}
Note that the right-hand side is strictly decreasing by hypothesis as a function of $n$. 
Whenever 
\[
R(n) := \sum_{\ell = n}^\infty \exp\Big(-\frac{1}{2} \frac{\e_\ell^2}{r(\ell)}\Big) <\infty  
\]
for some (and hence all) $n\in \NN$ and, in addition, 
\[
\sum_{n=1}^\infty a_n R(n) < \infty,  
\]
we infer inequality \eqref{e:Azuma} by Lemma~\ref{lem:Quant BC for e_n}. 
Combining the monotonicity of $\sS_{a, n_0}$, Markov's inequality and \eqref{e:Azuma} we have 
\[
\PP(\oO_\epsilon \gqq k) {\lqq \PP(\m_\epsilon \gqq k)} \lqq \sS_{a, 1}(k)^{-1} \cdot \EE[\sS_{a, 1}(\m_\epsilon)]\qquad k\gqq 1.
\]
A subsequent optimization over the respective sequences $a$ yields the second statement. 
With the help of \eqref{e:AzumainWkt} it remains to solve the fixed point equation 
\[
\eta_n =  \exp\biggl(-\frac{1}{2}\frac{\eta_n^2}{r(n)}\biggr),
\]
which yields \eqref{e:AzumaKyFan}. 
\end{proof}

\medskip

\begin{cor}\label{cor:Azumapol}
Assume the hypotheses and notation of Theorem~\ref{thm:cerraduraAzuma}. 
Suppose for some $C>0$, $q\in (0,1]$ and $n_0\in \NN_0$ we have 
\[
r(n) \lqq \frac{C}{(n+1)^q}, \qquad \mbox{ for all }n\gqq n_0.   
\]
Then we have the following tradeoff: 
for all $a = (a_n)_{n\gqq n_0}$ positive, nondecreasing and 
$\epsilon = (\e_n)_{n\gqq n_0}$ positive, nonincreasing 
such that
\begin{align*}
K(a, \epsilon, n_0) := 2\sum_{n=n_0}^\infty a_n \sum_{m=n}^\infty \exp\Big(-\frac{\e_m^2 (m+1)^q}{2C}  \Big) < \infty 
\end{align*}
we have 
\[
\limsup_{n\ra\infty} |X_n-X|\cdot  \e_n^{-1}\lqq 1
\qquad \PP\mbox{-a.s.}
\]
versus 
\begin{align*}
\EE[\sS_{a, n_0}(\oO_{\epsilon, n_0})] {\lqq \EE[\sS_{a, n_0}(\m_{\epsilon, n_0})]} \lqq K(a, \epsilon).\\ 
\end{align*}
Moreover, we have the following special cases:  
\begin{enumerate}
 \item For $q\in (0,1]$, $\e_n = \sqrt{\frac{2C(2+\theta)\ln(n+1)}{(n+1)^q}}$, $\theta>0$, and $a_n = n^{p}$, $0 < p< \theta$ we have 
\begin{align*}
\limsup_{n\ra\infty} |X_n-X|\cdot \sqrt{\frac{(n+1)^q}{\ln(n+1)}} \lqq \sqrt{2C (2+\theta)}, \qquad \PP\mbox{-a.s.} 
\end{align*}
while 
\begin{equation}
\EE[\oO_{\epsilon, n_0}^{p+1}] {\lqq \EE[\m_{\epsilon, n_0}^{p+1}]}\lqq 2\theta \zeta(1+\theta -p; n_0 ). 
\end{equation}
For $k\gqq e^{\psi(n_0) + \theta^{-1}}$ it follows 
\[
\PP(\oO_{\epsilon, n_0} \gqq k) {\lqq \PP(\m_{\epsilon, n_0} \gqq k)} \lqq 2\theta \cdot k^{-(1+\theta)}\cdot  k^{\frac{1}{\ln(k)- \psi(n_0)}} \cdot \zeta(1+\theta -\frac{1}{\ln(k) -\psi(n_0)}; n_0).
\]
\item For $q\in (0,1)$ and fixed $\e>0$ we have 
\begin{align}\label{e:blosserLimes}
\limsup_{n\ra\infty} |X_n-X| = 0, \qquad \PP\mbox{-a.s.} 
\end{align}
while the deviation frequency $\oO_{\e, n_0}$ satisfies for all $p\in (0,1)$ 
\begin{align*}
\EE\Big[e^{\frac{\e^2}{2C} p \oO_{\e, n_0}}\Big] {\lqq \EE\Big[e^{\frac{\e^2}{2C} p \m_{\e, n_0}}\Big]}\lqq K(e^{-\frac{\e^2}{2C}}, p, q, n_0)
\end{align*}
as defined in \eqref{e:Weibullkonstante} of Example~\ref{ex:Weibull}
with tail decay for $k\gqq 2$ 
\begin{align*}
\PP(\oO_{\e, n_0}\gqq k){\lqq \PP(\m_{\e, n_0}\gqq k)} \lqq e^{-\frac{\e^2}{2C} \cdot (k-1)^q} 2 (d + D(k-1)^{2-q}),  
\end{align*}
where the constants are given in Lemma~\ref{lem:Weibull}. 
\item For $q = 1$ and fixed $\e>0$ we have \eqref{e:blosserLimes} 
and for any $0 < p<\frac{\e^2}{2C}$ the finite exponential moment  
\begin{equation}\label{e:EMDF} 
\EE\left[e^{p\oO_{\e, n_0}}\right]{\lqq \EE\left[e^{p\m_{\e, n_0}}\right]}\lqq 
1+   \frac{e^{-\frac{\e^2}{2C}(n_0-1)}}{1-e^{-\frac{(1-p)\e^2}{2C}}}.
\end{equation}
\noindent Moreover for any $k\gqq 1$ we have by an application of \cite[Lemma 5]{HS23}
\begin{equation}\label{e:EMDFdecay}
\PP(\oO_{\e, n_0}\gqq k){\lqq \PP(\m_{\e, n_0}\gqq k)}
\lqq 2e^{\frac{9}{8}} \cdot \big[k(2e^{-\frac{\e^2}{2C}(n_0-1)}+1)+1\big]\cdot e^{-\frac{\e^2}{2C}k}.
\end{equation}
\item We obtain the following upper bound for the Ky Fan metric 
\[
d_{\mathrm{KF}}(X_n, X) \lqq \eta_n, \qquad \mbox{ where }\qquad \eta_n = \sqrt{\frac{C W\big(\frac{(n+1)^q}{C}\big)}{(n+1)^q}},  
\]
and $W$ is Lambert's W-function.
\end{enumerate}
\end{cor}

\noindent The proof is an application of Theorem~\ref{thm:cerraduraAzuma} combined with  Example~\ref{ex:poly} and Example~\ref{ex:Weibull}. 
This result is applied in Subsection~\ref{ss:M-estimator} in Theorem~\ref{thm:momAzuma} in order to quantify the a.s.~convergence of M-estimators. 

For the higher dimensional, and in particular infinite-dimensional case, we state the following immediate  simplification of the Azuma inequality proven in \cite[Theorem 1.2]{Luo21}.

\begin{thm}[\textbf{Azuma-Hoeffding in infinite dimensions}]\label{thm:AzumaInf}
Let $B$ be a $p$-smooth Banach space for $1<p\leq 2$ and let $X=(X_n)_{n\in \NN_0}$ be a martingale with values in $B$ and differences $(\Delta X_n)_{n\in \NN}$. Assume that the sequence $(\Delta X_n)_{n\in \NN}$ is a.s.~bounded by a non-negative sequence $(c_n)_{n\in \NN}$. Then, there is a constant $K$, only depending on $X$ such that for all $\e>0$,
\begin{align*}
\PP\Big(\sup_{j\in \NN_0}\|X_j-X_0\|\geq \e\Big)\leq 2\exp\bigg(-\frac{\e^p}{2K\sum_{j=1}^\infty c_j^p}\bigg).
\end{align*}
\end{thm}

The according MDF martingale convergence tradeoff now takes the following form

\begin{thm}(\textbf{Martingale convergence by Azuma in higher dimensions})\label{thm:cerraduraAzumaInf}\hfill\\
Assume the hypotheses of Theorem~\ref{thm:AzumaInf}. Assume that 
\[
\sum_{n=1}^\infty c_n^p<\infty\] 
and set $r(n) := \sum_{k=n+1}^\infty c_k^p$ for $n\in \NN$. Further, assume that for a nonincreasing positive sequence $\epsilon = (\e_n)_{n\in \NN}$ and a positive, nondecreasing sequence $a = (a_n)_{n\in \NN}$,
\begin{equation}\label{e:sumabilidad1Inf}
K(a, \epsilon) := 2\sum_{n=1}^\infty a_n\sum_{m=n}^\infty\exp\Big(-\frac{1}{2}\frac{\e_m^p}{r(m)}\Big)  < \infty.
\end{equation}
Then the following assertions hold:
\begin{enumerate}
\item There exists an a.s.~finite random variable $X_\infty$ such that $$\limsup_{n\ra\infty} \|X_n-X_\infty  \| \cdot \e_n^{-1}\lqq 1, \qquad \PP\mbox{-a.s.}$$
\item For $\oO_\epsilon=\sum_{n=1}^\infty\bI\{\|X_\infty-X_n\|\gqq\e_n\}${, $\m_{\epsilon} = \max\{n\gqq 1~|~\|X_\infty-X_n\|\gqq\e_n\}$}, and $\sS_{a, 1}$ for \eqref{def:Sa}, we get
\[
\mathbb{E}\left[\mathcal{S}_{a, 1}(\oO_\epsilon)\right]{\lqq \mathbb{E}\left[\mathcal{S}_{a, 1}(\m_\epsilon)\right]}\lqq K(a, \epsilon).\] 
\item In particular, 
\[
\PP(\oO_\epsilon \gqq k) {\lqq \PP(\m_\epsilon \gqq k)} \lqq \inf_{a}~ \sS_{a, 1}^{-1}(k)\cdot {2}\sum_{n=0}^\infty a_n\sum_{m=n}^\infty\exp\Big(-\frac{1}{2}\frac{\e_m^p}{r(m)}\Big), \qquad k\gqq 1, 
\]
where we optimize over suitable sequences of positive, nondecreasing numbers $a = (a_n)_{n\in \NN}$ satisfying \eqref{e:sumabilidad1Inf}. 
\item We obtain the following upper bound for the Ky Fan metric 
\begin{equation*}
\mathrm{d}_{\mathrm{KF}}(X_n, X_\infty) \lqq \eta_n, \qquad \mbox{ where } \qquad \eta_n = \sqrt{r(n) \cdot W\big(r(n)^{-1}\big)},
\end{equation*}
where $W$ is Lambert's W-function, with the well-known asymptotics \eqref{e:Lambertasymptotics}, see \cite[Theorem 2.7]{HH08}.
\end{enumerate}
\end{thm}

\noindent Again, the proof is an application of Theorem~\ref{thm:AzumaInf} combined with  Example~\ref{ex:poly} and Example~\ref{ex:Weibull}. 

\begin{exm}[\textbf{Exponential MDF convergence for P\'olya's urn}]\label{ex:Polya}\hfill\\ 
Consider P\'olya's urn model as seen in \cite[Example 12.29]{Kle08} for an urn containing $N$ balls, out of which $B$ are black, and $N-B$ are white. Let $(\Delta Y_n)_{n\gqq 1}$ be the sequence of independent draws from the urn, such that $\Delta Y_n = 1$ if the $n$-th ball is black, and $\Delta Y_n = 0$ otherwise. Also, for each draw, the ball picked returns to the urn together with an additional ball of the same color. Then, if $Y_n=\sum_{i=1}^n \Delta Y_i$, we can establish the martingale representing the proportion of black balls in the urn after $n$ draws as $X_n :=\frac{Y_n + B}{n+N}$ {with $X_0 = \frac{B}{N}$}. Here, the martingale differences are bounded, since 
    \begin{align}
    |X_n - X_{n-1}| &= \left| \frac{Y_n + B}{n+N} - \frac{Y_{n-1} + B}{n-1+N}\right| 
    = \frac{1}{n+N} \left|Y_n + B - \left(\frac{n-1+N}{n-1+N}+ \frac{1}{n-1+N}\right)(Y_{n-1} + B)\right| \nonumber\\
    &= \frac{1}{n+N} \left|Y_n + B - \left(1+ \frac{1}{n-1+N}\right)(Y_{n-1} + B)\right| \nonumber\\
    & =\frac{1}{n+N}\left| \Delta Y_n +  \frac{1}{n-1+N} (Y_{n-1}+ B)\right|\lqq \frac{2}{n+N}.\label{e:Polyasakrament}
    \end{align}
    Hence $X_n\ra X_\infty$ a.s.~and $X_\infty \sim \mbox{Beta}(B, N-B)$. In particular, $(X_n)_{n\in \NN_0}$ and $X_\infty$ satisfy the conditions of Corollary~\ref{cor:Azumapol} with 
    \[
    r(n) = \sum_{k=n}^\infty \frac{2}{(N+k)^2}\lqq \frac{2}{n-1+N}{\lqq \frac{3}{n}},
    \] 
    such that \eqref{e:EMDF} and \eqref{e:EMDFdecay} hold for $C = \frac{1}{3}$ and $p=1$. As a consequence, this P\'olya urn martingale converges a.s.~with exponential MDF in the sense of \eqref{e:EMDF} and, for any $\e>0$ fixed, the decay of the overlap $\oO_{\e} := \sum_{n=1}^\infty \ind\{|X_n-X_\infty|>\e\}$, and $\m_{\e}:= \max\{n\gqq 1~|~|X_n-X_\infty|>\e\}$, $n_0 = 1$ {is bounded by  
    \[
    \PP(\oO_\e\gqq k){\lqq \PP(\m_\e\gqq k)}\lqq
    2e^{\frac{9}{8}}\cdot \Big(2k+1\Big) \cdot e^{-\frac{\e^2}{6}k}, \qquad k\gqq 1.  
    \]
    }
    \noindent For any $p< \frac{1}{2}$ and $\epsilon = (\e_n)_{n\in \NN}$ with 
    $\e_n = \sqrt{\frac{2}{3n^p}}$ we have that 
    \[
    \PP(|X_n - X_\infty|>\e_n)\lqq 2\exp(- n^{1-2p}). 
    \]
    By Example~\ref{ex:Weibull} we obtain for any $\theta \in (0,1)$ a constant $K(e^{-1}, \theta, 1-2p, 1)>0$ given in \eqref{e:Weibullkonstante} such that 
    \[
    \EE[e^{\theta \oO_{\epsilon,1}^{1-2p}}] {\lqq \EE[e^{\theta \m_{\epsilon,1}^{1-2p}}]} \lqq K(e^{-1}, \theta, 1-2p, 1) <\infty, 
    \]    
    and there are constants $d, D>0$ defined in Lemma~\ref{lem:Weibull} such that 
    \[
    \PP(\oO_{\epsilon,1}\gqq k) {\lqq \PP(\m_{\epsilon,1}\gqq k)}\lqq 2(d+D(k-1)^{1+2p})e^{-(k-1)^{1-2p}}, \qquad \mbox{ for all }k\gqq 2. 
    \]

\end{exm}

\bigskip
\begin{exm}[\textbf{Doubly exponential tradeoff bound for a super-critical Galton-Watson process}]\label{ex:branching}
\noindent Branching processes have a long history and are very well-studied objects with precisely known dynamics. For the different regimes of sub-critical, critical and super-critical branching we 
quantify the a.s.~MDF dynamics. Let $Z = (Z_n)_{n\in \NN_0}$ be a Galton-Watson process with i.i.d.~offspring variables $(Y_{i,n})_{i, n\gqq 1}$ and expectation $\EE[Y_{1,1}] = \mm\in[0,\infty)$, where 
\begin{align}\label{def:GaltonWatson}
Z_{n+1} = \sum_{i=1}^{Z_n} Y_{i, n+1}\qquad Z_0 = 1.
\end{align}
We define $v := \mbox{Var}(Y_{1,1}) \in [0, \infty]$. 
It is well-known, see e.g. \cite[Proof of Theorem 8.1]{Harris64}, 
that $X_n:=\frac{Z_n}{\mm^n}$ defines a martingale with respect to the natural filtration. 
Consider a super-critical Galton-Watson process with $\mm > 1$, $v<\infty$ and bounded support 
$C:= \sup (\mbox{\textnormal{supp}}(Y_{1,1}))  < \infty$. Then 
\[
|\Delta Z_n| \lqq  \frac{C}{\mm^n} = c_n, \mbox{ which is clearly square summable.}  
\]
Hence $r(n) = \sum_{i=n}^\infty c_i^2 = C \frac{\mm^{-2n}}{\mm
-1}$ and for $n-1\gqq n_0$ for some $n_0\in \NN_0$  
\begin{align*}
R(n) &= \sum_{\ell=n}^\infty \exp\Big(-\frac{\e_\ell^2 (\mm -1) \mm^{\ell}}{2 C}\Big) 
\lqq \int_{n-1}^\infty \exp\Big(-\frac{\e_x^2 (\mm -1) \mm^{x}}{2 C}\Big)dx\\
&\lqq \exp\Big(-\frac{\e_{n-1}^2 (\mm -1) \mm^{n-1}}{2 C}\Big) \ln(1+\exp\big(\frac{\e_{n-1}^2 (\mm -1) \mm^{n-1}}{2 C}\big)\Big)\\
&\lqq 2 \exp\Big(-\frac{\e_{n-1}^2 (\mm -1) \mm^{n-1}}{2 C}\Big) \frac{\e_{n-1}^2 (\mm -1) \mm^{n-1}}{2 C}\\
&\lqq \frac{\mm-1}{C}\exp\Big(-\frac{\e_{n-1}^2 (\mm -1) \mm^{n-1}}{2 C}\Big) \e_{n-1}^2 \mm^{n-1}.
\end{align*}
Hence for any $\rho \in (1, \mm)$, $\e_n(\rho) = (\frac{\rho}{\mm})^{\frac{n}{2}}$ and $a_n(\ti \rho, \rho) = e^{\ti \rho^{n}}$, $1 < \ti \rho < \rho$, we have 
\[
K(\ti \rho, \rho, n_0) := \sum_{n=n_0}^\infty a_n(\ti \rho, \rho) R(n) < \infty, 
\]
and by \eqref{def:Sa}, $\sS_a(N)\gqq e^{\ti \rho^{N-1}}$ and $\sS_a(0) = 0$. Consequently, for $\oO_\epsilon = \sum_{n=n_0}^\infty \{|X_n-X_\infty|>\e_n\}$ {and $\m_{\epsilon} = \max\{n\gqq n_0~|~|X_n-X_\infty|>\e_n\}$} 
we have the doubly exponential decay 
\[
\PP(\oO_\epsilon \gqq k) {\lqq \PP(\m_\epsilon \gqq k)} \lqq e^{-\ti \rho^{k-1}} K(\ti \rho, \rho, n_0),\qquad k\gqq 1,
\] 
which can be further optimized over suitable exponents $\rho$ and $\ti \rho$. 
\end{exm}

\bigskip 
\begin{exm}[\textbf{The tradeoff for discrete stochastic integrals}]\hfill\\
Let $\Delta =(\Delta_n)_{n\gqq 1}$ be a sequence of i.i.d.~centered random variables on a probability space $(\Omega,\mathcal{A},\PP)$ which are bounded by a positive constant $C_1>0$. Define $\mathcal{F}_{n}:=\sigma(X_{k}:1\lqq k\lqq n)$ for $ n\gqq 1$ and set $\mathcal{F}_0:=\{\emptyset,\Omega\}$. Let $g=(g_n)_{n\gqq 1}$ be a sequence of random variables, uniformly bounded by a real number $C_2>0$ and such that $g_n$ is $\mathcal{F}_{n-1}$-measurable. 
Set 
\[
X_0:=0\qquad \mbox{ and }\qquad X_n:=\sum_{k=1}^ng_k\frac{\Delta_k}{k} \qquad \mbox{ for }n\gqq 1.
\] 
Then, $X=(X_n)_{n\gqq 0}$ converges to an $X_\infty$ $\PP$-a.s.~with $c_n=\frac{C_1\cdot C_1}{n}$ and $r(n)=(C_1\cdot C_2)^2\sum_{k=n}^\infty\frac{1}{k^2}$. It follows that there is a constant $C_3>0$ depending on $C_1$ and $C_2$ such that $\frac{1}{r(n)}\gqq C_3 n$. Then the hypotheses of Corollary \ref{cor:Azumapol} are valid for $(X_n)_{n\in \NN_0}$ and $X_\infty$. In particular, Corollary~\ref{cor:Azumapol} (a) and (b) apply. 
\end{exm}

\bigskip 

\section{\textbf{The tradeoff in the strong law for martingale differences (MDs)}}\label{s:SLLNMD}

\noindent To obtain the deviation frequencies for the strong law of large numbers for (centered) martingales $X$, i.e. quantifying the $\PP$-a.s.~convergence of $\frac{X_n}{n}\to 0$, we need several estimates for the martingale's moments, and for its difference sequence. Those will result in appropriate concentration inequalities. The case of absolute, monomial moments will be covered 
in Theorem~\ref{thm:martSLLNq}. Exponential moments will be treated afterwards using Theorem \ref{lem:exponentialLdp}
. The bounded case is treated in Example~\ref{rem:martbdSLLN}.

\subsection{\textbf{The tradeoff in the strong law for MDs in $L^p$}}\label{ss:Dharma}\hfill\\

\noindent The following result is an application of the Burkholder-Rosenthal inequality (see e.g.~\cite{JFZ85,Osek12,Rosenthal70,Talagrand89}). Note that the martingale differences in the subsequent result may or may not be uniformly bounded in $L^p$. In fact, the optimal result for uniformly bounded martingale differences in $L^p$ is given in Subsection \ref{ss:BK}. 

\begin{thm} \label{thm:martSLLNq}
Let $X=(X_n)_{n\in \NN_0}$ be a separable Hilbert space $(H,\langle\cdot,\cdot\rangle)$-valued martingale with $X_0=0$ with respect to a filtration $\FF=(\cF_n)_{n\geq 0}$. Set $\Delta X_j:=X_j-X_{j-1}$ for $j\in \NN$. Let $p\gqq 2$ and assume that for
\begin{equation}\label{e:Dharma}
\beta_{n,p}:=\frac{1}{n}\EE\bigg[\sum_{j=1}^n\|\Delta X_j\|^p+\Big(\sum_{j=1}^n\EE\left[\|\Delta X_j\|^2\middle|\fF_{j-1}\right]\Big)^\frac{p}{2}\bigg]
\end{equation}
we have finiteness of the value 
\[
K_{p} :=\sum_{n=1}^\infty \frac{\beta_{n, p}}{n^{p-1}} <\infty.  
\]
Then 
\begin{align*}
\lim_{n\to\infty}\frac{X_n}{n}=0,\quad \PP\text{-a.s.}
\end{align*} 
In addition, we have the following tradeoff: 
\begin{enumerate}
 \item For any $\e>0$ with $\mathcal{O}_\e:=\sum_{n=1}^\infty \bI\{\|\frac{X_n}{n}\|>\e\}${, $\m_{\e} := \max\{n\gqq 1~|~ \|\frac{X_n}{n}\|>\e\}$} and any positive, nondecreasing sequence $a = (a_n)_{n\in \NN}$ such that 
\[
K_{a, p}:= \sum_{n=1}^\infty a_n \sum_{m=n}^\infty \frac{\beta_{m, p}}{m^{p-1}} <\infty, 
\]
we have for $C_p:=2^{\frac{3p}{2}}\big(\frac{p}{4}+1\big)\big(1+\tfrac{p}{\ln(p/2)}\big)^p$,
\begin{align}\label{e:SLLNpolMDF}
&\EE \left[\sS_{a, 1}(\oO_\e)\right]{\lqq \EE \left[\sS_{a, 1}(\m_\e)\right]}<C_p K_{a, p} \quad\text{ and }\\[5mm]
&\PP(\oO_\e\gqq k){\lqq \PP(\m_\e\gqq k)}\lqq \sS_{a, 1}(k) C_p K_{a, p},\quad \text{for all }k\gqq 1.\\\nonumber
\end{align}
\item For any positive, nonincreasing sequence $\epsilon = (\e_n)$ and any positive, nondecreasing sequence $a = (a_n)_{n\in \NN}$ such that 
\[
K_{a, \epsilon, p}:= \sum_{n=1}^\infty a_n \sum_{m=n}^\infty \frac{\beta_{m, p}}{\e_m^p m^{p-1}} <\infty 
\]
we have 
\begin{equation}\label{e:SLLNpolas}
\limsup_{n\ra\infty} \Big\|\frac{X_n}{n}\Big\| \cdot \e_n^{-1} \lqq 1 \qquad \PP\mbox{-a.s.} 
\end{equation}
and the respective quantities 
$$
\mathcal{O}_\epsilon:=\sum_{n=1}^\infty \bI\Big\{\Big\|\frac{X_n}{n}\Big\|>\e_n\Big\}{\qquad \mbox{ and }\qquad \m_\epsilon := \max\Big\{n\gqq 1~|~\Big\|\frac{X_n}{n}\Big\|>\e_n \Big\}}
$$ 
satisfy
\begin{equation}\label{e: summable}
\EE[\sS_{a, 1}(\oO_{\epsilon})] {\lqq \EE[\sS_{a, 1}(\m_{\epsilon})] }\lqq C_p K_{a, \epsilon, r}, 
\end{equation}
where $\sS_{a, 1}$ is defined in \eqref{def:Sa} for $n_0 = 1$.  
\end{enumerate}
\end{thm}
\noindent \noindent A criterion similar to (b) is applied for the tradeoff of a.s.~convergent M-estimators in Theorem~\ref{thm:momCesaro} in Section~\ref{ss:M-estimator}.

\begin{proof} We use Markov's inequality 
\begin{align*}
\PP\Big(\frac{\|X_n\|}{n} > \e_n \Big) 
&\lqq \e_n^{-p}\cdot \EE\Big[\Big(\frac{\|X_n\|}{n}\Big)^p\Big]
= \frac{\EE[\|X_n\|^p]}{(\e_n n)^p}, 
\end{align*}
and apply the Burkholder-Rosenthal inequality from \cite{Osek12},
\[
\EE[\|X_n\|^p]\leq C_p\EE\bigg[\sum_{j=1}^n\|\Delta X_j\|^p+\Big(\sum_{j=1}^n\EE\left[\|\Delta X_j\|^2\middle|\fF_{j-1}\right]\Big)^\frac{p}{2}\bigg]=C_p n \beta_{n,p},
\]
where $C_p$ is the value stated in the assertion. This yields $\PP\Big(\frac{\|X_n\|}{n} > \e_n \Big)  \lqq \frac{\beta_{n, p}}{\e_n^p n^{p-1}}$.
\end{proof}

\begin{rem}\label{rem:SLLnq}
Instead of the Burkholder-Rosenthal inequality, it is also possible to use a result of \cite{Dharma68}, not taking into account the martingale's conditional variances, stating that (for the real valued case)
\begin{align*}
\EE\Big[|X_n|^p\Big] \lqq \tilde{C}_p \cdot n^{\frac{p}{2}}\cdot \tilde{\beta}_{n, p},\qquad n\in \NN, 
\end{align*}
where $\tilde{\beta}_{n,p}=\frac{1}{n}\sum_{j=1}^n \EE\big[|\Delta X_j|^p\big]$ and $\tilde{C}_p=[8(p-1)\max\{1,2^{p-1}\}]^p$. We then obtain the results of Theorem \ref{thm:martSLLNq} with $\beta_{n,p}$ replaced by $\tilde{\beta}_{n,p}$ and 
the $n^p$ in the denominators by $n^{\frac{p}{2}}$.
\end{rem}

\begin{exm}\label{ex: iidex}
Taking $X_n:=\sum_{i=1}^n \Delta_i$ for a centered, i.i.d.~sequence $(\Delta_n)_{n\in \NN}$, the values $\beta_{n,p}$ in \eqref{e:Dharma} equal $\EE\|\Delta_1\|^p+n^{\frac{p}{2}-1}(\EE[\|\Delta_1\|^2])^\frac{p}{2}$.
Hence our convergence condition in Theorem \ref{thm:martSLLNq} turns to
\[
K_{a, \epsilon, p}:= \sum_{n=1}^\infty a_n \sum_{m=n}^\infty \frac{\EE[\|\Delta_1\|^p]+m^{\frac{p}{2}-1}(\EE[\|\Delta_1\|^2])^\frac{p}{2}}{\e_m^p m^{p-1}} <\infty,
\]
which is finite whenever $\sum_{n=1}^\infty \frac{a_n}{n^{\frac{p}{2}-1}}$ converges (and of course the expectations above are finite). To obtain finite $q$-th moments of $\oO_\epsilon$ and $\m_\epsilon$ for $1<q$, choose $a_n=n^{q-1}$. Then, the condition for finiteness of $K_{a,\epsilon,p}$ is $1<q<\frac{p}{2}-1$, which shows that Theorem \ref{thm:martSLLNq} includes the result of \cite[Theorem 7]{EstraHoeg22} (where p and q are switched and the constant $C_p$ differs).

Note that sequences with $\EE[|X_1|^p]< \infty$, as in this case here, are trivially bounded in $L^p$, such that 
the Baum-Katz-Nagaev type results as given in Subsection~\ref{ss:BK} apply. 
\end{exm}

\bigskip

\subsection{\textbf{The tradeoff for Baum-Katz-Nagaev type strong laws for MDs uniformly bounded in $L^p$}}\label{ss:BK}\hfill\\ 

\noindent  We start with a version of the classical Baum-Katz-Nagaev strong law of large numbers \cite[Theorem 3]{BK65}, which in general treats renormalized sums of centered i.i.d.~random variables $\frac{1}{n^\alpha} X_n$, $X_n = \sum_{i=1}^n \Delta_i$ for some $\alpha\lqq 1$ in the presence of certain finite moments $\EE[|\Delta_i|^p], p>2$. It is an extension of the strong law by Hsu-Robbins-Erd\"os \cite{Er49,HR47}. Recently, these results were further improved to randomly weighted sums of random variables, see \cite{MS18}.

\begin{thm}[Baum-Katz-Nagaev Strong Law]\label{thm:BK}
Consider an i.i.d.~family of centered random variables $(\Delta_n)_{n\in \NN}$. 
Then for any $\alpha>1$ and $p> 1$ such that $\frac{1}{2} <  \frac{\alpha}{p}\lqq 1$ 
the following statements are equivalent: 
\begin{enumerate}
 \item $\EE[|\Delta_1|^p]< \infty$.
 \item $\sum\limits_{n=1}^\infty n^{\alpha -2}\cdot \PP\Big(\frac{|X_n|}{n} > \eta n^{\frac{\alpha}{p}-1}\Big)<\infty$ for all $\eta>0$.
 \item $\sum\limits_{n=1}^\infty n^{\alpha -2}\cdot \PP\Big(\max\limits_{k \gqq n} \frac{|X_k|}{k^{\frac{\alpha}{p}}} > \eta\Big) < \infty$ for all $\eta>0$.
\end{enumerate}
\end{thm}

\noindent We use the preceding summabilities in order to obtain estimates on the mean deviation frequency.

\bigskip

\begin{cor}\label{cor:BK}
Assume the hypotheses of Theorem~\ref{thm:BK}. 
We define for $\eta>0$, and $\alpha, p>1$  
and $\epsilon(\alpha, \eta, p) = (\e_n(\alpha, \eta, p))_{n\in \NN}$ where $\e_n(\alpha, \eta, p) := \eta n^{\frac{\alpha}{p}-1}$ and some $\e>0$ fixed 
\begin{align*}
\oO_{\epsilon, n_0} &:= \sum_{n=n_0}^\infty \ind\Big\{\frac{|X_n|}{n} > \e_n(\alpha, \eta, p)\Big\}{,
\qquad \m_{\epsilon, n_0} := \max\Big\{n\gqq n_0~|~\frac{|X_n|}{n} > \e_n(\alpha, \eta, p)\Big\},}\\
\tilde \oO_{\e, n_0} &:= \sum_{n=n_0}^\infty \ind\Big\{\max_{k \gqq n} \frac{|X_k|}{k^{\frac{\alpha}{p}}}>\e\Big\}{, 
\qquad \tilde \m_{\e, n_0} := \max\Big\{n\gqq n_0~|~\max_{k \gqq n} \frac{|X_k|}{k^{\frac{\alpha}{p}}}>\e\Big\}}. 
\end{align*}
Assume $p>3$ and $\EE[|\Delta_1|^p] <\infty$. Then we have the following tradeoff:
\begin{enumerate}
 \item Then for any $\alpha>3$ with 
 $\frac{1}{2} < \frac{\al}{p} \lqq 1$ and $0\lqq \tilde p <\alpha -3$ and we have a constant $C>0$ such that 
 \[
 \sum_{n=n_0}^\infty n^{\tilde p}  \sum_{m=n}^\infty  \PP\Big(\frac{|X_n|}{n} > \e_n(\alpha, \eta, p)\Big)
\lqq 
C (\alpha-1)\zeta(\alpha-2-\tilde p, n_0) < \infty, 
  \]
  such that 
 \[
\limsup_{n\ra\infty} |\frac{X_n}{n}| \cdot \Big(\eta n^{\frac{\alpha}{p}-1}\Big)^{-1} \ra 0\qquad \PP\mbox{-a.s.},  
\]
and 
\[
\EE[\oO_{\epsilon, n_0}^{1+\tilde p}]{\lqq \EE[\m_{\epsilon, n_0}^{1+\tilde p}]} \lqq C (\alpha-1) \zeta(\alpha-2-\tilde p, n_0)
\]
such that 
\[
\PP(\oO_{\epsilon, n_0}\gqq k){\lqq \PP(\m_{\epsilon, n_0}\gqq k)}\lqq k^{-(1+ \tilde p)}\cdot 
C(\alpha-1) \zeta(\alpha-2-\tilde p, n_0)
, \qquad k\gqq 1. 
\]
\bigskip
\item Then for any $\alpha>2$ with 
 $\frac{1}{2} < \frac{\al}{p} \lqq 1$ and $0\lqq \tilde p <\alpha -2$ 
 we have a constant $C>0$ such that 
\[
\sum_{n=1}^\infty n^{\tilde p}  \PP\Big(\max_{k \gqq n} \frac{|X_k|}{k^{\frac{\alpha}{p}}}>\e\Big) 
\lqq 
C \zeta(\alpha-1-\tilde p, n_0) < \infty
\]
which implies 
  \[
\lim_{n\ra\infty} \max_{k \gqq n} \frac{|\Delta_k|}{k^{\frac{\alpha}{p}}} \lqq \e \qquad \PP\mbox{-a.s.}  
\]
and
\[
\EE[\tilde \oO_{\e, n_0}^{1+\tilde p}]{\lqq \EE[\tilde \m_{\e, n_0}^{1+\tilde p}]} \lqq C \zeta(\alpha -1 -\tilde p; n_0)
<\infty. 
\]
In particular, 
\[
\PP(\tilde \oO_{\e, n_0}\gqq k){\lqq \PP(\tilde \m_{\e, n_0}\gqq k)}\lqq k^{-\tilde p+1}\cdot C \zeta(\alpha -1 -\tilde p; n_0), \qquad k\gqq 1.  
\]
\end{enumerate}
\end{cor}
\noindent Note that the nestedness in part (b), slightly improves our MDF result for the same value of $\alpha>3$, while the a.s.~error tolerance remains the same. An asymptotically better version for large values of $k$ is given in Example~\ref{ex:poly}.

\begin{proof}[\textbf{Proof of Corollary~\ref{cor:BK}:} ] 
We recall Kronecker's lemma \cite[(12.7)]{Wi91}: For two positive sequences $(b_n)_{n\in \NN}$ and $(c_n)_{n\in \NN}$, where $\lim_{n\ra\infty} b_n = \infty$ we have that 
\begin{align*}
\sum_{n=1}^\infty \frac{c_n}{b_n} < \infty \qquad \mbox{ implies } \lim_{n\ra\infty} \frac{1}{b_n} \sum_{i=1}^n c_i = 0.  
\end{align*}
Assume 
\[
\sum_{n=1}^\infty \frac{n^{\alpha-2}}{\frac{1}{p_n}} = \sum_{n=1}^\infty n^{\al-2} p_n  < \infty. 
\]
Then Kronecker's lemma yields 
\[
\lim_{n\ra\infty} p_n \cdot \sum_{k=1}^n k^{\al-2} \lqq C \lim_{n\ra\infty } p_n \cdot \int_1^n x^{\al-2} dx=\lim_{n\ra\infty }  \frac{p_n}{\al-1}( n^{\al-1} -1)  =  0. 
 \]
In case of $p_n = \PP\Big(\frac{|X_n|}{n} > \eta n^{\frac{\alpha}{p}-1}\Big)$ we have for $c_n = n^{\alpha-2}$ and $b_n = p_n^{-1}$ that 
\[
\sum\limits_{n=1}^\infty \frac{n^{\alpha -2}}{p_n^{-1}} <\infty 
\]
and the fact that $p_n\searrow 0$ monotonically implies 
\[
0 = \lim_{n\ra\infty } p_n \sum_{k=1}^n k^{\alpha-2} \gqq  \lim_{n\ra\infty } p_n \int_{2}^n x^{\alpha-2} dx = \lim_{n\ra\infty } p_n \frac{1}{\alpha-1}(n^{\alpha -1} - 2^{\alpha-1})\gqq 0.
\]
Hence $\lim\limits_{n\ra\infty} p_n n^{\alpha-1} = 0$. Therefore there exists a $C>0$ such that 
\[
p_n \lqq \frac{C}{n^{\alpha-1}} \qquad \mbox{ for all }n\in \NN.  
\]
In other words, by the summability of Theorem~\ref{thm:BK}(b) 
there exists some $C>0$ such that for all $n\in \NN$ 
\begin{equation}\label{e:asymp}
\PP\Big(\frac{|X_n|}{n} > \e_n(\alpha, \eta, p))\Big)\lqq \frac{C}{n^{\alpha -1}}. 
\end{equation}
Then we apply Example~\ref{ex:poly} for $\alpha >3$. This finishes the proof of item (a). 

The proof of item (b) uses that similarly to \eqref{e:asymp} we have 
\begin{equation}
\PP\Big(\max\limits_{k \gqq n} \frac{|X_k|}{k^{\frac{\alpha}{p}}} > \eta\Big) \lqq \frac{C}{n^{\alpha-1}}, \qquad n\gqq 1,  
\end{equation}
and the fact that the events 
\[
\Big\{\max\limits_{k \gqq n} \frac{|X_k|}{k^{\frac{\alpha}{p}}} > \eta\Big\} 
\]
are nested. Hence by the first parts of Lemma~\ref{lem:BC1} and Lemma~\ref{lem:Quant BC for e_n} 
combined with Example~\ref{ex:poly} we have
\[
\EE[\tilde \oO_{\e, n_0}^{1+\tilde p}] {\lqq \EE[\tilde \m_{\e, n_0}^{1+\tilde p}] }\lqq C \alpha \zeta(\alpha-1-\tilde p, n_0). 
\]
\end{proof}

\begin{rem}
Due to the boundedness of the i.i.d.~sequences $(\Delta_n)_{n\in \NN}$ in $L^q$ 
the preceding result yields for $\alpha = 1$ 
an improvement of the integrability of the overlap $\oO_\e$ 
in Etemadi's strong law of large numbers \cite[Theorem 7]{EstraHoeg22} 
from moments of orders $2\lqq 1+p < \frac{q}{2} -1$
to higher moments of orders $2 \lqq 1+p < q -1$. 
\end{rem}

\noindent It is remarkable that the following result generalizes the preceding strong law to martingale differences, which are uniformly bounded in $L^p$. A proof is found in \cite[Theorem]{St07}, see also \cite{MS18}. 

\begin{thm}[\textbf{Baum-Katz-Stoica Strong Law for MDs}]\label{thm:BKmart}\hfill\\
Consider a sequence $(\Delta X_n)_{n\in \NN}$ of martingale differences bounded in $L^p$. 
Then for {all $\eta>0$ and } any $\alpha>1$ and $p> 1$ such that $\frac{1}{2} <  \frac{\alpha}{p}\lqq 1$ { we have that} 
\begin{equation}\label{eq:BKmartsum}
\sum_{n=1}^\infty n^{\alpha-2} \PP\bigg(\frac{|X_n|}{n}\gqq \eta n^{\frac{\alpha}{p}-1}\bigg) < \infty. 
\end{equation}
\end{thm}
\noindent Note that by an application of Kronecker's lemma in the proof of the subsequent Corollary~\ref{cor:BK} we have the asymptotic decay 
$\PP(\frac{|X_n|}{n}\gqq \eta n^{\frac{\alpha}{p}-1}) \lqq C n^{-(\alpha-1)}$. 

There are several extensions of this result applied to arrays of martingales in \cite{HL12}. 
In particular, there are several precise summability results for $q\gqq 2$, however, the tradeoff relation of Lemma~\ref{lem:Quant BC for e_n} does not apply directly. 

\bigskip

\begin{cor}\label{cor:BKmart}
For $\alpha> 3$, $\eta>0$ and $p>1$ 
such that $\frac{1}{2} <  \frac{\alpha}{p}\lqq 1$ we 
define $\epsilon = \epsilon(\alpha, \eta, p) = (\e_n(\alpha, \eta, p)_{n\in \NN}$, $\e_n(\alpha, \eta, p) := \eta n^{\frac{\alpha}{p}-1}$ and $n_0\in \NN$ 
\[
\oO_{\epsilon, n_0} := \sum_{n=n_0}^\infty \ind\left\{\frac{|X_n|}{n} \gqq \e_n(\alpha, \eta, p)\right\} 
\qquad {\mbox{ and }\qquad \m_{\epsilon, n_0} := \max\Big\{n\gqq n_0~|~\frac{|X_n|}{n} \gqq \e_n(\alpha, \eta, p)\Big\}} 
\]
Then for any $0\lqq \ti p  <  \alpha-3$ and $\sup_{n\in \NN} \EE[|\Delta X_n|^p] <\infty$ we have a constant $C>0$ such that  
\[
\sum_{n=1}^\infty n^{\ti p} \sum_{m=n}^\infty  \PP\Big(\frac{|X_n|}{n} > \e_n(\alpha, \eta, p)\Big)
\lqq C (\alpha-1)\zeta(\alpha -2-\tilde p, n_0), 
\]
we have 
\[
\frac{X_n}{n} \cdot \e_n^{-1}(\alpha, \eta, p)\ra 0\qquad \PP\mbox{-a.s.},  
\]
and 
\[
\EE[\oO_{\epsilon, n_0}^{1+\ti p}] {\lqq  \EE[\m_{\epsilon, n_0}^{1+\ti p}] \lqq } 
C (\alpha-1)\zeta(\alpha -2-\tilde p, n_0).
\]
In particular, we have 
\[
\PP(\oO_{\epsilon, n_0}\gqq k){\lqq \PP(\m_{\epsilon, n_0}\gqq k)}\lqq k^{-(\ti p+1)} \cdot C (\alpha-1)\zeta(\alpha -2-\tilde p, n_0)
\qquad \mbox{ for }k\gqq 1. 
\]
\end{cor}

\medskip 
\begin{proof}[\textbf{Proof of Corollary~\ref{cor:BKmart}:} ] The proof is similar to 
the proof of Corollary~\ref{cor:BK}. 
\end{proof}
\noindent Corollary~\ref{cor:BKmart} is suitable for a quantification for the tradeoff of a.s.~convergent M-estimators in Theorem~\ref{thm:momBK} in  Subsection~\ref{ss:M-estimator}.\medskip

Baum-Katz estimates for martingale differences in the infinite dimensional setting have been shown by \cite{Als90,DedeckerMerlevede07,Giraudo18, Hao13, HaoLiu14}, among others. We state a result from \cite[Theorem 2.4 (3)]{Giraudo18}.

\begin{thm}[\textbf{Baum-Katz estimate for Banach spaces}]\label{thm:BKInf}\hfill\\
Consider a martingale difference sequence $(\Delta X_n)_{n\in \NN}$ in a $2$-smooth Banach space $B$, let $p> 2$ and $\alpha\in(\tfrac{1}{2},1]$. Assume that $(\|\Delta X_n\|)_{n\in \NN}$ is identically distributed and $\EE[\|\Delta X_1\|^p]<\infty$.
Then there is a constant $C(p,B)$ such that 
\begin{align*}
\sum_{n=1}^\infty n^{p(\alpha-\frac{1}{2})-1}\PP\bigg(\frac{\max_{1\leq k\leq n} \|X_k\|}{n}>\eta n^{\alpha-1}\bigg)<C(p,B)\frac{\EE[\|X_1\|]^p}{\eta}\quad\text{for all }\eta>0.
\end{align*}
\end{thm}

The respective MDF quantification reads as follows.
\begin{cor}
With the assumptions of Theorem \ref{thm:BKInf} with initial index $n_0\in \NN_0$, assume $\alpha\in (\tfrac{1}{2},1]$, $p>\frac{2}{\alpha-\frac{1}{2}}$ and consider $\epsilon=(\varepsilon_n)_{n\in\NN}$ with $\e_n(\alpha, \eta, p) := \eta n^{\alpha-1}$.

Then, for any $0<\tilde{p}<p(\alpha-\tfrac{1}{2})-2$, we have the following:
\begin{enumerate}
\item There is a constant $C>0$ such that  
\[
\sum_{n=n_0}^\infty n^{\ti p} \sum_{m=n}^\infty  \PP\Big(\frac{|X_n|}{n} > \e_n(\alpha, \eta, p)\Big)
\lqq C (p(\alpha-\tfrac{1}{2})-1)\zeta(p(\alpha-\tfrac{1}{2}) -2-\tilde p, n_0).
\]
\item We have the convergence 
\[
\max_{n_0\leq k\leq n}\frac{\|X_k\|}{n} \cdot \e_n^{-1}(\alpha, \eta, p)\ra 0\qquad \PP\mbox{-a.s.}.
\]
\item The moments \[
\EE[\oO_{\epsilon, n_0}^{1+\ti p}] {\lqq  \EE[\m_{\epsilon, n_0}^{1+\ti p}] \lqq } 
C (p(\alpha-\tfrac{1}{2})-1)\zeta(p(\alpha-\tfrac{1}{2}) -2-\tilde p, n_0)\quad\text{are finite.}
\]
\item For any $k\gqq 1$ we have \[
\PP(\oO_{\epsilon, n_0}\gqq k){\lqq \PP(\m_{\epsilon, n_0}\gqq k)}\lqq k^{-(\ti p+1)} \cdot C (p(\alpha-\tfrac{1}{2})-1)\zeta(p(\alpha-\tfrac{1}{2}) -2-\tilde p, n_0)
\]
\end{enumerate}
\end{cor}

\begin{proof}
The proof is again similar to the one of Corollary \ref{cor:BK}.
\end{proof}

\bigskip
\subsection{\textbf{The tradeoff in a strong law for MDs with uniformly bounded exponential moments}}\label{ss:LesigneVolny}
\hfill\\\vspace*{0.2em}

\noindent For the exponential case we cite the following large deviations type result for martingales. 

\begin{thm}[{\cite[Theorem 3.2]{lesigneVolny00}}]\label{lem:exponentialLdp}\hfill\\
Let $X=(X_n)_{n\in \NN_0}$ be a martingale with $X_0=0$ with respect to a filtration $\FF$. Set $\Delta X_n:=X_n-X_{n-1}$ for $n\gqq 1$. Assume the existence of some $K>0$ and $\lambda>0$ such that $k\gqq 1$, $\EE\Big[e^{\lambda |\Delta X_k|}\Big]<K$. 
Then for any positive number $\delta \in (0,1)$ there exists a positive integer $n_0\in \NN$ such that for all $n\gqq n_0$  
\begin{align}\label{e:exponentialLdp}
\PP\left(\left|\frac{X_n}{n}\right|>\e\right)\lqq e^{-\tfrac{1-\delta}{2}\lambda^\frac{2}{3} \e^\frac{2}{3} n^\frac{1}{3}}.
\end{align}
In particular, $\lim\limits_{n\ra\infty}\frac{X_n}{n}= 0$ a.s.~ 
\end{thm}

\begin{rem}
Note that in this generality, Theorem \ref{lem:exponentialLdp} is the best one can achieve. In \cite{lesigneVolny00}, the authors construct a martingale $X$ in the context of ergodic dynamical systems such that $\EE[e^{|\Delta X_k|}]<\infty$ for all $k$ but still there is a constant $c>0$ such that
$$\PP\bigg(\bigg|\frac{X_n}{n}\bigg|>1\bigg)>e^{-cn^{1/3}}$$
for infinitely many $n$.
\end{rem}
\noindent The context of ergodic dynamical systems is another source for martingale differences, see \cite{lesigneVolny00}, \cite{Volny93}, \cite{Volny89}, where the following examples emerge:

\begin{exm}
Let $(\Omega,\cF,\PP)$ be a probability space and let $T\colon \Omega\to\Omega$ be a bijective, bimeasurable, measure preserving mapping. Assume that $\iI$ is the $\sigma$-algebra of all sets $A$ such that $TA=A$. Assume that for all $A\in \iI$ we have $\PP(A)\in\{0,1\}$ (i.e.~$\PP$ is ergodic). Let $\mM$ be a $T$-invariant $\sigma$-algebra, that is $\mM\subseteq T^{-1}\mM$. Let now $m=(m_k)_{k\gqq 1}$ be a sequence of stationary (i.e.~identically distributed) martingale differences with respect to the filtration $(T^{-n}\mM)_{n\gqq 0}$. In \cite{Volny89} it is shown that then, $m$ is of the form 
$$m_k=\EE[f|T^{k-i}\mM]-\EE[f|T^{k-i+1}\mM],\quad k\gqq 0,$$
for some $i\gqq 0$ and $f\in L^1$. 

Naturally, higher integrabilities such as $L^p$ or exponential integrability for $m$ are given by properties of the function $f$ which then brings us in the situations of Theorem \ref{thm:martSLLNq}, its subsequent Remark  \ref{rem:SLLnq} and Theorem \ref{lem:exponentialLdp}: Indeed, if $f\in L^p, p\gqq 1$, it follows
\begin{align*}
\EE[|m_k|^p]^\frac{1}{p}=\EE\left[\left|\EE[f|T^{k-i}\mM]-\EE[f|T^{k-i+1}\mM]\right|^p\right]^\frac{1}{p}\lqq 2\EE[|f|^p]^\frac{1}{p}<\infty,
\end{align*}
which is just Minkowski's inequality. For exponential moments, assume that there is $\lambda>0$ such that $\EE[e^{2\lambda|f|}]<\infty.$ Then we have 
\begin{align*}
\EE[e^{\lambda|m_k|}]&=\EE\bigg[e^{\lambda\left|\EE[f|T^{k-i}\mM]-\EE[f|T^{k-i+1}\mM]\right|}\bigg]\lqq \EE\bigg[e^{\lambda\big(\EE[|f||T^{k-i}\mM]+\EE[|f||T^{k-i+1}\mM]\big)}\bigg]\\
&\lqq \frac{1}{2}\left(\EE\bigg[e^{2\lambda\big(\EE[|f||T^{k-i}\mM]}\bigg]+\EE\bigg[e^{2\lambda\big(\EE[|f||T^{k-i+1}\mM]}\bigg]\right)\lqq \EE[e^{2\lambda|f|}]<\infty.
\end{align*}
Here, the first estimate in the second line is Young's inequality, the second one is Jensen's inequality for conditional expectations (with subsequent use of the tower property).
\end{exm}

\begin{cor}\label{thm:martExpSLLN}
Under the assumptions of Theorem~\ref{lem:exponentialLdp}, we get the following tradeoff for $\lim\limits_{n\to\infty}\frac{X_n}{n}= 0, \PP\text{-a.s}$. 
\begin{enumerate}
 \item For any $\e>0$, $n_0\in \NN$, with $\mathcal{O}_{\e, n_0}:=\sum\limits_{n=n_0}^\infty \bI\{|\frac{X_n}{n}|>\e\}$, 
 $\m_{\e, n_0}:=\max\{n\gqq n_0~|~ |\frac{X_n}{n}|>\e\}$ 
 and $K, \delta>0$, we get that for all $0<p<\tfrac{1}{2}(1-\delta)\lambda^\frac{2}{3}\e^\frac{2}{3}$ we obtain the moment estimate 
\begin{align*}
\EE\Big[e^{p\oO_{\e, n_0}^\frac{1}{3}}\Big]\lqq \EE\Big[e^{p\m_{\e, n_0}^\frac{1}{3}}\Big]\lqq K(\lambda, \delta, n_0, p) <\infty,
\end{align*}
where $K(\lambda, \delta, n_0, p)$ is defined in \eqref{e:Weibullkonstante}, 
and by Example~\ref{ex:Weibull} there are positive constants $d, D>0$ such that  
such that we obtain 
\begin{align*}
\PP(\oO_{\e, n_0}\gqq k){\lqq \PP(\m_{\e, n_0}\gqq k)}\lqq (d + D(k-1)^{2-\alpha}) e^{-p (k-1)^{\frac{1}{3}}}
\mbox{ for } k\gqq 2. 
\end{align*}
\item In addition, we have 
\begin{align*}
\mathrm{d}_{\mathrm{KF}}\Big(\frac{X_n}{n}, 0\Big) \lqq \frac{2^\frac{5}{6}}{3^\frac{1}{3}}  \frac{ \Big(W(\frac{1-\delta}{3}\lambda^\frac{2}{3} n^\frac{1}{3})\Big)^\frac{3}{2}}{(1-\delta)^\frac{3}{2} \lambda n^\frac{1}{2} }, 
\end{align*}
where $W$ is Lambert's $W$ function. 
\item Moreover, for any $\theta>0$ and $\epsilon = (\e_n)_{n\gqq n_0}$, $\e_n := \frac{\ln^3(n+1)}{\sqrt{n}} \frac{2(1+\theta)}{(1-\delta) \lambda^\frac{2}{3}}$ 
we have 
\begin{equation}\label{e:fastsicher}
\limsup_{n\ra\infty} \frac{X_n}{n} \cdot \e_n^{-1} \lqq 1 \qquad \PP\mbox{-a.s.} 
\end{equation}
and the respective overlap statistics $\oO_{\epsilon, n_0}$ satisfies 
\begin{equation}\label{e:fastsicheroverlap}
\EE[\oO_{\epsilon, n_0}] \lqq \sum_{n=n_0}^\infty \frac{1}{n \ln^{1+\theta}(n+1)}. 
\end{equation}
\end{enumerate}
\end{cor}
\noindent Corollary~\ref{thm:martExpSLLN} is used in Theorem~\ref{thm:momUExp} in order to quantify the a.s.~convergence of M-estimators in Subsection~\ref{ss:M-estimator}. 

\begin{proof}
For $\delta, \e, K, \lambda, n_0$ as in Theorem \ref{lem:exponentialLdp} such that 
\[
\PP\left(\left|\frac{X_n}{n}\right|>\e\right) \lqq e^{-\tfrac{1-\delta}{2}\lambda^\frac{2}{3} \e^\frac{2}{3} n^\frac{1}{3}}, 
\]
we apply Lemma~\ref{lem:Quant BC for e_n} and Example~\ref{ex:Weibull}, which yields the desired result. 
The Ky Fan rate is obtained by solving  
\begin{align*}
\e = e^{-\tfrac{1-\delta}{2}\lambda^\frac{2}{3} \e^\frac{2}{3} n^\frac{1}{3}} 
\end{align*}
for $\e = \e_n$ which yields with the help of Lambert's $W$-function 
\begin{align*}
\e &= \e_n = 
\frac{\big(\frac{2}{3}\big)^\frac{1}{3} \Big(W(\frac{1-\delta}{3}\lambda^\frac{2}{3} n^\frac{1}{3})\Big)^\frac{3}{2}}{2 \big(\frac{1-\delta}{2}\lambda^\frac{2}{3} n^\frac{1}{3}\big)^\frac{3}{2} } 
= \frac{2^\frac{5}{6}}{3^\frac{1}{3}}  \frac{ \Big(W(\frac{1-\delta}{3}\lambda^\frac{2}{3} n^\frac{1}{3})\Big)^\frac{3}{2}}{(1-\delta)^\frac{3}{2} \lambda n^\frac{1}{2} }.
\end{align*}
For $\e_n := \frac{\ln^3(n+1)}{\sqrt{n}} \frac{2(1+\theta)}{(1-\delta) \lambda^\frac{2}{3}}$ 
we have that the right-hand side is of order $(n+1)^{-(1+\theta)}$, which is barely summable, and the classical first Borel-Cantelli lemma yields \eqref{e:fastsicher} and \eqref{e:fastsicheroverlap} by Lemma~\ref{lem:Quant BC for e_n}. 
\end{proof}

\begin{exm}[\textbf{Strong law tradeoff with other bounds}]\label{rem:martbdSLLN}
If we consider a martingale $(X_n)_{n\gqq 0}$ such that the sequence of centered martingale differences $(\Delta X_n)_{n\in \NN}$ is uniformly bounded a.s.~by, say, a positive constant $a>0$, then an application of the Azuma-Hoeffding inequality (Theorem \ref{thm:azuma}) yields
\begin{align*}
\PP\left(\left|\frac{X_n}{n}\right|>\e\right)\lqq 2\exp\Big(-\frac{n\e^2}{2a^2}\Big), \qquad n\in \NN. 
\end{align*}
Hence Lemma~\ref{lem:Quant BC for e_n} combined with Example~\ref{ex:exp} yields 
\begin{align*}
\EE\left[\exp\left(\frac{\e^2 p}{2a^2}\oO_{n_0}\right)\right] {\lqq \EE\left[\exp\left(\frac{\e^2 p}{2a^2}\m_{n_0}\right)\right]}\lqq 
 1+  \frac{2e^{-\frac{\e^2(n_0-1)}{2a^2}}}{1-e^{-\frac{\e^2(1-p)}{2a^2}}}
\end{align*}
and for all $k\gqq 1$ 
\begin{align*}
\PP(\oO_{n_0}\gqq k) 
&{\lqq \PP(\m_{n_0}\gqq k) 
}
\lqq 2e^{\frac{9}{8}} \cdot \left[k\left(2 e^{-\frac{\e^2(n_0-1)}{2a^2}}+1\right)+1\right] \cdot e^{-k\frac{\e^2}{2a^2}}.   
\end{align*}
\end{exm}

\begin{rem}\label{rem:LDPvergleich}
If we consider the situation of $X_n = \sum_{i=1}^n \Delta_i$ for a centered i.i.d.~sequence $(\Delta_i)_{i\in \NN}$ with exponential moments, we obtain by Cram\'er's theorem a large deviations principle (LDP), with the upper bound  
\begin{align*}
\PP\left(\left|\frac{X_n}{n}\right|>\e\right)\lqq \exp\Big(-n \inf_{|y|>\e} \Lambda_{\Delta_1}^*(y)\Big),
\end{align*}
where the exponent is given by the good rate function 
\[
\Lambda_{\Delta_1}^*(y) = \inf_{t\in \RR} ty - \Lambda_{\Delta_1}(t), \qquad \Lambda_{\Delta_1}(t) = \ln(\EE[e^{t \Delta_1}]).  
\]
For examples and comments on this setting we refer to \cite[Subsection 3.2.2]{EstraHoeg22}. 
In Theorem~\ref{thm:momGE} of Subsection~\ref{ss:M-estimator} a slight generalization (the so-called G\"artner-Ellis theorem) of this observation is used to quantify the a.s.~convergence of M-estimators in presence of an LDP. 
\end{rem}

\begin{exm}[Closed martingales with exponentially integrable limit]\hfill  
\begin{enumerate}
 \item Let $X$ be a centered random variable such that there is $\lambda>0$ with $\EE\exp(2\lambda|X|)<\infty$, and let $\FF = (\mathcal{F}_n)_{n\gqq 0}$ be a given filtration. Then the sequence given by $X_n:=\EE\left[X\middle|\cF_n\right]$ forms a martingale. We have for the differences that for $k\gqq 1$,
\begin{align*}
\EE\Big[\exp(\lambda|dX_k|)\Big]
&=\EE\Big[\exp(\lambda|dX_k|)\Big]
=\EE\Big[\exp\left(\lambda\left|\EE\left[X\middle|\cF_k\right]-\EE\left[X\middle|\cF_{k-1}\right]\right|\right)\Big]\\
&\lqq \EE\Big[\exp\left(\lambda\EE\left[\left|X\right|\middle|\cF_k\right]+\EE\left[\left|X\right|\middle|\cF_{k-1}\right]\right)\Big],
\end{align*}
which, by Young's inequality, is smaller than
\begin{align*}
\frac{1}{2}\EE\Big[\exp\left(2\lambda\EE\left[\left|X\right|\middle|\cF_k\right]\right)\Big]+\frac{1}{2}\EE \Big[\exp\left(2\lambda\EE\left[\left|X\right|\middle|\cF_{k-1}\right]\right)\Big],
\end{align*}
which can in turn be estimated using the conditional Jensen inequality via
\begin{align*}
\frac{1}{2}\EE\Big[\exp\left(2\lambda\EE\left[\left|X\right|\middle|\cF_k\right]\right)\Big]+\frac{1}{2}\EE \Big[\exp\left(2\lambda\EE\left[\left|X\right|\middle|\cF_{k-1}\right]\right)\Big]\lqq \EE\Big[\exp(2\lambda |X|)\Big]=K<\infty.
\end{align*}
Hence, the martingale $X$ satisfies the assumptions of Lemma \ref{lem:exponentialLdp} and Theorem \ref{thm:martExpSLLN}, and we obtain the Weibull-type moments and decay rates for the overlap and the modulus. See Example~\ref{ex:Weibull}. 
\item Let $X$ be a centered, random variable, bounded by $a>0$, and let $(\mathcal{F}_n)_{n\gqq 0}$ be a filtration. Then again, the sequence given by $X_n:=\EE\left[X\middle|\cF_n\right]$ forms a martingale w.r.t.~$\FF$. By the boundedness condition, the assertion of Example~\ref{rem:martbdSLLN} holds true, and yields exponential moments and decay rates of the overlap statistic $\oO_\e$ and $\m_\e$, for fixed $\e>0$. See Example~\ref{ex:exp}. 
\end{enumerate}
\end{exm}

\bigskip 

\section{\textbf{Applications}}\label{s:applications}
\subsection{\textbf{The tradeoff in multicolor P\'olya urn models}}\label{ss:multicolor}\hfill\\

\noindent In Example \ref{ex:Polya} we presented the exponential MDF convergence for the two-color P\'olya's urn. However, a natural generalization involves introducing a broader range of types or colors for the balls, each with its own replacement rules. This is known as the \textit{multicolored P\'olya's urn} model, which is often found when representing a wide range of natural phenomena. Thus, in this subsection we define this process for finite colors and establish its limiting distribution. Moreover, we delve into how to interpret the process' parameters when applying it in the contexts of machine learning, genetics and biology. For a comprehensive survey about the applications of urn-like models, see \cite[Chapter 5]{JK77}. Our bounds provide a sharp guide to evaluate model performance, enabling researchers to assess its alignment with some expected parameters as outlined in the mock test below.

\begin{defn}
    A \textbf{generalized multicolor P\'olya urn process} is given by a $d$-dimensional Markov chain $(X_n)_{n\in \NN_0}$, $X_n := (X_{n,1}, \dots, X_{n,d})$, with transition probabilities 
    $$\PP(X_{n+1}= X_n + e_{i}R | X_1, \dots, X_n) = \frac{X_{n{,}i}}{\sum_{k=1}^{d} X_{n, k}}, \quad i=1, \dots, d,$$
    where $e_i=(0 \dots 1 \dots 0)$ is the row unit vector with $1$ in the {$i$-th} coordinate and $R=(r_{i,j})$ is a $d\times d$ deterministic matrix with integer coefficients, called replacement matrix.
\end{defn}
Intuitively, in the random vector $X_{n}$, the $i$-th component $X_{n,i}$ represents the number of balls in the urn with color $i$ at the $n$-th step of the process. Moreover, at each step a ball is drawn from the urn at random, its color is recorded and then returned along with $r_{i,j}$ additional balls of color $j$, $j=1, \dots, d$. Note that any negative $r_{i,j}$ represents balls being taken away from the urn. 
Problems may arise when there are negative replacements, which might leave some colors to run out of balls. Therefore, the author of \cite{Gouet97} introduces the notion of tenable generalized multicolor P\'olya urn, which guarantee{s} the long-term well-definedness {of} the process. 

A \textbf{tenable} generalized multicolor P\'olya urn process (TGMPU) is a generalized multicolor P\'olya urn 
with the following additional hypotheses on $R$. 
\begin{enumerate}
 \item $r_{ij}\gqq 0$ for all $i\neq 0$.\label{nr:i}
 \item $\sum_{j=1}^d r_{ij} = s\gqq 0$ for all $i=1, \dots,d$. 
 \item $r_{ii} < 0$ implies that $r_{ii}$ is a divisor (modulo sign) of $r_{ki} =1, \dots, d$.\label{nr:iii}
\end{enumerate}
The asymptotic behavior of such types of multicolored P\'olya urn has been extensively investigated using the martingale version of the Borel-Cantelli lemma \cite{BP85,Gouet89,Gouet93,Friedman49,NH82}. 
For the replacement matrix $R$, there is a natural notion of connected components and irreducibility of the submatrices of $R$, which underpins the fundamental theory by Seneta \cite{Se06} and allows to give a normal form of $R$. The long-term survival of $X_n / \sum_{i=1}^d X_{i, n}$ is known to be dominated by the irreducible components of $R$, whose dominant eigenvalue equals, precisely, to the row sum $s$. 
Those irreducible components are called \textbf{supercolors}.  
More precisely in \cite[Theorem 3.1]{Gouet97} it is shown that the vector $X_n / \sum_{i=1}^d X_{n, i}$ converges a.s.~to some lacunary random row vector $X_{\infty}$ which is distributed according to a Dirichlet mixture of the dominant eigenvectors (which are nonnegative and sum up to $1$) of the supercolors. All other entries, which {correspond} to transient states and hence irreducible components with leading eigenvalues $\tau < s$, are equal to $0$ in $X_\infty$. 
The parameters 
depend only on the initial total number of balls in each of the supercolors, 
the row sum $s$ of $R$ and the initial total number of balls in {the urn}. 

Let $(Y_n)_{n\gqq 1}$ be the sequence of independent draws from the urn, given by the $d$-dimensional vectors $Y_n=(Y_{n,1}, \dots, Y_{n,d})$ representing the number of balls of each color at stage $n$ chosen according to the replacement matrix $R$ satisfying \eqref{nr:i}-\eqref{nr:iii}, and $T_n = \sum_{i=1}^{d} Y_{n,i}$..
In \cite[Proposition 4.1]{Gouet97}(i) it is shown{, that up to a reordering of states,} the $\RR^d$ valued process $X_n := Y_n / T_n$ has the following shape: There are $r$ supercolors for some $r\in \{1, \dots, d\}$ and {we denote the number of colors composing the $i$-th supercolor by $d_i$.} Hence 
\[
X_n = (M_n, S_n), \qquad n\in\NN,
\]
where $(M_n)_{n\in \NN_0}$ is a $\RR^{\sum_{i=1}^{r+1}d_i}$-valued nonnegative martingale which converges a.s.~to a random vector 
\[
M_\infty \sim \sum_{i=1}^{r+1} X_{\infty, i}\cdot u_i, 
\]
where $u_i \in \RR^r$ consists of $0$ up to the $i$-th entry, which is given by the dominant $d_i$-
eigenvector of the $i$-th supercolor. The random vector $M_\infty$ has the following density 
\begin{equation}\label{def:Dirichlet}
f_{M_\infty}(y_1, \dots, y_{r+1}) = \Gamma\left(\frac{1}{s} \sum_{i=1}^{\sum_{j=1}^r d_j} X_{0, i}\right) \prod_{i=1}^{r+1} \frac{y_{{i}}^{\frac{1}{s} \sum_{j= d_{i}+1}^{d_{i+1}} X_{0, j}}}{\Gamma(\frac{1}{s} \sum_{j= d_{i}+1}^{d_{i+1}} X_{0, j})}.
\end{equation}
The vector $(S_n)_{n\in \NN_0}$ is a nonnegative supermartingale which tends to $0$, so  
\[
X_\infty = (M_\infty, 0).  
\]
Furthermore, by the hypotheses for tenable replacement matrices we have 
\[
\sup_{i=1, \dots, d} |Y_{n, i} - Y_{n-1, i}|\lqq \max_{i,j} |r_{ij}|:= C. 
\]
That is, the asymptotic proportions between the supercolors are random and Dirichlet distributed, 
while the asymptotic proportions within the supercolors are asymptotically deterministic. 
We quantify the a.s.~convergence result of $X_n = (M_n, {S_n}) \ra (M_\infty, 0)$ 
for the vector-valued martingale $(M_n)_{n\in \NN_0}$ and the nonnegative (componentwise) vector-valued supermartingale $(S_n)_{n\in \NN_0}$.  

For convenience of this article, we assume that the normal form of the matrix consists of a finite union of irreducible components{. That} is, the supermartingale $S_n\ra 0$ does not show up. Obviously, it can be studied with similar methods{. However,} this does require a closer look into the spectral structure of the normal form and the respective Doob-Meyer decomposition \cite{Rao69}. Under this assumption, we have that $X_n = M_n$ and our setting falls under the hypotheses of Corollary~\ref{cor:Azumapol}.  
Note that similarly to \eqref{e:Polyasakrament} the increments satisfy 
\[
|X_n - X_{n-1}| =  \frac{1}{T_n} |Y_n - Y_{n-1}| \lqq  \frac{C}{T_n} = \frac{C}{\sum_{k=1}^d Y_{0, k} + ns}\lqq {\Big(}\frac{C+1}{s}{\Big)} \frac{1}{n} =:c_n, \qquad n\gqq 1,  
\]
and the increments are almost surely square summable since for all $n\gqq 2$ we have 
\[
r(n) = \sum_{k=n+1}^\infty c_k^2 \lqq \Big(\frac{C+1}{s}\Big)^2 \frac{1}{n}.
\]
Corollary~\ref{cor:Azumapol} implies for each component $i$ of $(M_n)_{n\in \NN_0}$ the 
following tradeoff: for all $a_i = (a_{n,i})_{n\in \NN_0}$ positive, nondecreasing and 
$\epsilon_i = (\e_{n,i})_{n\in \NN}$ positive, nonincreasing 
such that {for}
\begin{align}\label{e:Polyaconstant}
K(a_i, \epsilon_i) := \sum_{n=n_0}^\infty a_{n,i} \sum_{m=n}^\infty \exp\Big(-\frac{\e_{m,i}^2 (m+1)s^2}{ (C+1)^2}  \Big) < \infty {,}
\end{align}
we have 
\[
\limsup_{n\ra\infty} |M_{n, i}-M_{\infty, i}|\cdot  \e_{n,i}^{-1}\lqq 1,
\qquad \PP\mbox{-a.s.},
\]
and 
\begin{align*}
\EE[\sS_{a, n_0,i}(\oO_{\epsilon, n_0, i})] {\lqq \EE[\sS_{a, n_0,i}(\m_{\epsilon, n_0, i})]} \lqq K(a_i, \epsilon_i), 
\end{align*}
with 
\begin{equation}\label{e:Polyadecay}
\PP(\oO_{\epsilon, n_0, i}\gqq k){\lqq \PP(\m_{\epsilon, n_0, i}\gqq k)}\lqq \sS_{a, n_0,i}^{-1}(k)  K(a_i, \epsilon_i). 
\end{equation}

\bigskip 
\begin{rem}
\begin{enumerate}
 \item Note that our quantification enables us to assert the likelihood of empirical estimates for the parameters constituting $\theta\in \RR^{d}$ of $X_\infty$, and guide the decision of when to halt sampling, while ensuring that the probability of future error incidences for a desired error tolerance falls below a given confidence level.
 
 \item We also refer to \cite{Fr17}, which derives a large deviations principle for 
 multicolor P\'olya urns, which also allows for similar (asymptotic) exponential quantifications. 
\end{enumerate}

\end{rem}

\bigskip 

{
\noindent \textbf{Mock test for model refutation: } Before we review the literature of P\'olya urn models in different contexts of applications
in the subsequent sub-subsections. Let us illustrate the utility of our cutoff convergence in form of a mock test, which can be implemented and certainly refined in many concrete situations mentioned below. In all those models with an embedded underlying P\'olya urn model we may take advantage of our precise knowledge of the tradeoff between the asymptotic rates of convergence $\epsilon = (\e_n)_{n\in \NN}$ and the decay of the tails $\PP(\oO_\epsilon\gqq k)$, $k\in \NN$,  
of the corresponding mean failure count $\oO_\epsilon$ in the sense of \eqref{def:Oe} in order to 
\textit{refute} models at a given level of confidence $\alpha\in (0,1)$. 
Given $\alpha \in (0,1)$ fixed we define null hypothesis $H_0$ v. the alternative $H_1$ by  
\[
H_0: \mbox{ the data stem from a known P\'olya urn model}\qquad \mbox{v. }\qquad H_1:  \mbox{ else}.
\]
For $i$ and $\e>$ fixed consider the theoretical failure count statistics 
$$\oO = \oO_{\e, n_0, i} = \sum_{n=n_0}^\infty \ind\{|M_{n} - M_{\infty}|>\e\}$$ 
and the cutoff version for $N$ data starting in $n_0$ 
$$\uU_{N} := \sum_{n=n_0}^N \ind\{|M_{n} - M_{N}|>2\e\}.$$ 
First of all note that 
$\{|M_n-M_N|> 2\e\} \subseteq \{|M_N- M_\infty| > \e\} \cup \{|M_n-M_\infty| > \e\}$ and 
\begin{align*}
\uU_{N} 
&= \sum_{n=n_0}^N \ind\{|M_{n} - M_{N}|> 2\e\} \lqq \sum_{n=n_0}^N \ind\{|M_{n} - M_{\infty}|>\e\}\cup \{|M_{N} - M_{\infty}|> \e\} \\
&\lqq \oO +\sum_{n=n_0}^N \ind\{|M_{n} - M_{\infty}|\lqq \e, |M_{N} - M_{\infty}|> \e\} \\
&= \oO +\sum_{n=n_0}^{N-1} \ind\{|M_{n} - M_{\infty}|\lqq \e, |M_{N} - M_{\infty}|> \e\} \lqq \oO +\eE_N,
\end{align*}
where $\eE_N := (N-n_0) \ind\{|M_{N} - M_{\infty}|> \e\}$. For some $p$ and $\e>0$ fixed the H\"older inequality implies that 
\begin{align*}
\EE\Big[e^{p(\oO +\eE_N)}\Big] 
&\lqq \EE\Big[e^{2p (N-n_0) \ind\{|M_{N} - M_{\infty}|> \e\}}\Big]^\frac{1}{2}  \EE\Big[e^{2p \oO}\Big]^\frac{1}{2}. 
\end{align*}
We calculate 
\begin{align*}
\EE\Big[e^{2p \eE_N}\Big] 
&= \EE\Big[e^{2p(N-n_0)} \ind\{|M_{N} - M_{\infty}|> \e\} + \ind\{|M_{N} - M_{\infty}|\lqq \e\}\Big]\\
&= e^{2(N-n_0)p} \PP(|M_{N} - M_{\infty}|> \e)+ \PP(|M_{N} - M_{\infty}|\lqq \e), 
\end{align*}
and hence by \eqref{e:Polyaconstant} we have for $p>0$ small enough 
\begin{align*}
\EE\Big[e^{p(\oO +\eE_N)}\Big] 
&\lqq \Big(e^{(N-n_0)p - \frac{s^2\e^2}{2(C+1)^2} (N+1)}+ 1\Big) \EE\Big[e^{2p \oO}\Big]^\frac{1}{2}
\lqq 2\EE\Big[e^{2p \oO}\Big]^\frac{1}{2}.
\end{align*}
Therefore for $N\in \NN$ sufficiently large and $p>0$ such that $0 < p \lqq \frac{s^2}{2(C+1)^2} \e^2$
we have by Example~\ref{ex:exp} 
\begin{align*}
\PP(\uU_N\gqq k) &\lqq \PP(\oO+ \eE_N \gqq k) \lqq e^{-pk} 2 \EE\Big[e^{2p \oO}\Big]^\frac{1}{2} 
\lqq 2 e^{-pk} \sqrt{1+  \frac{ b^{n_0-1}}{1-b^{1-2p}}},  
\end{align*}
where $b = e^{-\frac{s^2}{(C+1)^2} \e^2}$. Note that the right-hand side can still be optimized 
as in \eqref{e:exptailsub} of Example~\ref{ex:exp}. \\

\noindent \textbf{Run the test:}\\ 
\begin{enumerate}
 \item For given model parameters $\e, p, n_0, N$ and level of confidence $\alpha \in (0,1)$ 
 we calculate 
 $$k^*_\alpha := \mbox{argmax}_k\{2e^{-pk} \sqrt{1+  \frac{cb^{n_0-1}}{1-b^{1-2p}}}\lqq \alpha\}.$$  
 \item For given data $y_{i, 1}, \dots y_{i, N}$ count  
\[
U_N:= \sum_{n=1}^N \{|y_j- y_N|>\e\}.  
\]
\end{enumerate}
If $U_N > k^*_\alpha$ refute $H_0$.\\

\bigskip 

\subsubsection{\textbf{Applications of P\'olya's urn models in machine learning}}\hfill\\

\noindent P\'olya's multicolor urn models can be used to build random recursive trees (RRT) and preferential attachment trees (PAT), and thus facilitate the modeling of complex networks, see \cite{MAILLER20} for further details. In these structures, the nodes in the network are akin to the balls in the urn, and their out-degrees serve as their colors.\footnote{Some open-ended questions arise when allowing the nodes to have infinite out-degree. However, a common practical resolution for such cases is to consider nodes above a predefined threshold $\tau$ of connections as the same color.} The process begins with two nodes linked by an edge at time zero. Subsequently, a new node is introduced to the tree at each time step, and an edge is established between this new node and an existing one. How the existing node is chosen dictates the type of tree formed. In the RRT variant, the selection is random, whereas in the PAT, the choice is influenced by the nodes' degrees, with higher-degree nodes having a higher probability of selection.

\noindent Both RRT and PAT are pivotal data structures. RRTs facilitate the inference of missing values by capitalizing on the inherent relationships within the tree structure, and can enhance a model's predictiveness, efficiency, and interpretability during the stage of feature selection. The hierarchical nature of a tree structure is ideal for unsupervised tasks like hierarchical clustering, enabling proximity measurements between data points. For example, in the context of natural language processing, these RRTs have been used for stemma construction in philology to reconstruct and analyze the similarities between different versions of a text \cite{NH82}. On the other hand, PATs excel in developing recommender systems by leveraging the `rich get richer' principle, where higher-degree nodes gather more new links and mirror the tendency for popular items to receive heightened recommendations \cite{LauKoo20}. For both scenarios, P\'olya's urn analysis contributes to robust insights about the nodes' out-degree distribution and the network's growth and evolution \cite{MAILLER20}. More applications in machine learning are given, for instance, in \cite{CNWDD17}. 

\subsubsection{\textbf{Applications of P\'olya's urn models in genetics, psychology and biology}}\hfill\\

\noindent P\'olya's urn model, applied in genetics and populations studies \cite{Ewens69}, envisions an urn filled with colored balls, symbolizing genetic traits. Balls are drawn, noted, and returned with more of the same color, mimicking reproduction and natural selection. The more drawings occur, the composition of colors in the urn evolves, reflecting the changing genetic or trait distribution over generations and illuminating biological dynamics succinctly.\\

\noindent Additionally, urn processes can be used to model learning curves. For this, we reference the Audley-Jonckheere urn process \cite{AudleyJonckheere56}, which is a special case of a two-color P\'olya's urn where the replacement rule is not limited to returning balls of the observed color. This process is employed in learning experiments, where participants respond to stimuli with successful or unsuccessful outcomes. The focus of the experiment lies in tracking the proportion of correct and incorrect responses, and predicting the number of errors preceding a specific sequence of successes. During each trial, a ball is drawn, observed, and returned alongside other balls. The proportion of balls returned after each draw reflects the reward and punishment system, which will determine the replacement matrix $R$. Note that for this system, $R$ can encode not just a net profit for correct responses but also the regret from choosing incorrectly in failed trials.

\noindent A similar scenario based on learning in animals involves asking whether or not ants can learn the shortest path between their colony and a food source based on the stigmergy phenomenon: “ants stimulate other ants by modifying the environment via pheromone trail updating”. In \cite{KMS221}, the authors present a probabilistic reinforcement-learning model that captures this behavior. In it, the nest $N$ and food source $F$ are two nodes within a finite graph, and the ants embark on successive random walks, stopping upon hitting the food source. Their paths are influenced by previous walks, as the ants deposit pheromones on each edge they cross. The process mirrors a P\'olya's urn, where the number of coloured balls in the urn is analogous to the pheromone levels on the graph's edges. The conjecture for this recent problem is that, when time grows large, almost all ants go from $N$ to $F$ through the shortest path, which has been shown for specific types of graphs and return patterns \cite{KMS222}.\\

More applications for multicolor generalized P\'olya urn models are found in resource allocations, computer memory management \cite{BP85, DS72, Fagin75}, computer imaging \cite{BBA99, SLZWJA17}, statistical physics \cite{Harkness70}, remote sensing \cite{JYYQ18}, and parallel computing \cite{TMJD19}, and the references therein.
}
\bigskip

\subsection{\textbf{The tradeoff for the Generalized Chinese Restaurant Process (GCRP)}}\label{ss:ChineseRestaurant}\hfill\\

\noindent In this subsection, we show how state-of-the-art results in machine learing can be further sharpened in a useful way. In \cite[Thm. 3.2]{OPR22} the authors show a non-asymptotic random concentration result for the GCRP. Recall that the GCRP generates a sequence of random partitions $\pP_n$ of $[n]:=\{1,\dots, n\}$ for $n\in \NN$. Their results study the case where the 
growth of maximal components in $\pP_n$ behaves like $n^\alpha$, $n\in \NN$ for a parameter $\alpha \in (0,1)$, with a particular interest in the concentration limits of the total number of components with size $k$ in each $\pP_n$, that is:
\[
N_n(k) := |\{A \in \pP_n : |A| = k\}|. 
\]

More precisely, the model is given as a 
Markov chain $\pP_1, \pP_2,\pP_3, \dots$, where, for each $n \in \NN$, $\pP_n$ is a partition of $[n]$ composed by $V_n := |\pP_n|$ disjoint parts $A_{i,n}$, $i = 1, \dots, V_n$. Then, the process will evolve following a ``Chinese restaurant'' metaphor. In it, $A_{i,n}$ are the tables occupied by customers $1$ to $n$ (who come in sequentially), $V_n$ represents the total number of occupied tables and $\pP_n$ describes the table arrangements, which follow that

\begin{enumerate}
    \item Customer 1 sits by herself (i.e. $\pP_1 = \{\{1\}\}$).
    \item Given $\pP_1, \dots, \pP_n$, $\pP_{n+1}$ is set up by choosing where to sit customer $n+1$. That is, all the other customers will remain in their previously assigned tables, while customer $n+1$ will sit either at an occupied table $A_{i,n}$ with probability
    $$ \PP(n+1 \in A_{i,n+1} ~|~ \pP_1, \dots, \pP_n) = \frac{|A_{i,n}| - \alpha}{n + \theta}, \quad \text{for } i = 1, \dots, V_{n-1} \text{ and } \alpha, \theta \in \RR,$$
    \noindent or, alternatively, sit at a new table by herself with probability
    $$ \PP(n+1 \in A_{n+1,n+1} ~|~ \pP_1, \dots, \pP_n) = \frac{\alpha V_n + \theta}{n + \theta}, \quad \text{for } \alpha, \theta \in \RR.$$ \label{eq: GCRP n+1 table probability}
\end{enumerate}
Note that for the first scenario in \eqref{eq: GCRP n+1 table probability} $V_{n+1} = V_{n}$, while for the latter, $V_{n+1} = V_{n} +1$. However, most of the results in \cite{OPR22} will be set up for the normalized version $V_n/\phi_n$, where

\[
\phi_n := \frac{\Gamma(1 + \theta)}{\Gamma(1+ \theta + \alpha)} \frac{\Gamma(n + \alpha + \theta)}{\Gamma(n+ \theta)}.
\]

In particular, this is because the limit $V_{*}:= \lim_{n \to \infty} V_n/\phi_n$ exists and is almost surely positive, with an explicit density. Furthermore, the authors proposed a quantification of the almost sure convergence for $V_n/\phi_n$, which we show to be fit and quantifiable within the framework of Lemma~\ref{lem:Quant BC for e_n}.

\begin{thm}
Consider a realization $(\pP_n)_{n\in \NN}$ of the GCRP with parameters $\alpha \in (0,1)$ and $\theta>-\alpha$. 
Then there exist constants $n_0 = n_0(\alpha, \theta)\in \NN$ and $C = C(\alpha, \theta)$ such that 
the following holds for all $n\gqq n_0$. For any nondecreasing positive sequence $A = (A_n)_{n\gqq n_0}$, and nonincreasing positive sequence $\epsilon = (\e_n)_{n\gqq n_0}>0$ we define 
\begin{align*}
k_{\epsilon, n} &:= \bigg\lceil \frac{\e_n n^{\frac{1}{2}\frac{\al}{\al+2}}}{\ln(n)^\frac{1}{\alpha+2}} \bigg\rceil,\\   
c(\alpha, \theta) &:= \frac{\alpha \Gamma(1+\theta)}{\Gamma(1-\alpha) \Gamma(1+\alpha +\theta)}>0 ,\quad \text{and}\\  
E_n(A, \epsilon) &:= \Big\{\forall k\in \{1, \dots, k_{\e_n, n}\}: \big|N_n(k) - c(\alpha, \theta) \frac{\Gamma(k-\alpha)}{\Gamma(k+1)} V_* n^\alpha\big|\lqq C \frac{\Gamma(k-\alpha)}{\Gamma(k+1)} n^\alpha \e_n^{\alpha+2}\Big(1+ \frac{A_n}{\ln(n)}\Big)\Big\}.
\end{align*}
Then we have 
\begin{equation}\label{e:Chineseexpodecay}
\PP(E_n^c)\lqq e^{-A_n}.  
\end{equation}
Under the additional condition that $e^{-A_n}$ is summable we have the following MDF tradeoff: 
For $\oO_A := \sum_{n = n_0}^\infty \ind(E_n^c)$ {and $\m_A(\omega) := \max\{n\gqq n_0~|~\omega \in E_n^c\}$} 
and any sequence $a = (a_n)_{n\gqq n_0}$ of nonnegative, nondecreasing weights such that 
\[
C_{a, A} := \sum_{n=n_0}^\infty a_n \sum_{m= n}^\infty e^{-A_n} <\infty, 
\]
we have that for $\sS_{a, 0}$ defined in \eqref{def:Sa} and calculated explicitly in Example~\ref{ex:exp} the tradeoff satisfies
\[
\limsup_{n\ra\infty} \sup_{k\in \{1, \dots, k_{\e_n, n}\}} \big|N_n(k) - c(\alpha, \theta) \frac{\Gamma(k-\alpha)}{\Gamma(k+1)} V_* n^\alpha\big|\cdot \bigg(C \frac{\Gamma(k-\alpha)}{\Gamma(k+1)} n^\alpha \e_n^{\alpha+2}\Big(1+ \frac{A_n}{\ln(n)}\Big)\bigg)^{-1} \lqq 1 \qquad \PP\mbox{-a.s.} 
\]
and 
\[
\EE[\sS_a(\oO_A)]{\lqq \EE[\sS_a(\m_A)]}\lqq C_{a, A}. 
\]
\end{thm}
The proof is a direct consequence of the exponential decay \eqref{e:Chineseexpodecay} and Lemma~\ref{lem:Quant BC for e_n}. 

\bigskip
\begin{rem}
Important particular cases which highlight the play between the asymptotic a.s.~error bound and the MDF 
statistics are the following. Due to the asymptotics 
\[
\frac{\Gamma(k_{\e_n, n}-\alpha)}{\Gamma(k_{\e_{n}, n}+1)} n^\alpha \qquad \mbox{ of order } \qquad  k_{\e_{n}, n}^{-(1+\alpha)} = 
\Big\lceil \frac{\e_n n^{\frac{\al}{(2\al+4)}}}{(\ln(n)^\frac{1}{\alpha+2}} \Big\rceil^{-(1+\alpha)}
\]
only sequences of $A$ with an asymptotic behavior of 
\[
A = O(\e_n^{\alpha+3} n^{\frac{\al}{(2\al+4)}} \ln(n)^\frac{\alpha+1}{\alpha+2}) \qquad 
\mbox{ and }\qquad \frac{1}{A} = O((2+\delta) \ln(n)), \text{ for } \delta>0,  
\]
are meaningful. For fixed $\e_n = \e>0$,  extremal cases for $A$ are given by: 
\begin{enumerate}
 \item $A_n = n^{\frac{\al}{(2\al+4)}}\ln(n)^\frac{\alpha+1}{\alpha+2}$ and $p\in (0,1)$, 
 yield by Example~\ref{ex:Weibull} a constant $K(p, \alpha)>0$ such that 
\[
\EE[e^{p (\oO_A-1)^{\frac{1}{2}\frac{\al}{\al+2}}}]{\lqq \EE[e^{p (\m_A-1)^{\frac{1}{2}\frac{\al}{\al+2}}}]}\lqq K(p, \al),
\]
and $d, D>0$ such that for $k\gqq 2$ 
\[
\PP(\oO_A\gqq k){\lqq \PP(\m_A\gqq k)} \lqq (d+ D(k-1)^{2- \frac{1}{2}\frac{\al}{\al+2}} e^{-p (k-1)^{\frac{1}{2}\frac{\al}{\al+2}}},
\]
while there is a constant $\tilde C>0$ such that 
\[
\limsup_{n\ra\infty} \sup_{k\in \{1, \dots, k_{\e, n}\}} \big|N_n(k) - c(\alpha, \theta) \frac{\Gamma(k-\alpha)}{\Gamma(k+1)} V_* n^\alpha\big|\lqq 
\tilde C C  \e \qquad \PP\mbox{-a.s.}  
\]

\item $A_n = (2+\delta) \ln(n)$. Then Example~\ref{ex:poly} yields 
\[
\limsup_{n\ra\infty} \sup_{k\in \{1, \dots, k_{\e_n, n}\}} \big|N_n(k) - c(\alpha, \theta) \frac{\Gamma(k-\alpha)}{\Gamma(k+1)} V_* n^\alpha\big|\cdot \bigg( \frac{n^{\frac{\al}{(2\al+4)}}}{(\ln(n)^\frac{1}{\alpha+2}} \bigg)^{(1+\alpha)}
 \lqq C (3+\delta)\e \qquad \PP\mbox{-a.s.} 
\]
and a constant $K(\alpha, \delta)$ such that 
\[
\EE[\oO_A^{1+\delta}]{\lqq \EE[\m_A^{1+\delta}]}\lqq  K(\alpha, \delta).
\]
\end{enumerate}
For variable error tolerance $\epsilon = (\e_n)_{n\in \NN}$ even finer tradeoffs between the a.s.~asymptotic error tolerance and the mean deviation frequency (error incidence) can be derived. 
 \end{rem}

{

\subsection{\textbf{A tradeoff quantification of a.s.~convergent M-estimators}}\label{ss:M-estimator}\hfill\\

\noindent M-estimators are one of the most elementary classes of point estimators in statistics based on the law of large numbers. So far, it was complicated to quantify the respective results for the strong law, with the results in Subsection~\ref{ss:BK}, however, we may quantify the tradeoff between the a.s.~rate of convergence v. its mean deviation frequency. 

\begin{defn}
For $\ell\in \NN$ we call a sequence of random variables $(Y_n)_{n\in \NN}$ \textbf{weakly $\ell$-stationary} if $\EE[|Y_1|^\ell]<\infty$ and
\begin{align*}
\EE[Y_n^j] =  \EE[Y_1^j], \qquad \mbox{ for all }n\in \NN, \quad j=1, \dots, \ell. 
\end{align*}
\end{defn}

\begin{rem}
\begin{enumerate}
 \item 

The most natural example are sequences of i.i.d.~random variables $(Y_i(\theta))_{i\in \NN}$.   
For instance given by strongly irreducible and positive recurrent homogeneous Markov chains on a countable state space $\mathbb{S}$ starting in its dynamical equilibrium (stationary distribution) $\pi$, both of which are strictly stationary.

\item Let us clarify the scope of the results of this section. For an i.i.d.~sequence $(Y_i)_{i\in \NN}$ with third moments we have by Kolmogorov's strong law that 
\[
\lim_{n\ra\infty} \frac{1}{n} \sum_{i=1}^n Y_i = \EE[Y_1]. 
\]
More over it is clear that $(Y_i^j)_{i\in \NN}$, $j=1, 2$  is also an i.i.d.~family of random variables which has first moments. Hence, again by Kolmogorov's strong law, we obtain 
\[
\lim_{n\ra\infty} \frac{1}{n} \sum_{i=1}^n Y_i^j = \EE[Y_1^j]. 
\]
However for a sequence $(\Delta_i X)_{i\in \NN}$ of martingale differences with finite third moments we have that $X_n = \sum_{i=1}^n \Delta_i X$ is a martingale and martingale strong laws apply for the process $\frac{1}{n} X_n$ 
in that under the assymption of weakly $3$-stationarity 
\[
\frac{1}{n} X_n \ra \EE[\Delta_1 X] \qquad \mbox{ in probability, as } n\ra\infty.  
\]
If we consider now the sequence $((\Delta_i X)^j)_{i\in \NN}$, $j=2,3$ it is not any more a sequence of martingale differences and no law of large numbers can be guaranteed in general. 
While for $j=2$ there still is a theory available due to the Doob-Meyer decomposition for the quadratic variation, for $j=3$ (or even higher moments) this cannot be guaranteed in general. For this reason we present our results for independent weakly $\ell$-stationary, though not necessarily strictly stationary (i.i.d.) increments. 
\end{enumerate} 
\end{rem}
\bigskip 
\noindent \textbf{The basic setup: } Given $k\gqq 1$ and an open bounded subset $\Theta \subseteq \RR^\ell$ of parameters, $\theta = (\theta_1, \dots, \theta_\ell) \in \Theta$, we consider a sequence of weakly $\ell$-stationary independent random variables $(Y_i(\theta))_{i\in \NN}$ with values in $\RR$ and distributions $\mu_i(\theta) := \PP_{Y_i(\theta)}$ which depend on $\theta$. 
For any $1\lqq j\lqq \ell$, $n\in \NN$, we set $M_j(\theta) := \EE[Y_{n}^j(\theta)]$. Note that due to the weak $\ell$-stationarity these moments are well-defined, and independent of $n\in \NN$. Now we define the complete 
vector of moments by 
\[\theta \mapsto M(\theta) := (M_1(\theta), \dots, M_\ell(\theta)).\] 
Consider for any fixed $\theta_0\in \Theta$ the $M$-estimator of $\theta_0$ by 
$\hat \theta_n(\theta_0) := M^{-1}(\bar X_n(\theta_0))\in \RR^\ell$, $n\in\NN$. 
For convenience we write $\theta_0 = (\theta_{0,1}, \dots, \theta_{0, \ell})$ and   
\[
\hat \theta_n(\theta_0) = (\hat \theta_{n, 1}(\theta_0), \dots, \hat \theta_{n, \ell}(\theta_0)), \qquad 
\text{ for } \theta\in \Theta.  
\]
For any $1\lqq j\lqq \ell$ we set $X_{n,j}(\theta) := \sum_{i=1}^n Y^j_{i}(\theta)$ and 
$\bar X_{n, j} := \frac{X_{n, j}}{n}$ and define the complete vector of higher order sample means by 
\begin{align*}
&\bar X_n(\theta) := (\bar X_{n, 1}(\theta), \dots, \bar X_{n,\ell}(\theta)).
    \end{align*}
\textbf{Assumptions: }
\begin{enumerate}\item[\textbf{\textnormal{(i)}}] Let $\sup_{\theta\in \Theta}\EE[|Y_i(\theta)|^{q}] < \infty$ for some $q>\ell$ and all $i\in \NN$. 
 \item[\textbf{\textnormal{(ii)}}] The mapping $\Theta \ni \theta \mapsto M(\theta)\in M(\Theta) \subseteq \RR^\ell$ is continuous and bijective. 
 \item[\textbf{\textnormal{(iii)}}] 
 The inverse $M^{-1}$ is continuously differentiable in $\Theta$.
\end{enumerate}
\bigskip
\noindent \textbf{Reduction to the law of large numbers: } 
We fix some $\theta_0\in \Theta$. By (ii) and (iii) There is $\e>0$ sufficiently small such that 
\begin{equation}\label{e:spectral}
\la = \la(\e) = \min\{|\mu|, \mu \in \mbox{spec}(D_{\theta_0} M)\}-\e>0, 
\end{equation}
where $D_{\theta_0} M$ is the Jacobi matrix of $M$ at the foot point $\theta_0$ and $B_{\delta}(x) = \{\|x-z\|< \delta\}\subseteq \RR^\ell$. 
For any $\e>0$ sufficiently small, there are $\delta_1, \delta_2 \in (0,1)$ such that 
    \begin{align}
    A_n(\e) &:= \{\|\hat \theta_n- \theta_0\|\gqq \e\} 
    = \{M^{-1}(\bar X_n(\theta_0)) \in B^c_{\e}(\theta_0)\}
    = \{\bar X_n(\theta_0) \in M(B^c_{\e}(\theta_0))\}\nonumber\\[2mm]
    & = \{\bar X_n(\theta_0) \in M(\Theta) \setminus M(B_{\e}(\theta_0))\} \subseteq \{\bar X_n(\theta_0) \in M(B_{\e}(\theta_0))^c\}\nonumber\\
    &\subseteq \{\bar X_n(\theta_0) \in \big((D_{\theta_0}M)(B_{\delta_1 \e}(\theta_0))\big)^c\}
     \subseteq  \{\bar X_n(\theta_0) \in (D_{\theta_0} M) B_{\delta_1 \e}^c(M(\theta_0))\}\nonumber\\[2mm] 
    &\subseteq \{\bar X_n(\theta_0) \in B_{\delta_1\delta_2\cdot \la\cdot \e}^c(M(\theta_0))\}
    = \{\|\bar X_n(\theta_0)-M(\theta_0)\|\gqq \delta_1\delta_2\cdot \la\cdot \e\}. \label{e:inclusion}
    \end{align}
   Consequently, for  
   \begin{align}\label{e:pivot}
    \PP(\|\hat \theta_{n}(\theta_0) -\theta_{0} \|>\e) \lqq   
    \PP(\|\bar X_{n}(\theta_0) -M(\theta_0) \|>\delta_1\delta_2\cdot \lambda \cdot \e).
    \end{align}
    \noindent We denote by $B_n(\e):= \{\|\bar X_n(\theta_0)-M(\theta_0)\|\gqq \delta_1 \delta_2\cdot\la\cdot  \e\}$ which results with the help of \eqref{e:inclusion} in 
        \[
        \oO_{\e, n_0} = \sum_{n=n_0}^\infty \ind(A_n(\e)) \lqq \m_{\e, n_0} = \sum_{n=n_0}^\infty \ind\Big(\bigcup_{m\gqq n} A_m(\e)\Big)
        \lqq \sum_{n=n_0}^\infty \ind\Big(\bigcup_{m\gqq n} B_m(\e)\Big)=:\tilde \m_{\e, n_0},
        \] 
    by monotonicity. We now define for some $n_0\in \NN$, and some positive, nonincreasing sequence $\epsilon = (\e_{n})_{n\gqq n_0}$ the quantities  
\[
\oO_{\epsilon,n_0} = \sum_{n = n_0}^\infty \ind\{\|\hat \theta_{n} - \theta_{0}\| >\e_{n}\} 
\qquad \mbox{ and }\qquad \m_{\epsilon,n_0} = \max\{n\gqq n_0~|~\|\hat \theta_{n} - \theta_{0}\| >\e_{n}\}.  
\]
\noindent In the sequel we follow the arguments of Theorem~\ref{thm:martSLLNq}(b) in order to implement the method of moments. 

\begin{thm}[\textbf{Method of moments: Data with Ces\`aro convergent $p$-th moments, $p> 4\ell$}]\label{thm:momCesaro}
We assume the preceding notation and Assumptions (i) and (ii) for some $k\in \NN$ and $p> 4\ell$.    
\[
\beta_{n,p}^* :=\sup\limits_{\theta\in \Theta} \frac{1}{n} \sum_{i=1}^n \EE[|Y_i(\theta)|^{p}].
\]
Assume that for some $n_0\in \NN$, a positive, nondecreasing sequence $a = (a_n)_{n\gqq 0}$ and a positive nonincreasing sequence $\epsilon = (\e_n)_{n\gqq n_0}$ we have 
\begin{equation}\label{e:SLLNnp}
K_{a, \epsilon, p, k,\Theta} = \sum_{n=n_0}^\infty a_n \sum_{m=n}^\infty \bigg(\frac{\beta_{m, \frac{p}{\ell}}^*}{\e_m^{\frac{p}{\ell}} m^{\frac{p}{\ell}-1}}+\frac{1}{\e_m^{\frac{p}{\ell}}m^{\frac{p}{2\ell}}}\bigg) < \infty.   
\end{equation}
Then $\hat \theta_n \ra \theta_0$ a.s.~as $n\ra\infty$ with the following  tradeoff:   
We have 
\[
\limsup_{n\ra\infty} \|\hat \theta_n(\theta_0)  - \theta_0\|\cdot \e_n^{-1} \lqq 1 \qquad \PP\mbox{-a.s.} 
\]
and there are positive constants $C_1,C_2>0$ such that 
\begin{align*}
\EE[\sS_{a, n_0}(\oO_{\epsilon, n_0})]{\lqq \EE[\sS_{a, n_0}(\m_{\epsilon, n_0})]}\lqq C_1 K_{a, \epsilon, p, k,\Theta} + C_2,  
\end{align*}
where $\sS_{a}$ is given in \eqref{def:Sa}. 
The constants depend on $\lambda, \delta_1, \delta_2$ given in \eqref{e:spectral} and \eqref{e:inclusion},
$a$, $\epsilon$, $p$, $\ell$, $\sup_{\theta\in \Theta} M_{2\ell}(\theta)$, $\sup_{\theta\in \Theta} \|M(\theta)\|$ and are given in the proof.
\end{thm}

\begin{rem}
\begin{enumerate}
 \item  Note that formally $p>2\ell$ is the only formal restriction in Theorem~\ref{thm:martSLLNq}(b), however, since $a$, $\beta^*_{\cdot, p}$ are nondecreasing and $\epsilon$ is nonincreasing, 
 the cases $\frac{p}{2\ell} \lqq 2$ cannot be quantified with our method, since \eqref{e:SLLNnp} cannot hold true in this case. 
 Therefore we assume $\frac{p}{\ell} > 4$ without loss of generality. 
 \item Since the Ces\`aro convergence is strictly weaker than norm convergence, 
 we can cover cases where $\beta_{m, \frac{p}{\ell}}^*$ diverges weakly at the cost of $a$ and $\e$. The fine play between the convergence rates of $\beta_{m, \frac{p}{\ell}}^*$, $a$ and $\e$, and the sizes of $p, \ell$, and $K_{a, \epsilon, p, \ell}$ is given by formula \eqref{e:SLLNnp}. 
\end{enumerate}
\end{rem}

\begin{proof} 
Note that the components $j=1, \dots, \ell$ of the process
\[
\bar X_n(\theta_0)-M(\theta_0) 
\]
are empirical means of centered independent random variables with moments of order at least $\frac{p}{j}$. 
Further recall that by \eqref{e:inclusion} we have 
\begin{align}
\{|\hat \theta_n- \theta_0|\gqq \e_n\}  
&\subseteq \{|\bar X_n(\theta_0)-M(\theta_0)|\gqq \delta_1\delta_2\cdot \la\cdot \e_n\}.\label{e:oderschranke}
\end{align}
Note that the components $j=1, \dots, \ell$ of the process
\[
\bar X_n(\theta_0)-M(\theta_0) 
\]
are empirical means of centered independent random variables with moments of order at least $\frac{p}{j}$.  
Since sums of centered independent random variables are martingales, 
the components of the process and the process itself are likewise 
empirical means of martingale differences.  
Now, due to $j\lqq \ell\lqq \frac{p}{4}$ 
\begin{align*}
\beta_{n,j}^* 
&=\sup\limits_{\theta\in \Theta} \frac{1}{n} \sum_{i=1}^n \EE[|Y_i(\theta)|^{j}] \lqq \sup\limits_{\theta\in \Theta} \Big(\frac{1}{n} \sum_{i=1}^n \EE[|Y_i(\theta)|^{p}]\Big)^\frac{j}{p}\lqq \max\Big\{ \sup\limits_{\theta\in \Theta}\frac{1}{n} \sum_{i=1}^n \EE[|Y_i(\theta)|^{p}], 1\Big\} = \max\{\beta_{n, p}^*, 1\}. 
\end{align*}
We assume without loss of generality that $\e_n \sqrt{n}\gqq 1$ for $n\gqq n_0$. 
By Markov's inequality we have that 
\begin{align*}
\PP(\|\bar X_n(\theta_0)-M(\theta_0)\|\gqq \delta_1\delta_2\cdot \la\cdot \e_n)
&\lqq (\delta_1\delta_2\cdot \la \e_n)^{-\frac{p}{\ell}} \cdot \EE[\|\bar X_n(\theta_0)-M(\theta_0)\|^{\frac{p}{\ell}}]. 
\end{align*}
By the Burkholder-Rosenthal inequality from \cite{Osek12}, as used in Theorem \ref{thm:martSLLNq} and Jensen's inequality there is a constant $C_{\frac{p}{\ell}}>0$ such that for $n\gqq n_0$ 
\begin{align*}
&\EE[\|\bar X_n(\theta_0)-M(\theta_0)\|^{\frac{p}{\ell}}]\\
&\lqq C_{\frac{p}{\ell}} n^{-\frac{p}{\ell}}\EE\Big[\sum_{i=1}^n \Big\| \Delta X_i(\theta_0)-M(\theta_0) \Big\|^\frac{p}{\ell}+\Big(\sum_{i=1}^n\EE\Big[\Big\|\Delta X_i(\theta_0)-M(\theta_0)\Big\|^2\big|\fF_{i-1}\Big]\Big)^{\frac{p}{2\ell}}\Big]\\
&=C_{\frac{p}{\ell}} n^{-\frac{p}{\ell}}\bigg(\EE\Big[  \sum_{i=1}^n \Big\|\Delta X_i(\theta_0)-M(\theta_0)\Big\|^\frac{p}{\ell}\Big]+\Big(\sum_{i=1}^n\EE\Big[\Big\|\Delta X_i(\theta_0)-M(\theta_0)\Big\|^2\Big]\Big)^{\frac{p}{2\ell}}\bigg)\\
&=C_{\frac{p}{\ell}} n^{-\frac{p}{\ell}}\bigg(\EE\Big[ \sum_{i=1}^n \Big\| \Delta X_i(\theta_0)-M(\theta_0)\Big\|^\frac{p}{\ell}\Big]+n^\frac{p}{2\ell}\EE\Big[\Big\|\Delta X_1(\theta_0)-M(\theta_0)\Big\|^2\Big]^{\frac{p}{2\ell}}\bigg)\\
&\lqq C_{\frac{p}{\ell}} n^{-\frac{p}{\ell}}\bigg(\EE\Big[ \sum_{i=1}^n \Big\| \Delta X_i(\theta_0)-M(\theta_0)\Big\|^\frac{p}{\ell}\Big]+n^{\frac{p}{2\ell}}\EE\Big[\Big\|\Delta X_1(\theta_0)-M(\theta_0)\Big\|^2\Big]^{\frac{p}{2\ell}}\bigg)\\
&\lqq C_{\frac{p}{\ell}} n^{-\frac{p}{\ell}}\bigg(2^{\frac{p}{\ell}-1}\Big(\EE\Big[\sum_{i=1}^n\Big\| \Delta X_i(\theta_0)\Big\|^\frac{p}{\ell}\Big]+n\Big\|M(\theta_0)\Big\|^\frac{p}{\ell}\Big)+n^{\frac{p}{2\ell}}2^{\frac{p}{2\ell}}\EE\Big[\Big\|\Delta X_1(\theta_0)\Big\|^2+\Big\|M(\theta_0)\Big\|^2\Big]^{\frac{p}{2\ell}}\bigg)\\
&\lqq C_{\frac{p}{\ell}} \Big(\frac{2}{n}\Big)^{\frac{p}{\ell}-1}\bigg(\frac{\ell^{\frac{p}{2\ell}-1}}{n}\EE\Big[ \sum_{i=1}^n\sum_{j=1}^\ell |Y_i^j(\theta_0)|^\frac{p}{\ell}\Big]+\Big\|M(\theta_0)\Big\|^\frac{p}{\ell}+n^{\frac{p}{2\ell}-1}\Big(\ell^{\frac{p}{2\ell}-1}\sum_{j=1}^\ell\big(\EE| Y_1^j(\theta_0)|^2\big)^\frac{p}{2\ell}+\Big\|M(\theta_0)\Big\|^\frac{p}{\ell}\Big)\bigg),
\end{align*}
where we also needed $\|v+w\|^r\lqq 2^{r-1} (\|v\|+ \|w\|)$ for $r\gqq 1$. Then, again by Jensen's inequality 
\begin{align*}
&\EE[\|\bar X_n(\theta_0)-M(\theta_0)\|^{\frac{p}{\ell}}] \\
&\lqq C_{\frac{p}{\ell}} \Big(\frac{2}{n}\Big)^{\frac{p}{\ell}-1}\bigg(\frac{\ell^{\frac{p}{2\ell}-1}}{n}\EE\Big[ \sum_{i=1}^n\sum_{j=1}^\ell |Y_i^j(\theta_0)|^\frac{p}{\ell}\Big]+\Big\|M(\theta_0)\Big\|^\frac{p}{\ell}+n^{\frac{p}{2\ell}-1}\Big(\ell^{\frac{p}{2\ell}-1}\sum_{j=1}^\ell\big(\EE| Y_1^j(\theta_0)|^2\big)^\frac{p}{2\ell}+\Big\|M(\theta_0)\Big\|^\frac{p}{\ell}\Big)\bigg)\\
&=C_{\frac{p}{\ell}}\Big(\frac{2}{n}\Big)^{\frac{p}{\ell}-1}\bigg(\frac{\ell^{\frac{p}{2\ell}-1}}{n}\EE\Big[ \sum_{i=1}^n\sum_{j=1}^\ell |Y_i(\theta_0)|^\frac{jp}{\ell}\Big]+\Big\|M(\theta_0)\Big\|^\frac{p}{\ell}+n^{\frac{p}{2\ell}-1}\Big(\ell^{\frac{p}{2\ell}-1}\sum_{j=1}^\ell\big(\EE| Y_1(\theta_0)|^{2j}\big)^\frac{p}{2\ell}+\Big\|M(\theta_0)\Big\|^\frac{p}{\ell}\Big)\bigg)\\
&\lqq C_{\frac{p}{\ell}} \Big(\frac{2}{n}\Big)^{\frac{p}{\ell}-1}\bigg(\ell^{\frac{p}{2\ell}-1}\sum_{j=1}^\ell \Big( \frac{1}{n}\sum_{i=1}^n\EE\big[|Y_i(\theta_0)|^p\big]\Big)^\frac{j}{\ell}+\sup_{\theta\in \Theta}\Big\|M(\theta)\Big\|^\frac{p}{\ell}\\
&\quad\quad\quad\quad\quad\quad\quad\quad\quad\quad\quad+n^{\frac{p}{2\ell}-1}\Big(\ell^{\frac{p}{2\ell}-1}\sum_{j=1}^\ell\big(\EE| Y_1(\theta_0)|^{2\ell}\big)^\frac{p}{2\ell}+\sup_{\theta\in\Theta}\Big\|M(\theta_0)\Big\|^\frac{p}{\ell}\Big)\bigg). 
\end{align*}
Finally we use for $x^\frac{1}{r} \lqq \max\{x, 1\}$ for $r\gqq 1$ and $x\gqq 0$ and conclude 
\begin{align*}
&\EE[\|\bar X_n(\theta_0)-M(\theta_0)\|^{\frac{p}{\ell}}] \\
&\lqq C_{\frac{p}{\ell}}\Big(\frac{2}{n}\Big)^{\frac{p}{\ell}-1} \bigg(\ell^{\frac{p}{2\ell}-1}\Big(\max\Big\{ \frac{1}{n}\sum_{i=1}^n \EE\big[|Y_i(\theta_0)|^p\big],1\Big\}+n^{\frac{p}{2\ell}-1}\big(\EE\big[| Y_1(\theta_0)|^{2\ell}\big]\big)^\frac{p}{2\ell}\Big)\\
&\quad\quad\quad\quad\quad\quad\quad\quad\quad\quad\quad\quad+\big(1+n^{\frac{p}{2\ell}-1}\big)\sup_{\theta\in \Theta}\|M(\theta)\|^\frac{p}{\ell}\bigg).
\end{align*}
By the definition of $\beta^*_{n,\frac{p}{\ell}}$, collecting constants of the above terms into $c_{p,\ell}$ and noting that $|Y_1(\theta_0)|^{2\ell}=M_{2\ell}(\theta_0)$, we get
\begin{align*}
\EE[\|\bar X_n(\theta_0)-M(\theta_0)\|^{\frac{p}{\ell}}] \lqq c_{p,\ell}n^{-\frac{p}{\ell}+1}\Big(\max\Big\{\beta^*_{n,\frac{p}{\ell}},1\Big\}+n^{\frac{p}{2\ell}-1}\sup_{\theta\in \Theta}\big(M_{2\ell}(\theta)\big)^\frac{p}{2\ell}+\big(1+n^{\frac{p}{2\ell}-1}\big)\sup_{\theta\in \Theta}\|M(\theta)\|^\frac{p}{\ell}\Big).
\end{align*}
\noindent We include the suprema involving $M$ in a new constant $\tilde{C}_{p,\ell,\Theta}$ and observe that, by Lemma~\ref{lem:Quant BC for e_n}, the finiteness of
\begin{align}
K_{a, \epsilon}^* = K^*_{a, \epsilon, p, \ell, \beta^*, \lambda, \delta_1, \delta_2} 
&:= \frac{ \tilde{C}_{p,\ell,\Theta}}{\delta_1\delta_2\lambda}  \sum_{n=n_0}^\infty a_n \sum_{m=n}^\infty \Big(\frac{\max\left\{\beta_{m, p}^*, 1\right\}}{\e_m^{\frac{p}{\ell}} m^{\frac{p}{\ell}-1}}+\frac{1}{\e_m^{\frac{p}{\ell}} m^{\frac{p}{2\ell}}}\Big),\label{e:richtigeSumme}
\end{align}
implies that 
\begin{equation}\label{e:komponentenweise}
\limsup_{n\ra\infty} \|\bar X_n(\theta_0)-M(\theta_0)\|\cdot \Big(\delta_1\delta_2\cdot \la\cdot \e_n\Big)^{-1} \lqq 1\qquad \PP\mbox{-a.s.,}
\end{equation}
which yields by \eqref{e:inclusion} 
\begin{equation}\label{e:komponentenweise2}
\limsup_{n\ra\infty} \|\hat \theta(\theta_0) -\theta_0\|\cdot \e_n^{-1} \lqq 1\qquad \PP\mbox{-a.s.}
\end{equation}
We end up stating that we have 
\[
\EE[\sS_{a, n_0}(\oO_{\epsilon, n_0})]{\lqq \EE[\sS_{a, n_0}(\m_{\epsilon, n_0})]\lqq \EE[\sS_{a, n_0}(\tilde \m_{\delta_1\delta_2\lambda \epsilon})]}\lqq   K^*_{a, \epsilon}. 
\]\end{proof}

In the sequel we apply Corollary~\ref{cor:BKmart} to the method of moments. 

\begin{thm}[\textbf{Method of moments: Uniformly $L^p$ bounded data, $p>6\ell$}]\label{thm:momBK}\hfill\\
For any $p>6\ell$ assume that $\sup_{\theta\in \Theta} \sup_{i\in \NN} \EE[|Y_i(\theta)|^p] <\infty$.
Then $\hat \theta_n \ra \theta_0$ a.s.~as $n\ra\infty$. 
For any $\eta>0$ and $\alpha> 3$  
we define $\epsilon = \epsilon(\alpha, \eta, r) = (\e_n(\alpha, \eta, r)_{n\in \NN}$, $\e_n(\alpha, \eta, r) := \eta n^{\frac{\alpha}{r}-1}$ and $n_0\in \NN$  
\[
\oO_{\epsilon, n_0} := \sum_{n=n_0}^\infty \ind\Big\{\|\hat \theta_n(\theta_0) - \theta_0\| \gqq \e_n(\alpha, \eta, \frac{p}{\ell})\Big\}{,\qquad 
\m_{\epsilon, n_0} := \max\Big\{n\gqq n_0~|~\|\hat \theta_n(\theta_0) - \theta_0\| \gqq \e_n(\alpha, \eta, \frac{p}{\ell})\Big\}}. 
\]
Then we have the following  tradeoff: 
For any $\alpha>3$ such that $\frac{p}{2\ell} <  \alpha \lqq \frac{p}{\ell}$ there is a constant $C>0$ 
such that for $0 < \tilde p < \alpha-3$ we have 
\[
\|\hat \theta_n(\theta_0) - \theta_0\| \cdot \e_n^{-1}(\alpha, \eta, \frac{p}{\ell})\lqq 1 \qquad \PP\mbox{-a.s.},  
\]
and there is a constant $C>0$ such that for $\epsilon = (\e_{n}(\alpha, \eta, \frac{p}{\ell}))_{n\in \NN}$ we have 
\[
\EE[\oO_{\epsilon, n_0}^{1+\ti p}] {\lqq \EE[\m_{\epsilon, n_0}^{1+\ti p}]}\lqq 
C (\alpha-1)\zeta(\alpha -2-\tilde p, n_0).
\]
Moreover, we obtain by Example~\ref{ex:poly} that 
\[
\PP(\oO_{\e, n_0}\gqq k) {\lqq \PP(\m_{\e, n_0}\gqq k)} \lqq k^{-(\ti p+1)} \cdot C (\alpha-1)\zeta(\alpha -2-\tilde p, n_0)
\qquad \mbox{ for }k\gqq 1. 
\]
\end{thm}

\bigskip 
\begin{proof}
Note that $\bar X_n(\theta_0) - M(\theta_0)$ is a martingale with values in $\RR^\ell$ 
with moments of order $\frac{p}{\ell}$.    
Note that under $3<\frac{p}{2\ell}< \alpha < \frac{p}{\ell}$ the Kronecker argument in 
the proof of Theorem~\ref{thm:BK} tells us that for any $0 < \tilde p < \alpha -3$ there are $n_0$ and $C>0$ such that 
\[
\PP\Big(\|\hat \theta_n(\theta_0) - \theta_0\| \gqq \e_n\big(\alpha, \eta, \frac{p}{\ell}\big)\Big)
\lqq \PP\Big(\|\bar X_n(\theta_0) - M(\theta_0)\| \gqq  \e_n\big(\alpha, \delta_1\delta_2\lambda\eta, \frac{p}{\ell}\big)\Big)
\lqq \frac{C}{n^{\alpha-1}} 
\]
for all $n\gqq n_0$. By Lemma~\ref{lem:Quant BC for e_n} in combination with Example~\ref{ex:poly} we have 
\[
\limsup_{n\ra\infty} \|\hat \theta_n(\theta_0) - \theta_0\|\cdot \e_n^{-1}(\alpha, \eta, \frac{p}{\ell}) \lqq 1\qquad \PP\mbox{-a.s.}
\]
and 
\[
\EE\Big[\oO_{\epsilon, n_0}^{\tilde p+1}\Big]{\lqq \EE\Big[\m_{\epsilon, n_0}^{\tilde p+1}\Big]}\lqq C (\alpha-1) \zeta(\alpha-1-\tilde p, n_0). 
\] 
This finishes the proof. 
\end{proof}

Let us apply Theorem~\ref{lem:exponentialLdp} in the method of moments.  

\begin{thm}[\textbf{Method of moments: Data with uniformly bounded exponential moments}]\label{thm:momUExp} 
 Assume there is a constant $\gamma >0$ such that 
 \[
\sup_{\theta\in \Theta} \sup\limits_{i\gqq n_0} \EE\Big[e^{\gamma |Y_i(\theta)|}\Big] <\infty.
 \]  
 Then for any $0 \lqq \alpha < \frac{1}{2}$ the choice of $\epsilon = (\e_n)_{n\in \NN}$ 
 $\e_{n} = n^{-\alpha}$, $n\in \NN$,  we have for any $\ell\in \NN$ 
 \[
 \limsup_{n\ra\infty} \| \hat \theta(\theta_0)-\theta_0\|\cdot \e_n^{-1}\lqq 1 \qquad \PP\mbox{-a.s.} 
 \]
 and for $c_0 = (\delta_1\delta_2 \lambda \gamma)^\frac{2}{3}$ and $p\in (0,1)$ there is a constant $K :=K(\alpha, n_0)>0$ such that 
 the respective quantities $\oO_{\epsilon, n_0}$ and $\m_{\epsilon, n_0}$ satisfy 
 \[
 \EE\Big[\exp\big(p c_0 (\oO_{\epsilon, n_0})^{1-2\alpha}\big)\Big] {\lqq \EE\Big[\exp\big(p c_0 (\m_{\epsilon, n_0})^{1-2\alpha}\big)\Big]} \lqq K. 
 \]
 By Example~\ref{ex:Weibull} there exist constants $d, D>0$ such that for all $k\gqq 2$ 
 \[
 \PP(\oO_{\epsilon, n_0}\gqq k){\lqq  \PP(\m_{\epsilon, n_0}\gqq k)} \lqq
 \Big(d + D(k-1)^{1+2\alpha}\Big)\, e^{- c_0 k^{1-2\alpha}}. 
 \]

\end{thm}

\begin{proof} Note that $n \bar X_n(\theta_0)- M(\theta_0) = X_n(\theta_0)- n M(\theta_0)$ is 
a random walk of centered random variables in $\RR^\ell$. Further note that 
\begin{align*}
\EE[|Y_i(\theta_0)|^\ell] 
&= \frac{\ell!}{\gamma^\ell} \EE\Big[\frac{(\gamma |Y_i(\theta_0)|)^\ell}{\ell!}\Big]\lqq \frac{\ell!}{\gamma^\ell} \sup_{\theta\in \Theta} \EE[e^{\gamma |Y_i(\theta)|}].
\end{align*}
By Remark~\ref{rem:LDPvergleich} for any $\delta\in (0,1)$ 
there is $n_0\in \NN$ such that for $n\gqq n_0$ we have 
\begin{align*}
\PP(\| \bar X_n(\theta_0)- M(\theta_0)\| > \e_n) 
&= \PP(\sum_{j=1}^\ell |\bar X_{n, j}(\theta_0)- M_j(\theta_0)|^2  > \e_n^2) 
\lqq \sum_{j=1}^\ell \PP(|\bar X_{n, j}(\theta_0)- M_j(\theta_0)|  > \frac{\e_n}{\sqrt{\ell}})\\ 
&\lqq 2\ell e^{- n \inf_{|y|>\frac{\e_n}{\sqrt{\ell}}} \Lambda^*(y)}
\lqq \ell e^{-\frac{1-\delta}{2\ell } |(\Lambda^*)''(0)| n \e_n^2  },
\end{align*}
where $\Lambda^*$ is the good rate function defined in Remark~\ref{rem:LDPvergleich}. Hence for $\epsilon = (\e_n)_{n\in \NN}$ $\e_{n} := n^{-\alpha}$, $0 < \alpha < \frac{1}{2}$, $n\in \NN$ \eqref{e:inclusion} we have 
\begin{align*}
\PP(\|\hat\theta_n(\theta_0) - \theta_0\|> \e_n) 
\lqq 
\PP(\| \bar X_n(\theta_0)- M(\theta_0)\| > \delta_1\delta_2 \lambda \e_n) 
\lqq \ell e^{-\frac{(1-\delta) \delta_1^2\delta_2^2 \lambda^2}{2\ell } |(\Lambda^*)''(0)| n^{1-2 \alpha}  }.
 \end{align*}
By Lemma~\ref{lem:Quant BC for e_n} combined with Example~\ref{ex:Weibull} we have 
\[
\limsup_{n\ra\infty} \| \bar X_n(\theta_0)- M(\theta_0)\|\cdot \e_n^{-1} \lqq 1 \qquad \PP\mbox{-a.s.}
\]
and constants $d, D>0$ such that for $k\gqq 2$ 
\[
\PP(\oO_{\epsilon, n_0}\gqq k){\lqq \PP(\m_{\epsilon, n_0}\gqq k)} \lqq (d + D(k-1)^{1+2 \alpha}) 
e^{-\frac{(1-\delta) \delta_1^2\delta_2^2 \lambda^2}{2\ell } |(\Lambda^*)''(0)| k^{1-2\alpha}}.
\]
\end{proof}

\noindent Very similar to the above, in the sequel we present the tradeoff in the presence of a large deviation principle. 

\begin{thm}[\textbf{Data in a G\"artner-Ellis setting}]\label{thm:momGE} 
 Assume for all $z\in \RR^\ell$ 
 \[
 \Lambda_{n}(z) := \ln\Big(\EE\Big[e^{\lgl z, \bar X_{n}\rgl}\Big]\Big)<\infty. 
 \]
Further assume that for all $z\in \RR^\ell$ the limit 
 \[
 \lim_{n\ra\infty} \frac{\Lambda_{n}(nz)}{n}= \Lambda(z)\in \bar \RR
 \]
 exists as an extended real number $\bar \RR = \RR \cup\{\infty\}$. 
 Finally, we assume that $\Lambda$ is finite in an open neighborhood of $0$.  
 We denote by $\Lambda^*(\zeta):= \sup_{y\in \RR^\ell} \big(\lgl y,\zeta\rgl-\Lambda(y)\big)$, $z\in \RR^\ell$ the Fenchel-Legendre transform of $\Lambda$, see \cite[Def. 2.2.2]{DZ98}, by $D^2 \Lambda^*(0)$ the Hessian of $\Lambda^*$ evaluated at $0$, and by $\lvert\lvert\lvert D^2\Lambda^*(0)\lvert\lvert\lvert$ its the operator norm. 
 
Then for $0 < \alpha < \frac{1}{2}$ there are constants $\delta_1, \delta_2, \delta_3 \in (0,1)$ and $\lambda>0$ and $0 <p < \frac{1}{2} \delta_1\delta_2 (1-\delta_3)\lambda \lvert\lvert\lvert D^2\Lambda^*(0)\lvert\lvert\lvert$ the choice 
$\e_n = n^{-\alpha} $, $n\in \NN$, yields 
\[
\limsup_{n\ra\infty} \|\hat\theta_n(\theta_0)- \theta_0\| \cdot n^{-\alpha} \lqq 1\qquad \PP\mbox{-a.s.} 
\]
and some $n_0\in \NN$ such that the respective overlap statistics $\oO_{\epsilon, n_0} := \sum_{n=n_0}^\infty \ind\{\|\hat \theta_n(\theta_0)-\theta_0\|>\e_n\}$ and $\m_{\epsilon, n_0} := \max\{n\gqq n_0~|~ \|\hat \theta_n(\theta_0)-\theta_0\|>\e_n\}$ satisfies 
 \[
 \EE\Big[e^{p (\oO_{\epsilon, n_0})^{1-2\alpha}}\Big] \lqq \EE\Big[e^{p (\m_{\epsilon, n_0})^{1-2\alpha}}\Big] < \infty.
 \]
 Furthermore, there are some constants $d, D>0$ such that for all $k\gqq 2$ 
 \[
 \PP(\oO_{\epsilon, n_0}\gqq k)\lqq \PP(\m_{\epsilon, n_0}\gqq k)\lqq C (d + D(k-1)^{1-2\alpha}) \exp(-k^{1-2\alpha} \frac{1-\delta}{2} \lvert\lvert\lvert D^2\Lambda^*(0)\lvert\lvert\lvert)
 \]
\end{thm}

\begin{proof}
The upper bound in the G\"artner-Ellis theorem \cite[Thm. 2.3.6 (a)]{DZ98} yields a constant $C>0$ and some $n_0\in \NN$ such that for $n\gqq n_0$ we have  
  \[
  \PP(|\bar X_{n}(\theta_0) - M(\theta_0)|>\e) \lqq C \exp(-n \inf_{|\zeta|>\e} \Lambda^*(\zeta)).
  \]
  
Now, $\Lambda$ is smooth in a neighborhood of $0$ and $\Lambda(0) = 0$. 
In addition, it is the minimum, so the Taylor expansion implies for each $\delta \in (0,1)$, a value $\e>0$ such that $\|y\| \lqq \e$ implies 
  \[
 \Lambda_j^*(y) \gqq \frac{1-\delta}{2}  \lgl y, D^2\Lambda_j^*(0)y\rgl
 \gqq \frac{1-\delta_3}{2} \lvert\lvert\lvert D^2\Lambda^*(0) \lvert\lvert\lvert ~\|y\|^2,    
  \]
  for some $\delta_3\in (0,1)$. 
  Hence by continuity of $\Lambda^*$ we have 
  \[
  \inf_{\|y\|>\e}  \Lambda^*(y) \gqq (1-\delta_1) \frac{\e^2}{2}\lvert\lvert\lvert D^2\Lambda^*(0)\lvert\lvert\lvert. 
  \]
  Therefore, $\e_n = n^{-\alpha}$, $n\in \NN$, $\alpha \in (0,1)$ yields 
  \begin{align*}
  \PP(\|\bar X_{n}(\theta_0) - M(\theta_0)\|>\e_n) 
  &\lqq C \exp(-n \inf_{\|y\|>\e_n} \Lambda_j^*(y)) \lqq C \exp(-n^{1-2\alpha} \frac{1-\delta}{2} \lvert\lvert\lvert D^2\Lambda^*(0)\lvert\lvert\lvert).
  \end{align*}
 A combination of Lemma~\ref{lem:Quant BC for e_n} and Example \ref{ex:Weibull} finishes the proof. 
 \end{proof}
} 
 Finally we apply Theorem~\ref{thm:cerraduraAzuma} in a method of moments quantification, which can be applied for the estimation of parameters in P\'olya's urn models. 

\begin{thm}[\textbf{Almost surely uniformly bounded data}]\label{thm:momAzuma}\hfill\\
  Assume $|Y_n| \lqq c_n$ $\PP$-a.s.~for all $n\in \NN$ 
  for a positive sequence $c = (c_n)_{n\in \NN}$.  
  Assume further that for a sequence $\epsilon = (\e_{n})_{n\in \NN}$ positive nonincreasing and  
  $a = (a_{n})_{n\in \NN}$ positive nondecreasing such that for the constants $\lambda, \delta_1, \delta_2>$ given in \eqref{e:inclusion} we have  
  \[
  K_{a,\epsilon, \ell, N} := 2 \sum_{n= 1}^\infty a_n \sum_{m=n}^\infty  \exp\Big( - \frac{(\lambda\delta_1\delta_2)^2}{2\ell} \frac{m^2 \e_m^2}{\sum_{i=m+1}^\infty c_i^{2}}\Big) <\infty, \mbox{ for some }N\in \NN. 
  \]
 Then 
 \[
 \limsup_{n\ra\infty} \|\hat \theta_n(\theta_0)- \theta_0\| \cdot \e_n^{-1} \lqq 1\qquad \PP\mbox{-a.s.} 
 \]
 and there exists $n_0\in \NN$ such that 
 for the respective statistics $\oO_{\epsilon, n_0}$ {and $\m_{\epsilon, n_0}$} we have 
 \[
 \EE[\sS_{a, n_0}(\oO_{{\epsilon}, n_0})]{\lqq \EE[\sS_{a, n_0}(\m_{{\epsilon}, n_0})]}\lqq K_{a,\epsilon, \ell, n_0}.  
 \]
 \end{thm}

 \bigskip 
\begin{proof}
  If $|Y_i| \lqq c_i$ a.s.~for all $n\in \NN$, then $|Y_i|^j\lqq c_i^j$. 
  Furthermore $\sum_{m=n}^\infty c_m^2< \infty$ for some $n\in\NN$ implies that 
  $c_m\ra 0$ and hence there is $n_0$ such that $c_m< 1$. 
  Therefore 
  \[
  \|(Y_i, Y_i^2, \dots, Y_i^\ell)\|\lqq \sum_{j=1}^k |Y_i|^j \lqq \sum_{j=1}^\ell c_i^j\lqq k c_i.
  \]
  For any $\epsilon = (\e_{n})_{n\in \NN}$ and $a = (a_{n})_{n\in \NN}$ positive nonincreasing 
  and $a = (a_n)_{n\in \NN}$ positive nondecreasing such that 
  \[
  K := 2 \sum_{n= 1}^\infty a_n \sum_{m=n}^\infty 2 \exp\Big( - \frac{m^2 \e_m^2}{2\ell \sum_{i=m+1}^\infty c_i^{2j}}\Big)< \infty
  \]
 Theorem~\ref{thm:cerraduraAzuma} implies that 
 \[
 \limsup_{n\ra\infty} \|\bar X_n(\theta_0)- M(\theta)\| \cdot \e_n^{-1} \lqq 1 \qquad \PP\mbox{-a.s.}   
 \]
and 
 \[
 \EE[\sS_{a, n_0}( \m_{\epsilon, n_0})] \lqq K. 
 \]
 An application of \eqref{e:inclusion} and an appropriate reparametrization finishes the proof.  
\end{proof}

 \bigskip 
\subsection{\textbf{Excursion dynamics of the Galton-Watson branching process}}\label{ss:branching}\hfill\\

\noindent We recall and keep the notation of Example~\ref{ex:branching}. In this subsection, we see that the tradeoff between error tolerance and mean deviation frequencies also works for a.s.~converging processes, but also for a.s.~divergence, such as for the super-critical branching processes. 

\subsubsection{\textbf{Sub-critical branching: $\mathbf{\mm<1}$}}\hfill\\

\noindent In case of sub-critical branching we know that $X_n\ra 0$ almost surely. 
With the help of the tradeoff in Lemma~\ref{lem:Quant BC for e_n} we may quantify 
the tradeoff between the number of excursions beyond a growing threshold $K(n)$. 
\begin{lem}\label{lem:subcritbranch}
{\color{black}
For $\mm\in (0,1)$ and $K>0$ we have for all $k\gqq 1$, for 
$\oO_K:=\#\{\ell\in\NN~|~Z_\ell \gqq K\}=\sum_{\ell\in \NN}\bI\{Z_\ell\gqq K\}$ and $\m_K = \max\{n\in \NN~|~Z_n\gqq K\}$ 
\begin{align*}
\PP(\oO_K\gqq k){\lqq \PP(\m_K\gqq k)}\lqq 
\frac{2e^{\frac{9}{8}}}{\mm(1-\mm)K}\cdot 
\Big[k \big((\frac{v}{\mm(1-\mm)K} \wedge 1)+1\big) +1\Big]\cdot \mm^k.
\end{align*}}
\end{lem}

\begin{proof}[\textbf{Proof of Lemma~\ref{lem:subcritbranch}}:]
For the case $\mm\in (0,1)$ it is well-known (using Markov's inequality, as well as Wald's and Blackwell-Girshwick's identities \cite[Theorem 5.5 and 5.10]{Kle08}) that for any $K>0$ fixed 
\begin{equation}
\PP(Z_\ell\gqq K)\lqq \mm^\ell\cdot \Big(\frac{1}{K}\wedge \frac{v}{\mm(1-\mm)K^2}\Big).
\end{equation}
Hence 
\begin{equation}
\sum_{\ell\gqq n}^\infty\PP(Z_\ell \gqq K)\lqq \mm^n\cdot \left(\frac{1}{(1-\mm)K}\wedge \frac{v}{\mm(1-\mm)^2K^2}\right)
\end{equation}
and Example~\ref{ex:exp} yields for any $p \in (0,1)$ and $a_n = \mm^{-pn}$, $n\in \NN$, that 
\[
\EE[\mm^{-p\oO_K}]  = \EE[e^{|\ln(\mm)| p\oO_K}] {\lqq \EE[e^{|\ln(\mm)| p\m_K}]\lqq} \Big(\frac{1}{1- \mm^{1-p}}+1\Big)\cdot
\Big(\frac{1}{\mm(1-\mm)K} \wedge \frac{v}{\mm^2(1-\mm)^2K^2}\Big). 
\]
In addition, by \cite[Lemma 5]{HS23}    
we have for all $k\gqq 1$ 
\[
\PP(\oO_K\gqq k) \lqq {\PP(\m_K\gqq k) \lqq \inf\limits_{0<q< |\ln(\mm)|} e^{-qk } \EE[e^{q\m_K}]
\lqq} 2 e^{\frac{9}{8}}\Big[k\big(\frac{1}{\mm(1-\mm)K} \wedge \frac{v}{\mm^2(1-\mm)^2K^2}+1\big) +1 \Big] \cdot\mm^{k-1}.
\]
\end{proof}

\subsubsection{\textbf{super-critical branching: $\mathbf{\mm>1}$}}\hfill\\

\begin{lem}\label{lem:supercritbranch}
 For $\mm>1$ and $\EE [Y_1^2]<\infty$ we have the following: {\color{black}
 \begin{enumerate}
  \item For any $\e>0$ and 
  $\oO_\varepsilon:=\#\{\ell\in\NN~|~|X_\ell-X_\infty|\gqq \e\}=\sum_{\ell\in \NN}\bI\{|X_\ell-X_\infty|\gqq \e\}$, 
  {and $\m_\varepsilon := \max\{\ell\in \NN~|~|X_\ell-X_\infty|\gqq \e\}$} 
  we have
\begin{align*}
\PP(\oO_\varepsilon\gqq k)
&\lqq 
\frac{2 e^{\frac{9}{8}} v}{(1-\tfrac{1}{\mm})(\mm^2-\mm)\e^2} \cdot (2k+1)\cdot  \mm^{k-1}, \qquad k\gqq 1. 
\end{align*}
 \item For any $\theta>1$ and $\e_n:= \sqrt{\tfrac{v n^\theta}{\mm^n (\mm^2 -\mm)}}$ 
 we have 
 \[
 \limsup_{n\ra\infty} |X_n-X_\infty|\cdot \e_n^{-1} \lqq 1\qquad \PP\mbox{-a.s.}
 \]
and, for $\oO_\epsilon:=\#\{\ell\in\NN~|~|X_\ell-X_\infty|\gqq \e_n\}=\sum_{\ell\in \NN}\bI\{|X_\ell-X_\infty|\gqq \e_n\}$ 
{ and $\m_\epsilon := \max\{n\in \NN~|~|X_n-X_\infty|\gqq \e_n\}$} we get
\begin{align*}
\PP(\oO_\epsilon\gqq k){\lqq \PP(\m_\epsilon\gqq k)}\lqq k^{-1} \cdot \zeta(\theta), \qquad k\gqq 1.
\end{align*}
 \end{enumerate}}
\end{lem}

\begin{rem}
Note that item (a) and (b) represent extremal cases. Case (a) counts the (random) number of infractions 
of a fixed error bar $\e>0$. Case (b) instead yields close to optimal rates of convergence $\e_n\ra 0$, $n\ra\infty$, which are obtained only with a linear decay of the deviation frequency, that is, only after many infractions. In other words, the trespassing probabilities are barely summable. 

Of course, all kind of useful tradeoff regimes between (1) and (2) can be derived by the same methodology. 
\end{rem}

\medskip

\begin{proof}[\textbf{Proof of Lemma~\ref{lem:supercritbranch}}:] 
For $\mm>1$, we get that  if $\EE [Y_1^2]<\infty$, $X_n:=\frac{Z_n}{\mm^n}$ is a martingale such that $\EE [X_n^2]=1+\frac{v\mm^n}{\mm^2-m}(1-\mm^{-n})$ (see e.g. \cite[Proof of Theorem 8.1]{Harris64}).
Moreover  
\begin{align*}
\EE[(X_{n+\ell}-X_n)^2]=\frac{v\mm^{-n}}{\mm^2-\mm}(1-\mm^{-\ell}),
\end{align*}
 showing the convergence $X_n\to X_\infty$ in $L^2$, and, in addition, 
\begin{equation}\label{e:summability}
\sum_{\ell=n}^\infty \PP(|X_\ell-X_\infty|\gqq \e)\lqq \sum_{\ell=n}^\infty \frac{\EE[(X_\infty-X_\ell)^2]}{\e^2}=\sum_{\ell=n}^\infty \frac{v\mm^{-\ell}}{(\mm^2-\mm)\e^2}=\frac{v\mm^{-n}}{(1-\tfrac{1}{\mm})(\mm^2-\mm)\e^2}<\infty.
\end{equation}
Hence, Example~\ref{ex:exp} yields for any $p \in (0,1)$, $a_n = \mm^{pn}$, $n\in \NN$, that 
\begin{align*}
\EE[\mm^{p\oO_\e}] {\lqq \EE[\mm^{p\m_\e}]} \lqq \frac{v}{(1-\tfrac{1}{\mm})(\mm^2-\mm)\e^2} \Big(\frac{1}{1-\mm^{p-1}}+1\Big), 
\end{align*}
such that 
\begin{align*}
\PP(\oO_\varepsilon\gqq k)
{\lqq \PP(\m_\varepsilon\gqq k)}
&\lqq \inf\limits_{p\in (0,1)} e^{-pk} \EE[\mm^{p\oO_\e}]\lqq \frac{2 e^{\frac{9}{8}} v}{(1-\tfrac{1}{\mm})(\mm^2-\mm)\e^2} \cdot (2k+1)\cdot  \mm^{-k}. 
\end{align*}
For the second statement we use the classical first Borel-Cantelli lemma in \eqref{e:summability} for $n=1$. 
We calculate $\e_\ell$ by setting 
\[
\frac{v\mm^{-\ell}}{(\mm^2-\mm)\e_\ell^2}  = \ell^{-\theta}. 
\]
This finishes the proof. 
\end{proof}

\bigskip 
\subsubsection{\textbf{Critical branching: $\mathbf{\mm=1}$}}\hfill\\ 
 
\noindent In the critical case, it is known that $X_n\to 0$ a.s., however, it is classical, that $X_n \not\ra 0$ in $L^1$. 
{\color{black} We consider for $K\gqq 1$ the count $\oO_K:=\sum_{n=1}^\infty\bI\{X_n\gqq K\}$

\begin{rem}\label{rem:critial} We first illustrate that the mean deviation estimates yield almost optimal rates of convergence. 
If $\EE[Y_1^2]<\infty$, $v:=\EE[Y_1^2]-1$, then by \cite{Ko38} $\PP(X_n\gqq 1)\sim \frac{2}{vn}$, which is not summable. However, the sequence $\PP(X_n>0)$ for $n\gqq 1$ is nested. Note that $$\oO_1:=\sum_{n=1}^\infty\bI\{X_n\gqq 1\}=\inf\{n\gqq 1: X_n=0\}-1.$$
In addition, $\oO_1 = \m_1$ due to the nestedness of the events for $\m_1(\omega) = \max\{n\in \NN~|~X_n\gqq 1\}$. 
By Lemma~\ref{lem:Quant BC for e_n}, we get for $\sS(N):=\sum_{n=1}^N a_n$ with $\sS(0) = 0$, that 
$\EE [\sS(\oO_1)] = \EE [\sS(\m_1)]=\sum_{n=1}^\infty a_n\PP(X_n>0)$. Setting $a_n:=\frac{1}{\ln^2(n+1)}$, $n\gqq 1$, we obtain a constant $\tilde{c}>0$ with
$$
\EE [\sS(\oO_1)] = \EE [\sS(\m_1)]\lqq \tilde{c}\sum_{n=1}^\infty\frac{2}{vn\ln^2(n+1)}<\infty.
$$
As there is a constant $c>0$ such that $c\frac{N}{\log(1+N)}\lqq \sS(N)$, we find that there is a constant $C$ such that 
$$
\EE \left[\frac{\oO_1}{\ln(1+\oO_1)}\right] = \EE \left[\frac{\m_1}{\ln(1+\m_1)}\right]\lqq C<\infty.
$$
The rate obtained by Markov's inequality yields $\PP(\oO_1\gqq k)\lqq C\ln(1+k)/k$, $k\in \NN$. 
Note that by integrating products of iterated  logarithms this result can be refined 
to upper bounds of order $\frac{\ln^{(j)}(1+n)}{n}$ for any $j\in \NN$, where $\ln^{(1)}(1+n) = \ln(1+n)$ and $\ln^{(j+1)}(1+n) = \ln\big(1+\ln^{(j)}(1+n)\big)$. 
For any such depth $j$, however, this result turns out to be slightly lower than the original asymptotics $\PP(\oO_1\gqq n)=\PP(X_{n-1}>0)\sim\frac{2}{v(n-1)}$, as $n\ra\infty$. This shows that the mean deviation estimates of Lemma~\ref{lem:BC1} yield results arbitrarily close to optimality. 
\end{rem}
}
\noindent 

\begin{rem}
For the probabilities $\PP(X_n\gqq K)$ for $K\gqq 2$, which are not nested events any more due to the 
lack of positive invariance of the set of states $\{0, \dots, K\}$ for any $K\gqq 1$. 
However, there are similar estimates available (see \cite[Theorem 10.1]{Harris64}, \cite{NagVak76}, \cite[Theorem 1]{Wachtel08}).
If $\EE [Y_1^2]<\infty$, then 
\begin{align}\label{eq:Xngeqkineq}
&\PP(X_n=K)\lqq c\min\left\{\frac{1}{n^2},\frac{1}{nK}\right\},\nonumber\\
&\PP(X_n=K_0)>\frac{c_{K_0}}{n^2},\quad \text{for }n\gqq n_{K_0},
\end{align}
where $K_0=\inf\{k>0:\PP(Y_1=K)\neq 0\}$.
Additionally, for all bounded sequences $\frac{K}{n}=\frac{K(n)}{n}$ we have 
\begin{align}\label{eq:Xngeqkineq2}
\PP(X_n\gqq K)=\frac{2}{vn}\exp\left(\frac{-2K}{vn}\right)(1+c_K(n)),
\end{align}
where $c_K(n)\to 0$, as $n\to\infty$. The right-hand side of \eqref{eq:Xngeqkineq2} is not summable and hence there is no  tradeoff obtainable in the sense of Lemma~\ref{lem:Quant BC for e_n}. However, by direct comparison we can obtain the moment estimate from the previous remark: Since $\PP(X_n\gqq K)\lqq\PP(X_n\gqq 1)$, also
\begin{align*}
\oO_K=\sum_{n=1}^\infty\bI\{X_n\gqq K\}\lqq \sum_{n=1}^\infty\bI\{X_n\gqq 1\}=\oO_1,
\end{align*}
and thus, using the monotonicity of $x\mapsto\frac{x}{\ln(1+x)}$, we also have
$$\EE \left[\frac{\oO_K}{\ln(1+\oO_K)}\right]<\infty.$$

Similarly we obtain 
$$\EE \left[\frac{\m_K}{\ln(1+\m_K)}\right]\lqq \EE \left[\frac{\m_1}{\ln(1+\m_1)}\right]
= \EE \left[\frac{\oO_1}{\ln(1+\oO_1)}\right]<\infty.$$
\end{rem} 

\begin{lem}\label{lem:critbranch}
Consider a Galton-Watson process \eqref{def:GaltonWatson} with critical branching $\mm = 1$. 
\begin{enumerate}
\item If $\EE[Y_1^r]<\infty$ for some $r\gqq 3$, and for $\eta\in (0,1)$ we assume 
\[
 v n (\tfrac{r}{2}-1-2\eta) \log(n)\lqq K_r(n) \lqq  v n (\tfrac{r}{2}-1-\eta) \log(n), \qquad n\in \NN.
\]
Then we have the following tradeoff between almost sure excursion size and excursion frequency. 
For all $r> 3+4\eta$ for some $\eta>0$ we have for 
\[
\oO_{K_r} := \sum_{n=n_0}^\infty \ind\{X_n\gqq K_r(n)\}\qquad {\mbox{ and }\qquad \m_{K_r} := \max\{n\gqq n_0~|~ X_n\gqq K_r(n)\}}
\]
that for all $0\lqq p < r-3-4\eta$ {\color{black} there is a constant $c_\infty>0$ s.t.}
\[
\EE[\oO_{K_r}^{1+p}] {\lqq \EE[\m_{K_r}^{1+p}]}\lqq\frac{2 c_\infty}{v} \zeta(r-2-4\eta-p), 
\]
and hence for all $k\in \NN$ 
\[
\PP(\oO_{K_r}\gqq k) {\lqq \PP(\m_{K_r}\gqq k)}\lqq k^{-(1+p)} \cdot \frac{2 c_\infty}{v} \zeta(r-2-4\eta-p).
\]
Moreover by Example~\ref{ex:poly} we have positive constants $c_1$ and $\psi(n_0)$ such that 
\begin{equation}
\PP(\oO_{K_r}\gqq k) {\lqq \PP(\m_{K_r}  \gqq k) \lqq 
}
 c_1  \cdot k^{-(q-1)} \cdot \Big(\ln(k)+\frac{1}{n_0}-\psi(n_0)\Big) \qquad \mbox{for all }k\gqq e^{\frac{1}{q-2} + \psi(n_0)}. 
\end{equation}
\item If $\EE[e^{r Y_1}]< \infty$ for some $r>0$, 
and $K(n) = n^2 \sqrt{\ln(n+2)}$. 
Then we have the following tradeoff. 
For any $p>0$ there is $C = C(v, p)>0$ such that 
\[
\EE[e^{p \oO}]{\lqq \EE[e^{p \m}]} \lqq C(v, p). 
\]
such that by Markov's inequality 
\[
\PP(\oO\gqq k){\lqq \PP(\m\gqq k)}\lqq \inf_{p>0} e^{-pk} C(v, p). 
\]
The same results remain true for $K(n) = \delta n^2$ for some $\delta>0$ and $n\in \NN$ {\color{black} with the additional factor $\delta$ appearing in front of $(1+c_\infty)$.}
\end{enumerate}
\end{lem}

\begin{proof}[\textbf{Proof of Lemma~\ref{lem:critbranch}}: ] 
We start with the proof of item (a). 
{\color{black}
By assumption, we have $K_r(n)\lqq v n (\tfrac{r}{2}-1-\eta) \log(n)$ and
in addition, we have the additional lower bound 
\[
K_r(n) \gqq v n (\tfrac{r}{2}-1-2\eta) \log(n),
\]
and therefore obtain}
\begin{align*}
\PP(X_n\gqq K_r(n))
&=\frac{2}{vn}\exp\left(\frac{-2K_r(n)}{vn}\right)(1+c_{K_r(n)}(n))\\
&\lqq \frac{2}{vn}\exp\left(\frac{-2(v(\tfrac{r}{2}-1-2\eta)n\log(n)}{vn}\right)(1+c_{K_r(n)}(n))\\
&= \frac{2}{v} \frac{1}{n^{r-1-4\eta}} (1+c_{K_r(n)}(n))\lqq \frac{2}{v} \frac{1}{n^{r-1-4\eta}} \underbrace{\sup\{ 1+c_{K_r(n)}(n)~\vert~n\in \NN\}}_{=:c_\infty}.
\end{align*}
{\color{black} The statement is then a consequence of Example \ref{ex:poly}.}
We continue with item (b). 
If $\EE[e^{r Y_1}]< \infty$ for some $r>0$, 
then \eqref{eq:Xngeqkineq2} even holds for $K(n)=o(n^2\log(n))$.  
First we consider the case of $K(n) = n^2 \sqrt{\ln(n+2)}$ and obtain  
\begin{align*}
\PP(X_n\gqq K(n))
&=\frac{2}{vn}\exp\left(\frac{-2 K(n)}{vn}\right)(1+c_{K(n)}(n))\\
&=\frac{2}{vn}\exp\left(\frac{-2 n \sqrt{\ln(n+2)}}{v}\right)(1+c_{K(n)}(n))\\
&\lqq \frac{2}{v}\exp\left(\frac{-2 n \sqrt{\ln(n+2)}}{v}-\ln(n)\right)(1+c_\infty).
\end{align*}
The latter is smaller than $\frac{2(1+c_\infty)}{v}\exp\left(-n \frac{2}{v}\right)$.
In case of $K(n) = \delta n^2$ for some $\delta>0$ we see that 
\begin{align*}\label{eq:Xngeqkineq2}
\PP(X_n\gqq K(n))
 &=\frac{2}{vn}\exp\left(\frac{-2 K(n)}{vn}\right)(1+c_{K(n)}(n))
\lqq \frac{2}{vn}\exp\left(-n \frac{2\delta }{v}\right)(1+c_\infty).
\end{align*}
Since the last expression is bounded by $\frac{2(1+c_\infty)}{v}\exp\left(-n \frac{2\delta }{v}\right)$, the assertion for both cases follows from Example \ref{ex:exp}.
\end{proof}

\bigskip

\section{\textbf{Other applications: Freedman's maximal inequality and the law of the iterated logarithm}}\label{s:other} 
\noindent Several situations do not fit neatly a tradeoff relation as given in Lemma~\ref{lem:Quant BC for e_n} between an asymptotic a.s.~rate and the mean deviation frequency. In this case we go back to Lemma~\ref{lem:BC1}.
\medskip
\begin{rem}
A word of caution: There is considerable literature \cite{Be27, Be62, Dey09, DvZ01, Fr75, Hae84, Kha09, Pi94, PKL07, vdG95} on maximal inequalities for martingales, all of which 
ultimativley go back to Freedman's original result, which we cite below. 
We refer to the introduction of \cite{FGL15} for an overview and systematic comparison. 
We also refer to inequalities (1.5)-(1.11), (1.13-1.18) in \cite{FGL15}. 
Such type of inequalities are typically probability estimates of events, which contain the martingale \textit{and} its quadratic variation simultaneously.
This often complicates the situation, since Lemma~\ref{lem:BC1} only covers the event summation of event probabilities with respect to the same probability measure, that is, conditioning w.r.t.~parameter dependent events do not fall under its scope. 
\end{rem}
\bigskip 
\begin{rem}\label{rem:boundedqv}
For a given martingale $(X_n)_{n\in \NN_0}$ it is well known \cite[12.13. Theorem]{Wi91} that in case of bounded martingale difference sequences 
\[
\{\lim_{n\ra\infty} X_n \in \RR\} = \{\lgl X\rgl_\infty< \infty\} \qquad \mbox{ up to a }\PP\mbox{-null set.}
\] 
Hence it is natural to ask for the supremum of $X_n$ conditioned to the event $\{\lgl M\rgl_\infty \lqq v\}$. This is a consequence of Freedman's inequality \cite{Fr75}.
\end{rem}
\medskip
\begin{thm}[\cite{Fr75} Freedman]\label{thm:Freedman}
Consider a martingale $(X_k)_{k\in\NN_0}$ and suppose there is some $\rho>0$ such that $|\Delta X_k|\lqq \rho$ a.s.~for all $k\in \NN$. Then we have 
for all $u, v > 0$
\begin{equation}\label{e:Friedman}
\PP\Big(\bigcup_{k\gqq 0} \{X_k \gqq u, \langle X \rangle_k\lqq v\}\Big)\lqq \exp\Big(-\frac{u^2}{2 (v^2 + \rho u)}\Big).
\end{equation}
Moreover, for any $v, \rho>0$ fixed and any positive, nondecreasing sequences $a = (a_n)_{n\gqq n_0},u= (u_n)_{n\gqq n_0}$ with $u_n\nearrow \infty$ as $n\ra\infty$ such that 
\[
K_{v, x, n_0} := \sum_{n=n_0}^\infty a_n \exp\Big(-\frac{u_n^2}{2 (v^2 + \rho u_n)}\Big)< \infty, 
\]
we have 
\[
\EE[\sS_{a, n_0}(\oO_{u, n_0})] \lqq \frac{K_{v, x, n_0}}{\PP(\lgl M\rgl_\infty \lqq v)}, 
\]
where 
\[
\oO_{u, n_0} := \sum_{n=n_0}^\infty \ind_{A_n}, \qquad \mbox{ and }\qquad   A_n := \{\sup_{k\gqq n_0} X_k \gqq u_n\}, \qquad \mbox{ for } n\gqq n_0, 
\]
and $\sS_{a, n_0}(N)$ is defined in \eqref{def:Sa}. Due the nestedness of $(A_n)_{n\gqq n_0}$ 
we have that $\oO_{u, n_0} = \m_{u,n_0}$ defined as before. 
\end{thm}
\begin{proof}[\textbf{Proof of Theorem~\ref{thm:Freedman}:} ] Since $\lgl X\rgl_k$ is positive and nondecreasing we have that $\{\lgl X\rgl_\infty \lqq v \} \subseteq \{\lgl X\rgl_k\lqq v\}$ for any $v>0$ and $k\in \NN_0$. Hence by {\cite{Fr75}} we obtain for any $u\gqq 0$ that 
\[
\PP\Big(\sup_{k\gqq 0} X_k \gqq u, \langle X\rangle_\infty\lqq v\Big)
\lqq \PP\Big(\bigcup_{k\gqq 0} \{X_k \gqq u, \langle X \rangle_k\lqq v\}\Big)
\lqq \exp\Big(-\frac{u^2}{2 (v^2 + \rho u)}\Big).
\]
Consequently for any $u, v>0$ we have 
\begin{align*}
\PP\Big(\sup_{k\gqq 0} X_k \gqq u~|~ \langle X\rangle_\infty\lqq v\Big) 
\lqq \frac{\exp\Big(-\frac{u^2}{2 (v^2 + \e u)}\Big)}{\PP(\lgl M\rgl_\infty \lqq v)}, 
\end{align*}
if $\PP(\lgl M\rgl_\infty \lqq v) >0$, and $0$ otherwise. 
Now consider for some positive and divergent sequence $(u_n)_{n\gqq n_0}$ the sequence 
of nested events $(A_n)_{n\in \NN}$ given by 
\[
A_n := \{\sup_{k\gqq 0} X_k \gqq u_n\}, \qquad n\gqq n_0, \qquad \mbox{ and }\qquad 
\oO_{u, n_0}  = \sum_{n=n_0}^\infty \ind_{A_n}.  
\]
With the help of Lemma~\ref{lem:BC1} we obtain the desired result. 
\end{proof}

\begin{exm} Consider a martingale as in Theorem~\ref{thm:Freedman} with $\rho< \frac{1}{2}$ and $v >0$. Then for $u_n = \ln(n+1)$ we have 
\begin{align*}
\PP(\sup_{k\gqq 0} X_k \gqq u_n~|~ \langle X\rangle_\infty\lqq v) 
&\lqq \frac{\exp\Big(-\frac{u_n^2}{2 v^2 + 2\rho u_n}\Big)}{\PP(\lgl X\rgl_\infty \lqq v)}
= \frac{n^{-\frac{\ln(n+1)}{2 (v^2 + \rho \ln(n+1))}}}{\PP(\lgl M\rgl_\infty \lqq v)}= \frac{n^{-\frac{1}{\frac{v^2}{\ln(n+1)} + 2 \rho}}}{\PP(\lgl M\rgl_\infty \lqq v)}.
\end{align*}
The right-hand side converges of order $n^{-\frac{1}{2\rho}}$. Hence only for $\rho <\frac{1}{2}$ {we} obtain summability and a.s.~MDF convergence. 
For $u_n = \ln^p(n+1)$, $n\gqq 0$, for some $p>1$ we have 
\begin{align*}
\PP(\sup_{k\gqq 0} X_k \gqq u_n~|~ \langle X\rangle_\infty\lqq v) 
&\lqq  \frac{e^{-\frac{\ln^{2p}(n+1)}{2 (v^2 + \rho \ln^{p}(n+1))}}}{\PP(\lgl X\rgl_\infty \lqq v)}=  \frac{e^{-\frac{\ln^{p}(n+1)}{2 (\frac{v^2}{\ln^{p}(n+1)} + \rho)}}}{\PP(\lgl M\rgl_\infty \lqq v)} \lqq \frac{C_{v, \rho}}{\PP(\lgl M\rgl_\infty \lqq v)} e^{-\frac{\ln^{p}(n+1)}{2 \rho}}.  
\end{align*}
This yields an intermediate (between polynomial and exponential) decay regime of the probabilities and the respective MDF integrability. Note that for $p>1$ the integrability does not depend on $\rho$. 
\end{exm}

\begin{rem}[Law of the iterated logarithm]
We also refer to the literature \cite{Eg87,Fi92, HH76,Hu90,To79,Vo91} on the law of the iterated logarithm for martingales and random walks, which falls in the scope of our results, and which can be quantified further. For preliminary results in this direction we refer to \cite[Theorem 12]{EstraHoeg22} and \cite[Section 2.4]{HS23}. 
\end{rem}

\bigskip
\appendix
\section{\textbf{The relation between a.s.~MDF convergence
and the Ky Fan metric}}\label{s:KyFan}

\noindent 
As mentioned in item (b) at the beginning of the introduction it is well-established that a.s.~convergence does not induce a topology on the space of (equivalence classes) of random variables. 
In this section we show, that taking into account the overlap statistic is the missing piece 
of information in order to obtain topological convergence in the following sense. For a summable sequence of rates of convergence $\epsilon = (\e_n)_{n\in \NN}$, $\e_n>0$, $n\in \NN$ and $\e_n\ra 0$, the convergence in the Ky Fan metric with rate of at least $\epsilon$ is equivalent to complete convergence with the rate $\epsilon$ and a suitable integrability $\sS_{a}$ of the overlap statistics for some suitable sequence $a = (a_n)_{n\in \NN}$ given in \eqref{def:Sa}. In other words, in case of summable rates of convergence, the convergence in probability and the respective a.s.~convergence with a certain mean deviation frequencies are qualitatively equivalent. 

\begin{defn}
Let $X$ and $Y$ be two random variables on a given probability space $(\Omega, \aA, \PP)$. 
Then, their distance in the Ky Fan metric (see \cite[p.330]{Du04}) is given by 
$$\mathrm{d}_{\mathrm{KF}}(X, Y) := \inf \{ \e > 0 ~|~ \PP(|X-Y|> \e) \lqq \e\}.$$ 
\end{defn}
\noindent It is well-known \cite[Subsection 1.1.5]{App04} that $\mathrm{d}_{\mathrm{KF}}$ metrizes the convergence in probability on the space $L^0 = \lL^0 / \sim$, where $X\sim Y$ is defined by $X-Y = 0$ $\PP$-a.s.~We start with the following simple observations. 

\begin{lem}\label{lem:rate}
We consider a nonnegative random variable with 
\begin{enumerate}
 \item a continuous distribution function $F_X$.  Then 
\[
\mathrm{d}_{\mathrm{KF}}(X, 0) = \inf_{\eta>0} \{\eta>0~|~\eta^{-1}\cdot \PP(X>\eta)\lqq 1\} = (g^{-1}(1))^{-1}\cdot \PP(X>g^{-1}(1)),   
\]
where $g(\eta) = \eta^{-1}\PP(X>\eta)$ and $g^{-1}$ is the inverse of $g$. 

\item a right-continuous general distribution function $F_X$ (the general case). Then 
\[
\mathrm{d}_{\mathrm{KF}}(X, 0) = \inf_{\eta>0} \{\eta>0~|~\eta^{-1}\cdot \PP(X>\eta)\lqq 1\} \lqq (g^{\leftarrow 1}(1))^{-1}\cdot \PP(X>g^{\leftarrow 1}(1)),   
\]
where $g(\eta) = \eta^{-1}\PP(X>\eta)$ and $g^{\leftarrow 1}$ is the right inverse of $g$. 
\end{enumerate}
\end{lem}

\begin{proof}
We show item (a). Note that the function $(0, \infty) \ra \eta\mapsto g(\eta) = \eta^{-1}\cdot \PP(X>\eta)\in (0, \infty)$ is strictly decreasing, as the product of a nonnegative strictly decreasing and a nonnegative nonincreasing function. Furthermore $\lim_{\eta\ra 0} g(\eta) = \infty$ and $\lim_{\eta\ra\infty} g(\eta) = 0$. Due to the continuity $g$ is a bijection from $\mbox{int}(\mbox{supp}(g))$ to $(0,\infty)$. In particular, $\eta_* := g^{-1}(1)$ is unique. 
For (b) it is enough to see that $\eta_* := (g^{\leftarrow 1}(1))^{-1}$ yields an upper bound. 
 \end{proof}

 \begin{rem}\label{rem:notopology}
Recall that the almost sure convergence of random variables does not define a topology on the space of random variables $L^0$. 
This result remains true if we replace $\lL^0$ by the quotient space $L^0 = \lL^0 / \sim_\PP$, where $X\sim_\PP Y$ iff $X-Y = 0$ $\PP$-a.s.~

The classical counterexample given e.g.~in \cite{Ord66} uses the property that from any sequence converging in probability but not $\PP$-a.s., (as, for example, a sequence of independent Bernoulli random variables $(X_n)_{n\in \NN}$ with $X_n \stackrel{d}{=} B_\frac{1}{n}$) one may always extract a subsequence that converges almost surely. Relevant to complete convergence, one may also always extract a completely convergent subsequence out of a sequence converging in probability. The contradiction consists in the topological fact which states that for a convergence resulting from a topology, if for any subsequence $(x_{n_k})_{k\gqq 0}$ of a sequence $(x_n)_{n\gqq 0}$ one can always find a subsubsequence converging to $x$ then also the original sequence $(x_n)_{n\gqq 0}$ converges to $x$. What is less known is that one may also use the 'dual version' of this fact (\cite[Ex.1.7.18]{En89}):
``If a sequence $(x_n)_n{\gqq 0}$ in a topological space $E$ does not converge to an element $x\in E$, then there is a subsequence $(x_{n_k})_{k\gqq 0}$ such that no subsequence of $(x_{n_k})_{k\gqq 0}$ converges to $x$.''
\end{rem}

\noindent While a.s.~MDF convergence does not define a topology, we can often transfer useful quantitative information between the rates of convergence of the Ky Fan metric and a.s.~MDF convergence.

\begin{cor}[Summable convergence rate in Ky Fan metric implies complete~convergence with MDF]\label{cor:KFMDF}    
    Assume there is some $n_0\in \NN$ such that the following holds true. 
    There is a sequence $(X_n)_{n\gqq n_0}$ of random variables and 
    a sequence $\epsilon = (\e_n)_{n\gqq n_0}$ of nonincreasing, summable, positive numbers. 
    Furthermore, there is a random variable $X$ such that 
    \[
    \mathrm{d}_{\mathrm{KF}}(X_n, X) \lqq \e_n, \quad \mbox{ for all }n\gqq n_0.  
    \]
    {\color{black}
    Then it follows that
	  \begin{align*}
	  \sum_{n=0}^\infty \PP(X_n-X)<\infty,
\end{align*}	    
   and the }the asymptotic error tolerance 
    \[
    \limsup_{n\ra\infty} |X_n - X| \cdot \e_n^{-1}\lqq 1 \qquad \PP\mbox{-a.s.}
    \] 
    with mean error incidence
    $$\EE[\sS_{\theta, n_0}(\oO_{\epsilon, n_0})] {\lqq \EE[\sS_{\theta, n_0}(\m_{\epsilon, n_0})]} \lqq \sum_{n=n_0}^{\infty} \frac{1}{n \log^{1+\theta}(n+1)}, \quad \text{for any } \theta > 0,$$
    where by \eqref{def:Sa} we have $\sS_{\theta, n_0}(0) = 0$ and 
    \[
    \sS_{\theta, n_0}(N) = \sum_{n=n_0}^{n_0+N-1} \Big(n \log^{1+\theta}(n+1) \sum_{m=n}^{\infty} \e_{m}\Big)^{-1}.
    \]
\end{cor}

\begin{proof}
    Since the sequence $(\e_n)_{n\gqq n_0}$ is summable, the usual first Borel-Cantelli lemma implies 
    \[
    \limsup_{n\ra\infty} |X_n-X|\cdot \e_n^{-1}\lqq 1, \qquad \PP\mbox{-a.s.} 
    \]
    Moreover for all $\theta > 0$ we have that 
    \begin{align*}
        &K_{\theta} := \sum_{n=1}^{\infty} a_n \sum_{m=n}^{\infty} \e_{m} < \infty, \quad \text{ for any }
        a_n \lqq \left(n \log^{1+\theta}(n+1) \sum_{m=n}^{\infty} \e_{m}\right)^{-1},
    \end{align*}
    and thus, that $K(a, \e) = \sum_{n=1}^{\infty} a_n  \sum_{m=n}^{\infty} p(\e_m, m) \lqq K_{\theta} < \infty$. From there, Lemma \ref{lem:Quant BC for e_n} implies that $\EE[\sS_{a, n_0}(\oO_{\epsilon, n_0})] {\lqq \EE[\sS_{\theta, n_0}(\m_{\epsilon, n_0})]}\lqq K(a, \e)$, where $\sS_{a, n_0}$ is defined in \eqref{def:Sa}. 

\end{proof}

\begin{rem}
The quantity $\ln^{1+\theta}(n+1)$ in the assertion of the above corollary can be generalized:
Recall that by integral comparison for any $m\in \NN$ fixed the sequence $n \mapsto (n \prod_{i=0}^{m} \ln^{\circ i}(n+1))^{-1}$ is not summable, while 
for any $\theta>0$ the respective sequence $n\mapsto (n \prod_{i=0}^{m-1} \ln^{\circ i}(n+1) (\ln^{\circ m}(n+1))^{1+\theta})^{-1})^{-1}$ is finitely summable. 
\end{rem}

\begin{cor}[$\PP$-a.s.~convergence with given mean deviation frequency 
provides an upper bound for the Ky Fan convergence]\label{cor:MDFKF}
    Let $(X_n)_{n\gqq 0}$ be a sequence of random variables, and suppose there is a positive nondecreasing $(a_n)_{n\gqq 0}$ and a nonincreasing summable sequence of positive numbers $(\e_n)_{n\gqq 0}$ such that 
    \begin{equation}\label{e:Gewichtsbalance} 
    \e_n \cdot \sum_{i=1}^n a_i \gqq 1 
    \end{equation}
and 
        \begin{equation}\label{e:sumabilidad}
\sum_{n=1}^{\infty} a_n \sum_{m=n}^{\infty}\PP (|X_m - X| > \e_m) < \infty, \quad \text{ for some random variable } X.     
    \end{equation}
    Then for all $n\in \NN$ we have 
    \begin{align*}
    \mathrm{d}_{\mathrm{KF}}(X_n, X)\lqq \e_n.\\
    \end{align*}
\end{cor}

\begin{proof}
            Kronecker's lemma \cite[(12.7)]{Wi91} applied to \eqref{e:sumabilidad} (for the sequences $\bigg(\frac{2}{\sum_{m\gqq n}\PP(X_m-X)}\bigg)_{n\gqq 1}, (a_n)_{n\gqq 0}$) yields 
    \[
    \sum_{\ell=n}^{\infty} \PP\left(|X_\ell - X| > \e_\ell\right) \cdot \sum_{i=1}^n a_i \ra 0\qquad \mbox{ as } n\ra\infty. 
    \]
Therefore there is some $n_0\in \NN$ such that for all $n\gqq n_0$ we have  
\[
\PP\left(|X_n - X| > \e_\ell\right) \lqq \sum_{\ell=n}^{\infty} \PP\left(|X_\ell - X| > \e_\ell\right) \lqq (\sum_{i=1}^n a_i)^{-1}.      
\]
Hypothesis \eqref{e:Gewichtsbalance} then implies the desired result 
\[
\PP\left(|X_n - X| > \e_n\right)\lqq \e_n.
\]
\end{proof}

\bigskip

\section{\textbf{Optimal tail decay rates in case of Weibull type moments}}\label{s:optimal}

 \noindent This section gives the details of the tail optimization of $\oO_{\epsilon, n_0}$ 
 and $\m_{\epsilon, n_0}$ in Example~\ref{ex:Weibull}. 

\begin{lem}\label{lem:Weibull}
{\color{black}For any $b, \alpha\in (0,1)$, $c>0$ and $n_0\in \NN$ there are (explicitly known) positive constants $d(n_0),D(n_0) \in \NN$ such that for $k\gqq 2$ we have 
\begin{align*}
\inf_{p\in (0,1)} b^{p (k-1)^\alpha} \sum_{n=n_0}^\infty c\bigg (1+\frac{1+\frac{\frac{1}{\alpha}-1}{|\ln(b)|}}{\alpha |\ln(b)|} n^{1-\alpha}\bigg)b^{(1-p) n^\alpha}
\lqq (d+D(k-1)^{2-\alpha})b^{(k-1)^\alpha}.
\end{align*}}
\end{lem}

\begin{proof}
To estimate the desired infimum, we calculate
\begin{align}\label{eq:bKp1}
&b^{p (k-1)^\alpha}K(p, \alpha)= b^{p (k-1)^\alpha}\sum_{n=n_0}^\infty c\bigg (1+\frac{1+\frac{\frac{1}{\alpha}-1}{|\ln(b)|}}{\alpha |\ln(b)|} n^{1-\alpha}\bigg)b^{(1-p) n^\alpha}\nonumber\\
 &=b^{p (k-1)^\alpha}c\bigg (1+\frac{1+\frac{\frac{1}{\alpha}-1}{|\ln(b)|}}{\alpha |\ln(b)|} n_0^{1-\alpha}\bigg)b^{(1-p) n_0^\alpha}+b^{p (k-1)^\alpha}\sum_{n=n_0+1}^\infty c\bigg (1+\frac{1+\frac{\frac{1}{\alpha}-1}{|\ln(b)|}}{\alpha |\ln(b)|} n^{1-\alpha}\bigg)b^{(1-p) n^\alpha}\nonumber\\
 &\lqq b^{p (k-1)^\alpha}c\bigg (1+\frac{1+\frac{\frac{1}{\alpha}-1}{|\ln(b)|}}{\alpha |\ln(b)|} n_0^{1-\alpha}\bigg)b^{(1-p) n_0^\alpha}+b^{p (k-1)^\alpha}\sum_{n=n_0+1}^\infty c\bigg (1+\frac{1+\frac{\frac{1}{\alpha}-1}{|\ln(b)|}}{\alpha |\ln(b)|}\bigg) n^{1-\alpha}b^{(1-p) n^\alpha}\nonumber\\
&\lqq  b^{p (k-1)^\alpha}c\bigg (1+\frac{1+\frac{\frac{1}{\alpha}-1}{|\ln(b)|}}{\alpha |\ln(b)|} n_0^{1-\alpha}\bigg)b^{(1-p) n_0^\alpha}+b^{p (k-1)^\alpha}C\int_{n_0}^\infty  e^{-|\ln(b)|(1-p) x^\alpha}x^{1-\alpha}dx,
\end{align}
where $C:=c\bigg(1+\frac{1+\frac{\frac{1}{\alpha}-1}{|\ln(b)|}}{\alpha |\ln(b)|}\bigg) $. Substituting $t = |\ln(b)|(1-p) x^\alpha$, the integral $\int_{n_0}^\infty  e^{-|\ln(b)|(1-p) x^\alpha}x^{1-\alpha}dx$ becomes
\begin{align*}
&\int_{|\ln(b)|(1-p)n_0^\alpha}^\infty  e^{-t}\frac{1}{\alpha|\ln(b)|(1-p)}\Big(\frac{t}{|\ln(b)|(1-p)}\Big)^{\frac{1}{\alpha}-1}\Big(\frac{t}{|\ln(b)|(1-p)}\Big)^{\frac{1}{\alpha}-1}dt\\
&=\frac{1}{\alpha(|\ln(b)|(1-p))^{\frac{2}{\alpha}-1}}\int_{|\ln(b)|(1-p)n_0^\alpha}^\infty  e^{-t}t^{\frac{2}{\alpha}-2}dt.
\end{align*}
{\color{black} By the  integral substitution, \eqref{eq:bKp1} turns to
\begin{align}\label{eq:Kpn0der}
b^{p (k-1)^\alpha}c\bigg (1+\frac{1+\frac{\frac{1}{\alpha}-1}{|\ln(b)|}}{\alpha |\ln(b)|} n_0^{1-\alpha}\bigg)b^{(1-p) n_0^\alpha}+\frac{C\Gamma\big(\frac{2}{\alpha}-1,|\ln(b)|(1-p)n_0^\alpha\big)}{\alpha(|\ln(b)|(1-p))^{\frac{2}{\alpha}-1}}b^{p (k-1)^\alpha},
\end{align}
where $\Gamma\big(x,a\big)$ denotes the upper incomplete Gamma function with lower limit $a$. Differentiating this term w.r.t.~$p$ and calculating the appearing zero $p^*$ is quite tedious. The resulting value for $p^*$ however is close to $1-\frac{1}{(k-1)^\alpha}$ (for high values of $n_0$). Subsequently, we will set $p$ to this value. this leaves us with
\begin{align*}
\left(c\bigg (1+\frac{1+\frac{\frac{1}{\alpha}-1}{|\ln(b)|}}{\alpha |\ln(b)|} n_0^{1-\alpha}\bigg)b^{\big(\frac{n_0}{k-1}\big)^\alpha-1}+\frac{C\Gamma\Big(\frac{2}{\alpha}-1,|\ln(b)|\big(\frac{n_0}{k-1}\big)^\alpha\Big)}{b\alpha|\ln(b)|^{\frac{2}{\alpha}-1}}(k-1)^{2-\alpha}\right)b^{(k-1)^\alpha}.
\end{align*}
}
\end{proof}

\bigskip    
\section*{\textbf{Acknowledgments}} 

\noindent 
Parts of this article were written when A.S.~was enrolled in the JYU Visiting Fellow program of the University of Jyv\"askyl\"a, Finland. 
The research of M.A.H.~has been supported by 
the project ``Mean deviation frequencies and the cutoff phenomenon'' (INV-2023-162-2850) 
of the School of Sciences (Facultad de Ciencias) at Universidad de los Andes, Bogot\'a, Colombia. 
M.A.H.~thanks E.~Hausenblas for the hospitality during a JESH exchange project of the Austrian Academy of Sciences exchange at Montanuniversit\"at Leoben, Austria, where this project was started. A.S.~and M.A.H.~are grateful to E.~Hausenblas and the work group of Applied Mathematics at MU Leoben for the scientific environment which made this collaboration possible. 

\section*{\textbf{Declarations}}
The authors declare no conflict of interest. Moreover they declare that the journal's ethical, diversity, environmental and AI policies have been respected. They all consented for publication.


\begin{thebibliography}{99}

\bibitem{Als90}
G. Alsmeyer.
\newblock Convergence rates in the law of large numbers for martingales.
\newblock Stochastic Process. Appl. 36 (1990), no. 2, 181-194.

\bibitem{App04} 
D. Applebaum, 
\newblock {\it L\'evy processes and stochastic calculus}. 
\newblock Cambridge University Press, Cambridge, 2004.

\bibitem{AudleyJonckheere56}
{Audley, R. J., Jonckheere, A. R.}, 1956.
\newblock The Statistical Analysis of the Learning Process.
\newblock The British Journal of Statistical Psychology. https://doi.org/10.1111/j.2044-8317.1956.tb00176.x

\bibitem{Az67}
Azuma, K.~1967. 
\newblock Weighted Sums of Certain Dependent Random Variables. 
\newblock T\^ohoku Mathematical Journal. 19 (3): 357--367. 

\bibitem{BP85}
Bagchi, A., and Pal, A. K., 1985.
\newblock Asymptotic normality in the generalized Polya-Eggenberger urn model, with an application to computer data structures.
\newblock Siam J. Alg. Disc. Math., vol. 6, p. 394-405.

\bibitem{BBA99}
Banerjee, A., Burlina, P., and Alajaji, F.,
\newblock {\it Image segmentation and labeling using the Polya urn model,} 
\newblock in: IEEE Transactions on Image Processing, vol. 8, no. 9, pp. 1243-1253, Sept. 1999, 

\bibitem{BK65}
Baum,~L.E., Katz,~M.~1965.
\newblock Convergence rates in the law of large numbers.
\newblock Trans. Amer. Math. Soc., 120 (1965), pp. 108-123.

\bibitem{Be27}
Bernstein,~S.,~1927.
\newblock Theorem of Probability. 
\newblock Moscow.

\bibitem{Be62}
Bennett,~G.,~1962.
\newblock Probabilities inequalities for the sum of independent random variables. 
\newblock J. Amer. Statist. Assoc. \textbf{57}, No. 297, 33-45. 


\bibitem{Bi99}
P. ~Billingsley, 
\newblock Convergence of probability measures,
\newblock 2nd edn., Wiley, New York, 1999.

\bibitem{Bo1909}
Borel,~E.,~1909.
\newblock Les probabilit{\'e}s d{\'e}nombrables et leurs applications arithm{\'e}tiques. 
\newblock Rend. Circ. Mat. Palermo (2) 27, 247--271.

\bibitem{Ca1917} 
Cantelli,~F.~P.,~1917.  
\newblock Sulla probabilit{\`a} come limite della frequenza. 
\newblock Atti Accad. Naz. Lincei 26:1, 39--45

\bibitem{CNWDD17}
{Caron, F., Neiswanger, W., Wood, F., Doucet, A., \& Davy, M. (2017).}
\newblock Generalized P\'olya urn for time-varying Pitman-Yor processes. 
\newblock Journal of Machine Learning Research, 18(27).

\bibitem{Ch12}
T.~K.~Chandra, 
\newblock The Borel-Cantelli Lemma, 
\newblock SpringerBriefs in Statistics, Vol.2, 
Chap.~2,~51--62, ~2012.

\bibitem{CE51}
Chung,~K.~L., Erd\H{o}s,~P.,~1951. 
\newblock On the application of the Borel-Cantelli Lemma. 
\newblock Trans. Am. Math. Soc. 72 (1): 179--186. 

\bibitem{DedeckerMerlevede07}
J. Dedecker and F. Merlev\'ede.
\newblock Convergence rates in the law of large numbers for Banach-valued dependent
variable.
\newblock Teor. Veroyatn. Primen. 52 (2007), no. 3, 562-587.

\bibitem{DS72}
Denning, P., J., Schwartz, S., C., 1972.
\newblock Properties of the working-set model.
\newblock Commun. ACM, 15, p. 191-198. 

\bibitem{Dey09}
Deylon,~B.,~2009.
\newblock Exponential inequalities for sums of weakly dependent variables.
\newblock Electronic J. Probab. \textbf{14}, No. 28, 752-779. 

\bibitem{DZ98} 
A. Dembo, O. Zeitouni,
\newblock Large deviation techniques and applications,   
\newblock 2nd ed., Springer, Appl. of Math., vol. 38., ~1998.

\bibitem{Dharma68}
Dharmadhikari,~S.~W., Fabian,~V., Jogdeo,~K.,~1968. 
\newblock Bounds on the Moments of Martingales.
\newblock Ann. Math. Statist. 39(5): 1719-1723.

\bibitem{Doob53}
J.L. Doob, 
\newblock Stochastic processes,
\newblock John Wiley \& Sons, Inc.,~1953. 


\bibitem{Du04}
R.M. Dudley, 
\newblock Real analysis and probability, 
\newblock Cambridge University Press, Cambridge, ~2004.


\bibitem{DvZ01}
Dzhaparidze,K.,van Zanten, J.H.,~2001. 
\newblock On Bernstein-type inequalities for martingales. 
\newblock Stochastic Process. Appl. \textbf{93}, 109-117. 



\bibitem{Eg87}
Egorov, G., 1987. 
\newblock On the strong law of large numbers and the law of the iterated logharithm for martingales and sums of independent random variables. 
\newblock Theory Propab. Appl. \textbf{35} (4), 653-666. 


\bibitem{En89}
R. Engelking, 
\newblock General topology, 
\newblock Helderman Verlag, 1989. 


\bibitem{Er49}
Erd{\"o}s, P.,~1949. 
\newblock On a theorem of Hsu and Robbins, 
\newblock Ann. Math. Statist. 20 (1949), 286-291.

\bibitem{EstraHoeg22}
Estrada,~L., H\"ogele,~M.A.,~2022.
\newblock  Moment estimates in the first Borel-Cantelli Lemma with applications to mean deviation frequencies.
\newblock Statistics and Probability Letters, 2022.

\bibitem{Et81}
Etemadi,~N.,~1981. 
\newblock An elementary proof of the strong law of large numbers. 
\newblock Z. Wahrsch. theor. Verw. Geb. 55(1):119-122.


\bibitem{Ewens69}
W. J. Ewens, 
\newblock Population Genetics,
\newblock London: Methuen, 1969. 

\bibitem{Fagin75}
Fagin, R., 1975.
\newblock Asymptotic miss ratios over independent references.
\newblock IBM Research Report, Rc5415, Yorktown Heights, N.Y.
\newblock (See also Not. Am. Math. Soc., Nov. 1975, vol 22(7), A-715)

\bibitem{FGL15}
Fan,~X.,~Gama,~I.,~Liu,~Q.,~2015. 
\newblock Exponential inequalities for martingales with applications 
\newblock Electron. J. Probab. \textbf{20} (2015), no. 1, 1-22. 

\bibitem{Fi92}
Fisher, E., 1992. 
\newblock On the law of the iterated logarithm for martingales. 
\newblock The Annals of Probability, \textbf{20} (2) 675-680.

\bibitem{Fr17}
{Franchini, S.} (2017). 
\newblock Large deviations for generalized polya urns with arbitrary urn function. 
\newblock Stochastic Processes and their Applications, 127(10), 3372-3411.

\bibitem{Fr75}
Freedman,~D.,1975. 
\newblock On tail probabilities for martingales. 
\newblock Ann. Probab., Vol. 3, No. 6, 100--118.

\bibitem{FS04} 
H. F\"ollmer, A. Schied,
\newblock Stochastic Finance: An Introduction in Discrete Time,  
\newblock Berlin, Boston: De Gruyter, 2004. 


\bibitem{Friedman49}
Friedman, B., 1949.
\newblock A simple urn model.
\newblock Pure Appl. Math., vol 2, p. 59-70.

\bibitem{G79}
Gerber,~H.~U.,~1979.  
\newblock A proof of the Schuette-Nesbitt formula for dependent events. 
\newblock Act. Res. Clearing House, 1: 9--10.

\bibitem{Giraudo18}
Giraudo, D.,
\newblock Deviation inequalities for Banach space valued martingales differences sequences and random fields.
\newblock  ESAIM: PS, Volume 23, 2019, 922-946.

\bibitem{Gouet89}
Gouet, R., 1989.
\newblock A martingale approach to strong convergence in a generalized Polya-Eggenberger urn model.
\newblock Statist. Prob. Lett., vol. 8, p. 225-228.

\bibitem{Gouet93}
Gouet, R., 1993.
\newblock Martingale functional central limit theorems for a generalized Polya urn.
\newblock Ann. Prob., vol. 21, p. 1624-1639.

\bibitem{Gouet97}
Gouet, R., 1997.
\newblock Strong convergence of proportions in a multicolor Polya urn.
\newblock Journal of Applied Probability, vol. 34, p. 426-435.

\bibitem{HL12}
{S. Hao, Q. Liu} 
\newblock {\it Baum-Katz type theorems for martingale arrays}
\newblock C. R. Acad. Sci. Paris, Ser. I 350 (2012) 91--96

\bibitem{Harkness70}
Harknes, W., L., 1970.
\newblock The classical occupancy problem revisited.
\newblock Random Counts in Physical Sciences, G. P. Patil (Ed.), University Park, Pa.: Penn State University Press, pp. 107-126.

\bibitem{Hae84}
Haeusler, E.,~1980.
\newblock An exact rate of convergence in the functional central limit theorems for special martingale difference arrays.
\newblock Probab. Theory Reltat. Fields. \textbf{65}, No. 4, 523-534. 

\bibitem{HH76}
Hall, P.G., Heyde, C.C.,~1976. 
\newblock On a uniﬁed approach to the law of the iterated logarithm for martingales. 
\newblock Bulletin of the Australian Mathematical Society, 14, pp 435-­447 

\bibitem{Hao13}
S. Hao. 
\newblock Convergence rates in the law of large numbers for arrays of Banach valued martingale differences.
\newblock Abstr. Appl. Anal. (2013), Art. ID 715054, 26.

\bibitem{HaoLiu14}
S. Hao and Q. Liu. 
\newblock Convergence rates in the law of large numbers for arrays of martingale
differences.
\newblock J. Math. Anal. Appl. 417 (2014), no. 2, 733-773.

\bibitem{Harris64}
Harris,~T.~E.,~1964.
\newblock The Theory of Branching Processes.
\newblock United States Airforce Project Rand, R-381-PR.

\bibitem{Hi83}
Hill,~T.~P.,~1983.
\newblock A stronger form of the Borel-Cantelli lemma. 
\newblock Illinois Journal of Mathematics, Vol 27 (2).

\bibitem{Hoe63}
Hoeffding,~W.,~1963. 
\newblock Probability inequalities for sums of bounded random variables. 
\newblock J. of the Am.~Stat.~Ass. 58 (301):~13-30.

\bibitem{HS23}
H\"ogele, M.A., Steinicke A., 
\newblock Deviation frequencies of Brownian path property approximations.
\newblock \url{https://arxiv.org/abs/2302.04115}

\bibitem{HH08}
Hoorfar, A., Hassani, M.,~2008.
\newblock Inequalities on the Lambert W Function and Hyperpower Function.  
\newblock JIPAM, Theorem 2.7, p.7, volume 9, issue 2, article 51. 

\bibitem{HR47}
Hsu,~P.L., Robbins,~H.,~1947.
\newblock Complete convergence and the law of large numbers.
\newblock Proc. Natl. Acad. Sci. USA, 33 (1947), pp. 25-31.

\bibitem{Hu90}
Huggins, R. M., 1990. 
\newblock The other law of the iterated logarithm for martingales
\newblock Bulletin of the Australian Mathematical Society, \textbf{41}(02), 307 -311.

\bibitem{JYYQ18}
{Jie, L. I., Yu, L. I., Yu, W., \& Quanhua, Z. H. A. O. (2018).} 
\newblock Panchromatic Remote Sensing Image Classification Combining Maximum Likelihood Algorithm and Polya Urn Model. 
\newblock Bulletin of Surveying and Mapping, (4), 36.

\bibitem{JFZ85}
W.B. Johnson, G. Schechtman and J. Zinn,
\newblock Best constants in moment inequalities for linear
combination of independent and exchangeable random variables,
\newblock Ann. Probab. 13 No. 1 (1985), 234-243.

\bibitem{JK77}
Johnson, N., Kotz, S., 1977.
\newblock Urn Models and Their Application: An Approach to Modern Discrete Probability Theory.
\newblock Series in Probability and Mathematical Statistics., vol 141(2), p. 265.

\bibitem{Ka02}  
O. Kallenberg, 
\newblock Foundations of Modern Probability, 
\newblock 2nd ed., Springer Series in Statistics, 2002.

\bibitem{Kha09}
Khan, R.A.,~2009.
\newblock $L_p$-version of the Dubins-Savage inequality and some exponential inequalities. 
\newblock J. Theor. Probab. \textbf{22}, 348-364.  

\bibitem{KMS221}
Kious, D., Mailler, C., Schapira, B., 2022.
\newblock The trace-reinforced ants process does not find shortest paths.
\newblock Journal de l’École polytechnique - Mathématiques, 2022. ffhal-03759494f.

\bibitem{KMS222}
Kious, D., Mailler, C., Schapira, B., 2022.
\newblock Finding geodesics on graphs using reinforcement learning.
\newblock Annals of Applied Probability, 32(5), 3889-3929. https://doi.org/10.1214/21-AAP1777

\bibitem{Kle08}
A. Klenke, 
\newblock Probability theory. A comprehensive Course., 
\newblock Springer-Verlag London, ~2008. 

\bibitem{Ko38}
Kolmogorov,~A.N.,~1938. 
\newblock Zur L\"osung einer biologischen Aufgabe. 
\newblock Izvestiya nauchno-issledovatelskogo instituta
matematiki i mechaniki pri Tomskom Gosudarstvennom Universitete 2, 1-6.

\bibitem{LauKoo20}
Lau, P. L., Koo, T. T. R., 2020.
\newblock Online popularity of destinations in Australia: An application of Polya Urn process to search engine data,
\newblock Journal of Hospitality and Tourism Management,42, p.277-285.

\bibitem{lesigneVolny00}
Lesigne,~E., Voln\'y,~D., 2000.
\newblock Large deviations for martingales.
\newblock Stoch.~Proc.~Appl.~96, 143-159.

\bibitem{Lev37}
P.~L{\'e}vy, 
\newblock Theorie de l'addition des variables aleatoires,  
\newblock Gauthier-Villars, Paris, ~1937.  

\bibitem{Luo21}
Luo, S.,
\newblock On Azuma-Type Inequalities for Banach Space-Valued Martingales. 
\newblock J Theor Probab 35, 772-800 (2022).

\bibitem{MS18}
Ma, H., Sun, Y., 2018. 
\newblock Complete convergence and complete moment convergence for randomly weighted sums of martingale difference sequence. 
\newblock J Inequal Appl 2018, 173 (2018). 

\bibitem{MAILLER20}
Mailler, C., 2020.
\newblock The Enduring Appeal of the Probabilist’s Urn.
\newblock London Mathematical Society Newsletter, i. 491, p. 24-31.

\bibitem{Maz09}
Mazliak,~L.,~2009.  
\newblock How Paul L\'evy saw Jean Ville and Martingales. 
\newblock Electronic Journal for History of Probability and Statistics, \textbf{(5)} 1, 

\bibitem{NagVak76}
Nagaev,~S.~V.~, Vakhtel,~V.~I., 2006.
\newblock On the local limit theorem for a critical Galton-Watson process.
\newblock Theory Probab. Appl., 50 (3), pp. 400-419.

\bibitem{NH82}
Najock, D. and Heyde, C. C., 1982.
\newblock On the number of terminal vertices in certain random trees with an application to Stemma construction in philology.
\newblock J. Appl. Prob., vol 19, p. 675-680.

\bibitem{Naor12}
\newblock On the Banach-space-valued Azuma inequality and small-set isoperimetry of Alon-Roichman graphs[J].
\newblock Comb. Probab. Comput., 2012, 21(4): 623-634

\bibitem{OPR22}
Oliveira, R.I., Pereira, A., and Ribeiro, R., 2022. 
\newblock Concentration in the Generalized Chinese Restaurant Process. 
\newblock Sankhya A 84, 628--670 (2022). 

\bibitem{Ord66} 
Ordman,~E.~T., 1966. 
\newblock Convergence almost everywhere is not topological
\newblock The American Mathematical Monthly, February, 1966

\bibitem{Osek12}
Os\k{e}kovski,~A., 2012.
\newblock A Note on the Burkholder-Rosenthal Inequality.
\newblock Bulletin of the Polish Academy of Sciences, Mathematics, 60 (2), p. 177-185.

\bibitem{PKL07}
de la Pe\~na,~V.H.,~Klass,M.J.,Lai,Z.L., 2007.
\newblock A general class of exponential inequalities for martingales and ratios.
\newblock Ann. Probab. \textbf{32}, No. 3, 1902-1933.


\bibitem{Pi94}
Pinelis, I.,~1994. 
\newblock Optimum bounds for the distributions of martingales in Banach spaces. 
\newblock Ann. Probab. \textbf{22}, 1679-1706. 

\bibitem{Pisier75}
G. Pisier,
\newblock  Martingales with values in uniformly convex spaces.
\newblock Israel J. Math. 20 (1975), no. 3-4, 326-350.

\bibitem{Pr90}
P. Protter,  
\newblock Stochastic integration and differential equations. A new approach., 
\newblock Applications of Mathematics New York 21. Springer-Verlag, Berlin, 1990.

\bibitem{Rao69}
{K. Murali Rao}
On decomposition theorems of Meyer.     
\newblock Mathematica Scandinavica
\newblock Vol. 24, No. 1 (1969), pp. 66-78

\bibitem{Rosenthal70}
H. P. Rosenthal, 
\newblock On the subspaces of $L^p (p > 2)$ spanned by the sequences of independent
random variables,
\newblock Israel J. Math. 8 (1970), 273-303.

\bibitem{Schnabel38}
 Z. E., Schnabel, 
\newblock The estimation of the total fish population of a lake, 
\newblock London: Griffin, 1938.

\bibitem{Se06}
E. Seneta, 
\newblock Non-negative matrices and Markov chains,  
\newblock Springer, New York, 2006. 

\bibitem{Sh99}
A.N. Shiryaev, 
\newblock Probability,  
\newblock 2nd edn. Graduate texts in mathematics, vol 95. Springer, New York, ~1996.

\bibitem{SLZWJA17}
{W. Song, M. Li, P. Zhang, Y. Wu, L. Jia and L. An}, 
\newblock {\it Unsupervised PolSAR Image Classification and Segmentation Using Dirichlet Process Mixture Model and Markov Random Fields With Similarity Measure} 
\newblock in: IEEE Journal of Selected Topics in Applied Earth Observations and Remote Sensing, vol. 10, no. 8, pp. 3556-3568, Aug. 2017, 

\bibitem{St07}
G. Stoica 
\newblock {\it Baum–Katz–Nagaev type results for martingales.} 
\newblock J. Math. Anal. Appl. 336 (2007) 1489--1492

\bibitem{St87}
J.M. Stoyanov,  
\newblock Counterexamples in Probability, second edition, 
\newblock John Wiley \& Sons, New York, ~1987.

\bibitem{Talagrand89}
M. Talagrand, 
\newblock Isoperimetry and integrability of the sum of independent Banach-space-valued
random variables
\newblock Ann. Probab. 17 No. 4 (1989), 1546-1570.

 \bibitem{TMJD19} 
A. Terenin, M. Magnusson, L. Jonsson and D. Draper, 
 \newblock {\it P\'olya Urn Latent Dirichlet Allocation: A Doubly Sparse Massively Parallel Sampler}
 \newblock in: IEEE Transactions on Pattern Analysis and Machine Intelligence, vol. 41, no. 7, pp. 1709-1719, 1 July 2019. 

 \bibitem{To79}
R.J. Tomkins, 1975.
\newblock A Law of the Iterated Logarithm for Martingales.
\newblock Z. Wahrscheinlichkeitstheorie verw. Gebiete 33, 65-68. 


\bibitem{vdG95}
van de Geer, S., 1995.
\newblock Exponential inequalities for martingales, with application to maximum likelihood estimation for counting processes. 
\newblock Ann. Stat. \textbf{23}, 1799-1801. 

\bibitem{Vi39}
J. Ville,  
\newblock \'Etude critique de la notion de collectif, 
\newblock Monographies des Probabilit\'es 3, Paris: Gauthier-Villars, ~1939.

\bibitem{Vo91}
Voit, M., 1991. 
\newblock A law of the iterated logarithm for martingales. 
\newblock Bulletin of the Australian Mathematical Society, 43, 181-­185.

\bibitem{Volny89}
Vonl\'y,~D.,~1989.
\newblock On non-ergodic versions of limit theorems, 
\newblock Apl. Mat. 34, 351-363.

\bibitem{Volny93}
Voln\'y,~D.,~1993.
\newblock Approximating martingales and the central limit theorem for strictly stationary processes.
\newblock Stoch.~Proc.~Appl.~44 (1), 41-71.

\bibitem{Wachtel08}
Wachtel,~V.~I., 2008.
\newblock Limit Theorems for Probabilities of Large Deviations of a Critical Galton-Watson Process Having Power Tails.
\newblock Theory Probab.~Appl.~54(4), 674-688.

\bibitem{Wi91}
D. Williams,
\newblock Probability with martingales,  
\newblock Cambridge University Press, ~1991.  

\bibitem{Yu99} 
Yukich,~J.~E.,~1999.  
\newblock Asymptotics for the length of a minimal triangulation on a random sample. 
\newblock Ann. Appl. Probab., 9(1), 27-45.
\end{thebibliography}
\end{document}